\documentclass[10pt]{amsart}

\usepackage{bm}
\usepackage{psfrag}
\usepackage{url}
\usepackage{algpseudocode}
\usepackage{fancyhdr}
\usepackage{xy}
\input xy
\xyoption{all}
\usepackage{stmaryrd}
\usepackage{calrsfs}

\usepackage{amsmath}
\usepackage{amssymb}
\usepackage{amsfonts}
\usepackage{amsthm}
\usepackage{setspace}
\usepackage{color}
\usepackage{mathtools}
\usepackage{mathrsfs}
\usepackage{tikz}
\usetikzlibrary{matrix,arrows,decorations.pathmorphing}
\usepackage{tikz-cd}
\usepackage{hyperref}
\hypersetup{colorlinks=true, linkcolor=blue, citecolor=magenta, urlcolor=wine}
\usepackage{verbatim}
\usepackage{pgfplots}
\usepackage{graphicx}
\usepackage{marvosym}
\usepackage{esvect}
\usepackage{geometry}
\usepackage{stmaryrd}
\usepackage{extarrows}
\usepackage[shortlabels]{enumitem}
\usepackage{tcolorbox}
\usepackage{cleveref}

\usepackage[perpage]{footmisc}

\DeclareMathAlphabet{\mathcal}{OMS}{cmsy}{m}{n}

\newtheorem*{maintheorem*}{Main Theorem}
\newtheorem{theorem}{Theorem}[section]
\newtheorem*{theorem*}{Main Theorem}
\newtheorem{prob}[theorem]{Problem}

\newtheorem{prop}[theorem]{Proposition}
\newtheorem{conj}[theorem]{Conjecture}
\newtheorem{lemma}[theorem]{Lemma}
\newtheorem{cor}[theorem]{Corollary}

\theoremstyle{definition}
\newtheorem{definition}[theorem]{Definition}
\newtheorem{remark}[theorem]{Remark}
\newtheorem{example}[theorem]{Example}
\numberwithin{equation}{section}

\newcommand{\inv}{^{-1}}

\newcommand{\ZZ}{\mathbb{Z}}

\newcommand{\NN}{\mathbb{N}}

\renewcommand{\O}{\mathcal{O}} 
\newcommand{\pp}{\mathfrak{p}}

\newcommand{\mm}{\mathfrak{m}}
\newcommand{\vp}{\varphi}
\newcommand{\id}{\mathrm{id}}

\newcommand{\oo}{\infty}

\newcommand{\im}{\operatorname{im}}
\newcommand{\iso}{\xrightarrow{\sim}}
\newcommand{\rt}{\sqrt}

\newcommand{\Spec}{\operatorname{Spec}}

\newcommand{\Skip}{\vskip 5pt \noindent}

\newcommand{\al}{\alpha}
\newcommand{\sub}{\subset}

\newcommand{\br}[1]{{\left\{ #1 \right\}}}
\newcommand{\pr}[1]{{\left( #1 \right)}}

\newcommand{\up}[1]{^{\pr{#1}}}

\newcommand{\F}{\mathcal{F}}
\newcommand{\Ff}{\mathscr{F}}

\makeatletter
\renewcommand*\env@matrix[1][*\c@MaxMatrixCols c]{%
  \hskip -\arraycolsep
  \let\@ifnextchar\new@ifnextchar
  \array{#1}}
\makeatother

\newenvironment{enum}
    {
    \begin{enumerate} [label = (\alph*)]
    }
    {
    \end{enumerate}
    }

\newcommand{\BB}{\mathbb{B}}
\newcommand{\LL}{\mathbb{L}}

\newcommand{\bb}{\mathfrak{b}}
\newcommand{\G}{\mathcal{G}}
\renewcommand{\H}{\mathcal{H}}

\newcommand{\N}{\mathcal{N}}

\newcommand{\C}{\mathcal{C}}
\newcommand{\mF}{\operatorname{mF}}

\newcommand{\B}{\mathcal{B}}

\newcommand{\cl}{\operatorname{cl}}
\newcommand{\gnf}{\operatorname{gnf}}
\newcommand{\lcf}{\operatorname{lcf}}
\newcommand{\Id}{\operatorname{Id}}
\newcommand{\Sub}{\operatorname{Sub}}

\renewcommand*\=[1]{\relax\ifmmode\overline{#1}\else\@equal{#1}\fi}
\renewcommand*\^[1]{\relax\ifmmode\widehat{#1}\else\@hat{#1}\fi}
\renewcommand*\~[1]{\relax\ifmmode\widetilde{#1}\else\@tilde{#1}\fi}

\makeatletter
\newcommand{\bigplus}{%
  \DOTSB\mathop{\mathpalette\mattos@bigplus\relax}\slimits@
}
\newcommand\mattos@bigplus[2]{%
  \vcenter{\hbox{%
    \sbox\z@{$#1\sum$}%
    \resizebox{!}{0.9\dimexpr\ht\z@+\dp\z@}{\raisebox{\depth}{$\m@th#1+$}}%
  }}%
  \vphantom{\sum}%
}
\makeatother

\keywords{quantale, multiplicative lattice, filter, multiplicative filter, locale, localization, Baire Category Theorem}

\subjclass[2020]{06F07, 06D22, 13B30}

\begin{document}
	
\mbox{}
\title{Commutative Quantale and Localization}

\author{Bangzheng Li}
\address{MIT \\ Cambridge, Massachusetts 02139}
\email{liben@mit.edu}

\author{Yichen Xiao}
\address{MIT \\ Cambridge, Massachusetts 02139}
\email{cedricx@mit.edu}

\date{\today}
	
\begin{abstract}
    In this paper we introduce the localization construction for quantales.
    A quantale is a complete semilattice combined with a multiplication.
    We mimic the notion of filter in a lattice to define multiplicative filters in a quantale, and construct the localization of the quantale at a multiplicative filter.
    We prove theorems with similar structure as ``$\Spec R$ is a sheaf" and use them to obtain several results in algebra and geometry, including the Baire Category Theorem and some of its generalizations.
    We also present an algebraic version of Baire Category Theorem, which we believe has not appeared in literature.
\end{abstract}

\bigskip
\maketitle

\section{Introduction}
\label{sec: intro}

\subsection{Quantale}

Follows from the rising interest in discrete math and logic, lattice theory, as well as its generalization quantale theory, has appeared in many studies such as \cite{K90, G21, PJ23, GCM25}.
Roughly speaking, a quantale is a complete (join-)semilattice that carries a multiplication operation.
In the regular definition of quantale, we impose the existence of bottom element.
However, in this paper we will drop this assumption, so that our theory becomes more general.
There are two main sources for quantales: the collection of opens in a topological space and the collection of ideals in a ring.
The former one is more geometric and the latter one is more algebraic.

As an important feature of algebraic systems, ideal theory has been developed inside lattices (\cite{R72, G11, KT12, KL17, N18, CL22}), alongside the closed related theory: filter theory. 
Aside from ideals and filters are well studied in lattice theory, there's also few mentions about these concepts in multiplicative lattices, such as quantales. Similar with the various ideals mentioned in \cite{PJ23, G23I, G23II}, there also exists several type of filters, such as prime filters, ultrafilters, etc. 

Developed through ideal theory, localization is considered as one of the most important part in ring theory, but there's a lack of localization inside quantale theory. In \cite{G25}, there exists some achievement about localization over prime elements, but there's lacking a more general localization theory inside quantales. 

The primary objective for this paper is to develop a localization theory inside quantales, using a multiplicative filter point of view.

\smallskip

\subsection{Three Theorems in Algebra/Geometry}

We introduce three theorems that may be well-known to the audiences.

\textbf{One sheaf axiom for $\Spec R$:} let $R$ be a commutative unital ring and $f_1, \ldots, f_n \in R$ such that $(f_1, \ldots, f_n) = R$.
Then the $R$-linear map $R \to \prod_{i = 1}^n R_{f_i}$ is injective.

\textbf{0 is a local property:} let $R$ be a commutative unital ring and $M$ an $R$-module.
If $M_\mm = \br{0}$ for all maximal ideals $\mm \sub R$, then $M = \br{0}$.

\textbf{Baire Category Theorem:} a locally compact regular (which is equivalent to locally compact Hausdorff) topological space is not the union of countably many nowhere dense subsets.

Among the three theorems, the first two come from algebra, while the last one comes from geometry.
You may think that they are three independent theorems, at least the first two and the last one has no relation with each other.
However, using the language of quantale, all of these three theorems can be unified into

\textbf{Filter Merging Theorem:} under some assumptions, for a $Q$-module $M$ and multiplicative filters $\br{\F_i}_{i \in I}$, the map
\[
M_\F \to \prod_{i \in I} M_{\F_i}
\]
is injective, where $\F = \bigcap_{i \in I} \F_i$ is the ``merged multiplicative filter."

Using this unified theorem, the proof idea of Baire Category Theorem can be summarized by the following procedure:
\begin{enum}
    \item We regard everything up to a nowhere dense subset.
    \item Then each nowhere dense subset becomes empty set.
    \item However, the whole space is not the empty set since the whole space is not nowhere dense.
    \item We conclude by the observation that a nonempty set is not the union of countably many empty sets.
\end{enum}
At first glance, this ``proof" makes no sense.
Here are some potential concerns:
\begin{enum}
    \item Why can we regard everything up to a nowhere dense subset?
    \item A nonempty set is not the union of arbitrarily many empty sets, so why Baire Category Theorem does not work for arbitrarily many nowhere dense subsets?
    \item Where did we use the locally compact Hausdorff condition?
\end{enum}
All these questions will have a satisfactory answer in this paper, so that the above ``proof" becomes a rigorous one.

Also, recall the other condition for the sheaf axiom of $\Spec R$: for $f_1, f_2, \ldots, f_n \in R$ such that $(f_1, \ldots, f_n) = R$, the image of $R \to \prod_{i = 1}^n R_{f_i}$ is precisely those $(a_1, \ldots, a_n) \in \prod_{i = 1}^n R_{f_i}$ such that the image of $a_i$ under map $R_{f_i} \to R_{f_i f_j}$ agrees with the image of $a_j$ under map $R_{f_j} \to R_{f_i f_j}$.
We will exhibit a similar statement for the case of quantales and their localization.
Furthermore, we will remove the somewhat superfluous condition ``$(f_1, \ldots, f_n) = R$."

\smallskip

\subsection{Overview}

We now present an overview toward this paper.
In \Cref{sec: prelim} we talk about basic definitions and propositions that is necessary to understand the main text of this paper.
The key definitions are quantales and their modules, Noetherian, compactness.

In \Cref{sec: mfilter and localization} we define the notion of multiplicative filters in a quantale and construct the localization of a quantale at a multiplicative filter.
It is notable that after doing localization the result may not always be a quantale/module, so we define what it means for a multiplicative filter to be localizable, under which assumption the resulting localization will indeed be a quantale/module.

In \Cref{sec: shrinkable and susp} we introduce shrinkability, a notion that serves important role in the study of localization of quantales.
A byproduct is the suspension of a quantale, which is used to define the shrinkability and has interest in its own.
We will also prove the first filter merging theorem.

In \Cref{sec: useful structures} we first recall the precoherence for a quantale.
With the help of suspension, we then define bloomingness for a quantale/module, and use such ideas to present the second filter merging theorem.
We also show that a very important class of topological spaces, locally compact Hausdorff spaces, induce blooming quantales.
Then we go ahead and define special classes of multiplicative filters, and use them to give the third filter merging theorem.

However, a lot of theorems about localization is based on the assumption that our multiplicative filters are localizable (or some stronger conditions), and using the original definition to show a multiplicative filter is localizable is not an easy task.
In \Cref{sec: normal and conormal} we present what it means for a multiplcative filter to be normal, a condition that is easy to check and can guarantee that our multiplicative filter is localizable.
Similarly we define conormal, the dual notion of normal, and show that it shares some similar properties as the normalness.
As an application of these notions, we show that some special multiplicative filters are normal/conormal, and they will serve an important role in the proof of Baire Category Theorem.

\Cref{sec: applications} is divided into three parts.
In the first part we compare our approach of defining localization of quantale with other author's.
In the second part we present several applications of our theory about localization of quantales.
We show that various algebraic/geometric facts can be deduced from localization of quantales, including the Baire Category Theorem.
We also give two generalizations of Baire Category Theorem, which can be proved in a few lines using the language of quantales.
In the end, we give an algebraic version of Baire Category Theorem, which can be proved exactly the same way as the Baire Category Theorem.
We think this is the strength of the theory about localization of quantales: it connect algebra and geometry together, so that a definition in algebra can be used to give intuition for geometric world, a theorem in algebra can be transferred to a theorem in geometry, and vice versa.
In the third part we give some conjectures and questions for future studies.

\bigskip
\section{Preliminary}
\label{sec: prelim}

In this section, we will go through basic notions of properties about quantales and modules. (some reference)

\smallskip

\subsection{Notations and Assumptions}

In this paper, when considering a ring, we assume it is a unital commutative ring. For a topological space $X$ and $Y \sub X$, we use $\cl Y, \operatorname{int}(Y)$ and $Y^C$ to represent the closure, interior and complement of $Y$ in $X$. We use $\NN$ to represent positive integers, and $\NN_0$ for nonnegative integers. We would write $\bigvee S, \bigwedge S, \sum S, \prod S$ in short of $\bigvee_{s \in S} s, \bigwedge_{s \in S} s, \sum_{s \in S} s, \prod_{s \in S} s$ when the operation is well-defined. 

There are two definitions for a (join-)semilattice: one requires the existence of bottom element, and the other does not.
Here we take the latter one as our definition.

\begin{definition}
    A \textit{semilattice} is a poset that admits the join of any two elements.
\end{definition}

\begin{definition}
    A \textit{complete semilattice} is a semilattice that admits the join of arbitrary nonempty subset.
\end{definition}

\begin{definition}
    A (commutative integral) \emph{quantale} is a structure $(Q, \vee, \cdot, 1)$, such that $(Q, \vee, 1)$ admits a complete semilattice structure (where $1$ is the top element), and the multiplication $\cdot: Q \times Q \rightarrow Q$ is commutative, associative, distributive ($a \bigvee_{i \in I} b_i = \bigvee_{i \in I} (a b_i)$ for nonempty $I$), with $1$ serves as multiplicative unit. 
\end{definition}

\begin{remark}
    In this paper, we take nonzero join instead of arbitrary (so a bottom element may not exist), which drops a necessary condition for the quantale structure (and generalizes the theory of quantale).
    Still, the results work for a quantale with original definition. 
\end{remark}

For convenience, we write $a + b$ for $a \vee b$ and $\sum S$ for $\bigvee S$.
As a generalization of quantale, we also introduce prequantale: 

\begin{definition}
    A (commutative integral) \emph{prequantale} is a structure $(Q, \vee, \cdot, 1)$, such that $(Q, \vee, 1)$ admits a semilattice structure with the top element $1$, and the multiplication $\cdot: Q \times Q \rightarrow Q$ is commutative, associative, distributive ($a(b \vee c) = (a b) \vee (ac)$), with $1$ serves as multiplicative unit. 
\end{definition}

First let us prove some basic properties of quantale.

\begin{prop} \label{prop: basic properties of quantale}
    Let $(Q, \vee)$ be a semilattice.
    Then
    \begin{enum}
        \item For $a, b \in Q$, we have $a \le b$ if and only if $a + b = b$.
        \item If $Q$ is a prequantale, then for $a, b, a', b' \in Q$, if $a \le a', b \le b'$, then $a b \le a' b'$.
        In particular, $a b \le a$ and $a b \le b$.
        \item If $Q$ is a quantale, then for $a, b \in Q$, their meet $a \wedge b = \sum \br{x \in Q : x \le a, x \le b}$ exists.
        The operations $\vee$ and $\wedge$ on $Q$ make $Q$ into a lattice.
    \end{enum}
\end{prop}

\begin{proof}
    Part (a): $\Rightarrow$: for $q \in Q$, we have $q \ge a$ and $q \ge b$ if and only if $q \ge b$, so $a + b = b$.

    \noindent $\Leftarrow$: we have $b \ge b$, so from the definition of addition we see $b \ge a$.

    \Skip Part (b): it suffices to show $a b \le a b'$.
    To see this, just note that $a b + a b' = a (b + b') = a b'$, so by Part (a) we see $a b \le a b'$.
    Now $a b \le a \cdot 1 = a$ and $a b \le 1 \cdot b = b$.

    \Skip Part (c): by Part (b) we have $a b \in S \triangleq \br{x \in Q : x \le a, x \le b}$, so $S$ is not empty.
    Now pick any $q \in Q$ such that $q \le a$ and $q \le b$, then $q \in S$, so $q \le \sum S$, from which we see $a \wedge b = \sum S$ exists, hence $(Q, \wedge, \vee)$ is a lattice.
\end{proof}

\begin{remark}
    Thus, given the addition structure on a semilattice, we can recover the partial order by \Cref{prop: basic properties of quantale} Part (a).
\end{remark}

\begin{prop} \label{prop: idempotent quantale is a locale}
    Let $Q$ be a quantale.
    The following are equivalent:
    \begin{enum}
        \item For all $q \in Q$, we have $q^2 = q$.
        \item For all $a, b \in Q$, we have $a b = a \wedge b$.
    \end{enum}
\end{prop}

\begin{proof}
    $\Rightarrow$: we have $a b \le a \wedge b$ by \Cref{prop: basic properties of quantale}.
    On the other hand, we have $a \wedge b = (a \wedge b)(a \wedge b) \le a b$ by \Cref{prop: basic properties of quantale}, so $a b = a \wedge b$.

    \Skip $\Leftarrow$: we have $q^2 = q \cdot q = q \wedge q = q$.
\end{proof}

\begin{definition}
    If a quantale satisfies any of the equivalent characterizations in \Cref{prop: idempotent quantale is a locale}, then we say the quantale is \textit{idempotent}.
\end{definition}

\begin{remark}
    In usual language, an idempotent quantale (with bottom) is called a locale.
    However, here we do not require the existence of bottom, so we will not call such structure locale.
\end{remark}

Let us see some examples (and at the same time definitions) of quantales.

\begin{definition}
    We use $\BB$ to denote the quantale $\br{0, 1}$, where $0 + 1 = 1 + 1 = 1, 0 + 0 = 0, 1 \cdot 1 = 1$, and $1 \cdot 0 = 0 \cdot 0 = 0$.
    Here the unit is $T$.
    In general, for $n \in \NN_0$ we use $\BB_n$ to denote the quantale $\br{-n, -(n - 1), \ldots, -1, 0}$ under partial order $\le$ and multiplication $a \cdot b = \min(a, b)$.
    We use $\BB_\omega$ to denote the quantale $\ZZ_{\le 0}$ under $\le$ and $a \cdot b = \min(a, b)$, and use $\BB_\oo$ to denote the quantale $\ZZ_{\le 0} \cup \br{-\oo}$ under $\le$ and $a \cdot b = \min(a, b)$ (with the agreement that $-\oo$ is the smallest element).
    The addition structure is given by $\sum S = \max S$.
    The unit is 0.
    Note that $\BB \cong \BB_1$ and all $\BB_n, \BB_{\omega}, \BB_\oo$ are idempotent quantales.
\end{definition}

\begin{definition}
    We use $\LL_\omega$ to denote the quantale $\ZZ_{\le 0}$ under partial order $\le$ and multiplication $a \cdot b = a + b$.
    For $n \in \NN_0$, we use $\LL_n$ to denote the quantale $\br{-n, -(n - 1), \ldots, -1, 0}$ under partial order $\le$ and multiplication $a \cdot b = \max(-n, a + b)$.
    We use $\LL$ to denote the quantale $\ZZ_{\le 0} \cup \br{-\oo}$ under partial order $\le$ and multiplication $a \cdot b = a + b$ (with the agreement that $-\oo + a = -\oo$).
    The addition structure is given by $\sum S = \max S$.
    The unit is 0.
    Note that $\LL_1 \cong \BB_1 \cong \BB$.
    Another description of $\LL$ is $\br{0, 1, \epsilon, \epsilon^2, \epsilon^3, \ldots}$, where $1 > \epsilon > \epsilon^2 > \cdots > 0$.
\end{definition}

\begin{definition}
    Let $X$ be a topological space.
    Then we use $\O(X)$ to denote the idempotent quantale of all open sets on $X$, with $\vee, \wedge$ defined as union and intersection operations. 
\end{definition}

\begin{definition}
    \cite{R62, D76}
    Let $R$ be a ring.
    Then we use $\Id(R)$ to denote the quantale of all ideals in $R$, with $+, \wedge, \cdot$ defined as ideal sum, intersection and multiplication. 
\end{definition}

Just like modules over a ring, we can define modules over a quantale.

\begin{definition}
    Let $Q$ be a quantale (resp. prequantale).
    Then a \emph{$Q$-module} (resp., $Q$-premodule) is a complete semilattice (resp., semilattice) $M$ (where join written additively) equipped with a associative multiplication map $\cdot : Q \times M \to M$ such that the multiplication distributes on both $Q$ and $M$ over nonempty (resp., finite nonempty) summation, and $1 \in Q$ serves as identity. 
\end{definition}

We're using $R$ for rings and $Q, P$ for quantales, to avoid confusion for ring module and quantale module. 

A morphism between quantales (resp., prequantales, modules, premodules) is a map that preserves all structures, and preserves the top element for the case of quantale/prequantale.

\begin{example}
    Let $f : X \to Y$ be a continuous map between topological spaces, then $f$ induces a quantale homomorphism $\O(f) : \O(Y) \to \O(X), U \mapsto f\inv(U)$.
\end{example}

\begin{example}
    Let $f : R \to S$ be a ring homomorphism, then $f$ induces a quantale homomorphism $\Id(f) : \Id(R) \to \Id(S), J \mapsto f(J) S$.
\end{example}

\smallskip

\subsection{Noetherian and Compactness}

In commutative algebra, Noetherian is one of the most important notion, saying that an ascending chain of ideals in $R$ eventually stabilizes.
Since ideals in $R$ corresponds exactly to elements in $\Id(R)$, we have a natural definition for Noetherian in a poset.

\begin{definition}
    Let $P$ be a poset.
    Then we say $P$ is \emph{Noetherian} if every ascending chain of element $p_1 \le p_2 \le \cdots$ in $P$ eventually stabilizes, i.e., there exists $N \in \NN$ such that $p_N = p_{N + 1} = p_{N + 2} = \cdots$
\end{definition}

The next lemma shows that under Noetherian assumption, a semilattice is equivalent to a complete semilattice.

\begin{lemma} \label{lem: Noetherian semilattice is complete}
    A Noetherian semilattice is complete.
\end{lemma}

\begin{proof}
    Let $L$ be a Noetherian semilattice and pick $S \sub L$.
    We claim that $\sum S$ exists and is equal to $\sum S_0$ for some finite subset $S_0 \sub S$.
    Assume the contrary, then we produce an ascending chain that does not stabilize in the following procedure.

    Pick any $s_1 \in S$.
    If $s_1 \ge s$ for all $s \in S$, then $\sum S = s_1$ and we arrive at a contradiction.
    Thus, there exists $s_2 \in S$ such that $s_2 \not\le s_1$.

    If $s_1 + s_2 \ge s$ for all $s \in S$, then $\sum S = s_1 + s_2$ and we arrive at a contradiction.
    Thus, there exists $s_3 \in S$ such that $s_3 \not\le s_1 + s_2$.

    Repeat this procedure, we get an ascending chain $s_1 < s_1 + s_2 < s_1 + s_2 + s_3 < \cdots$ that does not stabilize, contradicts the assumption that $L$ is Noetherian.
\end{proof}

Thanks to \Cref{lem: Noetherian semilattice is complete}, when we talk about Noetherian in a semilattice we can assume the lattice is complete.
Now let us give some equivalent characterizations of Noetherian.
To simplify our notations, let us introduce a convention first.

\begin{definition}
    Let $L$ be a complete semilattice and $x \in L, \br{x_i}_{i \in I} \sub L$.
    Then we write $x \le^* \sum_{i \in I} x_i$ if $x \le \sum_{i \in I_0} x_i$ for some finite subset $I_0 \sub I$.
    Also, we write $x =^* \sum_{i \in I} x_i$ if $x \ge x_i$ for all $i \in I$ and $x = \sum_{i \in I_0} x_i$ for some finite subset $I_0 \sub I$.
\end{definition}

\begin{prop} \label{prop: equivalent chan of Noetherian}
    Let $L$ be a complete semilattice.
    The following are equivalent:
    \begin{enum}
        \item $L$ is Noetherian.
        \item Every subset $S \sub L$ has a maximal element, i.e., there exists $s \in S$ such that for $t \in S$, the condition $s \le t$ implies $s = t$.
        \item For $x \in L, \br{x_i}_{i \in I} \sub L$, if $x \le \sum_{i \in I} x_i$, then $x \le^* \sum_{i \in I} x_i$.
        \item For $x \in L, \br{x_i}_{i \in I} \sub L$, if $x = \sum_{i \in I} x_i$, then $x =^* \sum_{i \in I} x_i$.
    \end{enum}
\end{prop}

\begin{proof}
    (a) $\Rightarrow$ (b): let $S \sub L$ be a subset.
    Assume the contrary, $S$ does not admit a maximal element, i.e., for all $s \in S$, there exists $t \in S$ such that $s < t$.

    Pick any $s_1 \in S$, then there exists $s_2 \in S$ such that $s_1 < s_2$, and there exists $s_3 \in S$ such that $s_2 < s_3$.
    Repeat this procedure, we get an ascending chain $s_1 < s_2 < s_3 < \cdots$ that does not stabilze, contradicts the assumption that $L$ is Noetherian.

    \Skip (b) $\Rightarrow$ (c): let $S = \br{\sum_{i \in I_0} x_i : I_0 \sub I \text{ is finite}}$.
    Then $S$ has a maximal element $\sum_{j \in J} x_j$, where $J \sub I$ is finite.
    Now for all $i \in I$, if $x_i \not\le \sum_{j \in J} x_j$, then $\sum_{j \in J \cup \br{i}} x_j > \sum_{j \in J} x_j$, contradicts with the maximality of $\sum_{j \in J} x_j$.

    Thus, $x_i \le \sum_{j \in J} x_j$ for all $i \in I$, so $\sum_{i \in I} x_i \le \sum_{j \in J} x_j$.
    Now $x \le \sum_{i \in I} x_i \le \sum_{j \in J} x_j$ implies $x \le^* \sum_{i \in I} x_i$.

    \Skip (c) $\Rightarrow$ (d): we have $x \le \sum_{i \in I_0} x_i$ for some finite $I_0 \sub I$.
    Now $x = \sum_{i \in I} x_i$ implies $x \ge x_i$ for all $i \in I$, hence $x \ge \sum_{i \in I_0} x_i$.
    Thus, $x = \sum_{i \in I_0} x_i$ and so $x =^* \sum_{i \in I} x_i$, as desired.

    \Skip (d) $\Rightarrow$ (a): pick an ascending chain $x_1 \le x_2 \le \cdots$
    Let $x = \sum_{i \in \NN} x_i$, then $x = \sum_{i \in I_0} x_i$ for some finite $I_0 \sub \NN$.
    Now let $N = \max I_0$, then $x = x_N$ since $x_i \le x_{i + 1}$ for all $i \in \NN$.
    Thus, $x_N \ge x_i$ for all $i \in \NN$, hence the chain $x_1 \le x_2 \le \cdots$ stabilizes.
\end{proof}

It turns out that the notion of Noetherian in complete semilattice agrees with the that of ring and topological spaces.

\begin{example}
    Let $R$ be a commutative unital ring.
    Then $R$ is Noetherian if and only if $\Id(R)$ is Noetherian.
\end{example}

\begin{example}
    Let $X$ be a topological space.
    Then $X$ is Noetherian (as a topological space) if and only if $\O(X)$ is Noetherian.
\end{example}

However, almost all topological spaces that arise from geometry is not Noetherian.
When topologists want to run some arguments with finiteness condition, they often require the space to be compact.
Recall a topological space is compact if every open cover admits a finite subcover.
Translate this definition into the language of complete semilattice we get the following definition.

\begin{definition}
    Let $L$ be a complete semilattice (with top element $1$).
    Then we say $L$ is \emph{compact} if for all $\br{x_i}_{i \in I} \sub L$, if $1 = \sum_{i \in I} x_i$, then $1 =^* \sum_{i \in I} x_i$.
\end{definition}

Note that a Noetherian complete semilattice is compact by \Cref{prop: equivalent chan of Noetherian}.

\begin{example}
    Let $X$ be a topological space.
    Then $X$ is compact if and only if $\O(X)$ is compact.
\end{example}

\bigskip

\section{Multiplicative Filters and Localization}
\label{sec: mfilter and localization}

Localization is one of the most fundamental construction in commutative algebra.
Thus, we want to carry it to the realm of quantales and quantale modules.

Throughout this section, we assume $Q$ is a quantale.

\smallskip

\subsection{Definitions of Multiplicative Filters}

Similar to the multiplicative subsets used in localization in rings, we perform the definition of a multiplicative filter: 

\begin{definition}
    A \emph{multiplicative filter}, or in short \emph{m-filter}, of a quantale $Q$ is a subset $\F \sub Q$ satisfying: 
    \begin{enumerate}[label=(\alph*)]
        \item (Nonempty) $1 \in \F$.
        \item (Upper closeness) If $a \in \F, b \in Q_{\ge a}$ then $b \in \F$.
        \item (Multiplicative closeness) If $a, b \in \F$, then $a b \in \F$.
    \end{enumerate}
\end{definition}

\begin{definition}
    For a set $S \subset Q$, the \emph{multiplicative filter generated by} $S$ is the smallest multiplicative filter (with respect to inclusion) $\Ff(S)$ containing all elements in $S$.
\end{definition}

\begin{lemma} \label{lem: describe FS}
    For $S \sub Q$, we have
    \[
    \Ff(S) = \br{q \in Q : q \ge \prod_{i = 1}^n s_i \text{ for some } n \in \NN, s_1, \ldots, s_n \in S}.
    \]
\end{lemma}

\begin{proof}
    The set defined above is clearly upper closed. Suppose $p, q$ are both in the set, with $q \ge \prod_{i = 1}^n s_i$, $p \ge \prod_{j = 1}^m t_j$, $s_i, t_j \in S$, then $pq \ge \prod_{i=1}^ns_i\prod_{j=1}^nt_j$, with right hand side being finite product of element in $S$, thus we have the set being closed under multiplication. Moreover, for $S \neq \br{0}$, we have $1$ in the set. 

    On the other hand, as $S \sub \Ff(S)$, we know all finite product of element in $S$ must lie in $\Ff(S)$, and as $\Ff(S)$ is upperclosed we know the set described above lies in $\Ff(S)$, and as it is indeed a multiplicative filter, it is $\Ff(S)$. 
\end{proof}

There are two natural operations on the multiplicative filters: taking the set-theoretic product $\Ff(\br{ab: a \in \F_1, b \in \F_2})$, and the set-theoretic sum $\br{a + b: a \in \F_1, b \in \F_2}$. The following lemma shows that the set-theoretic sum is the same as intersection. 

\begin{lemma} \label{lem: set sum of filters is intersection}
    Let $\F, \G$ be m-filters of $Q$.
    Then $\F \cap \G = \br{a + b : a \in \F, b \in \G}$.
\end{lemma}

\begin{proof}
    $\sub$: pick $x \in \F \cap \G$, then $x = x + x$ and $x \in \F, x \in \G$, so LHS is in RHS.

    \Skip $\supset$: pick $a \in \F, b \in \G$, then $a + b \in \F, a + b \in \G$, hence $a + b \in \F \cap \G$.
    Thus, RHS is in LHS.
\end{proof}

The notion of m-filter also gives us a way to construct new quantales from the existing ones.

\begin{prop} \label{prop: mF is an idempotent quantale}
    Let $\mF(Q)$ be the set of multiplicative filter on a quantale $Q$, then $\mF(Q)$ admits a natural quantale structure under $\sub$ and multiplication $\F \cdot \G = \F \cap \G$.
    The addition is given by $\sum_{i \in I} \F_i$ being the m-filter generated by $\bigcup_{i \in I} \F_i$.
    Moreover, $\mF(Q)$ is an idempotent quantale with least element $\br{1}$.
\end{prop}

\begin{proof}
    Commutativity and complete semilattice structure follows from the definition, and the element $Q \in \mF(Q)$ acts as the multiplicative unit. Thus, we only need to consider distributivity. 

    Suppose $q \in \G\big(\sum_{i \in I} \F_i\big)$, then there exists $I_0 \sub I$ finite, elements $b \in \G$ and $a_i \in \F_i$ for $i \in I_0$, satisfying $q \ge b + \prod_{i \in I_0} a_i$. Thus, $\prod_{i \in I_0} (b + a_i) = bc + \prod_{i \in I_0} a_i \le b + \prod_{i \in I_0} \le 1$ for some $c$, and $q \in \sum_{i \in I} \G\F_i$. 

    Conversely, suppose $q \in \sum_{i \in I} \G\F_i$, then there exists $I_0 \sub I$ finite, elements $b_i \in \G$ and $a_i \in \F_i$ for $i \in I_0$, satisfying $q \ge \prod (b_i + a_i)$. Since $\prod_{i \in I_0} b_i + \prod_{i \in I_0} a_i \le \prod_{i \in I_0} (b_i + a_i) \le q$ and $\prod_{i \in I_0} b_i \in \G$, we have $q \in \G\big(\sum_{i \in I} \F_i\big)$. 

    Thus, $\mF(Q)$ satisfies all axioms for quantale, thus it is a quantale.
    Finally, for $\F \in \mF(Q)$, we have $\F \cdot \F = \F \cap \F = \F$, hence $\mF(Q)$ is an idempotent quantale by \Cref{prop: idempotent quantale is a locale}.
    Obviously $\br{1} \in \mF(Q)$ is less than all other elements of $\mF(Q)$.
\end{proof}

We would now give some classes of m-filters.

\begin{example}
    The \emph{trivial filter} is $\br{1} \sub Q$.
\end{example}

\begin{example}
    The whole quantale $Q$ is a m-filter of $Q$.
\end{example}

When we specify at rings, here's also multiplicative filter generated from multiplicative sets: 

\begin{example} \label{example: real localization}
    For a ring $R$ and a multiplicative closed subset $S \sub R$, the set $\N = \br{J \in \Id(R): J \cap S \neq \varnothing}$ forms a m-filter in $\Id(R)$.
    Note that $\N = \Ff((s) : s \in S)$.
\end{example}

\smallskip

\subsection{Examples of Filters}

\begin{definition}
    Let $f \in Q$.
    Then we use $\F_f$ to denote the \emph{minimal filter} containing $f$.
    More concretely, $\F_f = \br{q \in Q : f^n \le q \text{ for some } n \in \NN}$.
\end{definition}

\begin{example}
    Let $Q$ be an idempotent quantale, then $\F_f = Q_{\ge f}$ (as sets).
\end{example}

\begin{prop} \label{prop: Ff cap Fg and Ff Fg}
    Let $f, g \in Q$, we have $\F_f\F_g = \F_{f + g}$ and $\F_f + \F_g = \F_{fg}$.
\end{prop}

\begin{proof}
    Suppose $h \in \F_f\F_g$, then $h \ge f^n, h \ge g^m$ for some $m, n \ge 0$, thus $h \ge (f+g)^{n+m} = f^nP_1(f, g)+g^mP_2(f, g)$ for some polynomials $P_1, P_2$. Suppose $h \in \F_{fg}$, then $h \ge (f+g)^n$ gives $h \ge f^n$ and $h \ge g^n$. 

    Suppose $h \in \F_f + \F_g$, then $h = h_fh_g$ such that $h_f \ge f^n, h_g \ge g^m$, thus $h \ge f^ng^m \ge (fg)^{n+m}$. Suppose $h \in \F_{fg}$, we have $h \ge (fg)^n$, thus $h+f^n \ge f^n, h+g^n \ge g^n$, thus $(h+f^n)(h+g^n) = h(h+f^n+g^n)+f^ng^n \in \F_f + \F_g$, thus $h \in \F_f + \F_g$ as it is a m-filter. 
\end{proof}

\begin{example}
    Let $R$ be a commutative unital ring and $Q = \Id(R)$.
    Then for $s \in R$, we have $\F_{(s)} = \br{J \in Q : s \in \rt{J}}$.
\end{example}

\smallskip

Recall two ideals $I, J$ in a ring $R$ are comaximal if $I + J = R$.
It turns out that the notion of comaximal actually gives a m-filter in a quantale.
To see this, we first need a lemma.

\begin{lemma} \label{lem: comaximal mult}
    Let $a, b, c \in Q$.
    If $a + b = 1, a + c = 1$, then $a + b c = 1$.
\end{lemma}

\begin{proof}
    We have $1 = (a + b)(a + c) = a(a + b + c) + b c \le a + b c$.
\end{proof}

This simple lemma has some surprising applications.

\begin{cor} \label{cor: sum ai = 1 implies sum ai up n = 1}
    Let $a_1, \ldots, a_n \in Q$.
    If $\sum_{i = 1}^n a_i = 1$, then $\sum_{i = 1}^n a_i^m = 1$ for all $m \in \NN$.
\end{cor}

\begin{proof}
    First $a_1 + (a_2 + a_3 + \cdots + a_n) = 1$, so $a_1^m + (a_2 + a_3 + \cdots + a_n) = 1$ by \Cref{lem: comaximal mult}.
    Now $a_2 + (a_1^m + a_3 + a_4 + \cdots + a_n) = 1$, so $a_2^m + (a_1^m + a_3 + a_4 + \cdots + a_n) = 1$ by \Cref{lem: comaximal mult}.
    Repeat this procedure, we will finally get $a_1^m + \cdots + a_n^m = 1$.
\end{proof}

\begin{cor} \label{cor: sum ideal is whole implies sum powers is whole}
    Let $R$ be a ring and $f_1, \ldots, f_n \in R$ such that $(f_1, \ldots, f_n) = R$.
    Then $(f_1^m, \ldots, f_n^m) = R$ for all $m \in \NN$.
\end{cor}

\begin{proof}
    Apply \Cref{cor: sum ai = 1 implies sum ai up n = 1} to $Q = \Id(R)$ and $a_i = (f_i)$.
\end{proof}

\begin{definition}
    Let $a \in Q$.
    Then we define the \emph{comaximal filter} $\F_{\perp a} = \br{q \in Q : q + a = 1}$.
    We have $\F_{\perp a}$ is a filter by Lemma~\ref{lem: comaximal mult}. 
\end{definition}

\begin{example}
    Let $X$ be a topological space $Q = \O(X)$.
    Then for $U \in Q$, we have $\F_{\perp U} = \br{V \in Q : U \cup V = X} = \br{V \in Q : U^c \sub V}$.
\end{example}

\smallskip
We now exhibit a class of m-filters that will play a crucial role in the proof of Baire Category Theorem.

\begin{definition}
    Let $a \in Q$.
    Then we define the \emph{codense filter}
    \[
    \F_{\nmid a} = \br{q \in Q : q x \le a \text{ implies } x \le a \text{ for all } x \in Q}.
    \]
\end{definition}

\begin{lemma} \label{lem: codense filter is filter}
    The codense filter $\F_{\nmid a}$ is a m-filter in $Q$ for all $a \in Q$.
\end{lemma}

\begin{proof}
    To see upper closeness, note that if $q \in \F_{\nmid a}, q' \ge q$, and $q' x \le a$, then $q x \le q' x \le a$, so $x \le a$.
    To see multiplicative closeness, note that if $q, q' \in \F_{\nmid a}$ and $q q' x \le a$, then $q' x \le a$, so $x \le a$.
\end{proof}

There is a special case where the codense filter is easy to describe.

\begin{definition}
    An element $p \in Q$ is \emph{prime} (or a \emph{prime element}) if $p \ne 1$ and $a b \le p$ implies $a \le p$ or $b \le p$ for all $a, b \in Q$.
\end{definition}

\begin{example}
    Let $R$ be a ring.
    Then $\pp \in \Id(R)$ is prime if and only if $\pp$ is a prime ideal.
\end{example}

\begin{prop} \label{prop: describe codense filter for prime element}
    Let $p \in Q$ be a prime element.
    Then $\F_{\nmid p} = \br{q \in Q : q \not\le p}$.
\end{prop}

\begin{proof}
    $\sub$: pick $q \in \F_{\nmid p}$.
    If $q \le p$, then $q \cdot 1 \le p$, so $1 \le p$, contradicts with the definition of prime element.
    Thus, $q \not\le p$.

    \Skip $\supset$: pick any $q \not\le p$.
    If $q a \le p$, then either $q \le p$ or $a \le p$.
    However, $q \not\le p$, so $a \le p$.
    Thus, $q \in \F_{\nmid p}$.
\end{proof}

In particular, if we consider m-filter associated with a ring, then $\F_{\nmid \pp}$ is the usual localization at $\pp$:

\begin{example}
    For a ring $R$ and prime ideal $\pp \in \Id(R)$, the m-filter $\br{J \in \Id(R): J \cap (R - \pp) \neq \varnothing}$ is just $\F_{\nmid \pp}$.
\end{example}

The name codense comes from topology.

\begin{example}
    Let $X$ be a topological space and consider the quantale $\O(X)$.
    Then $\F_{\nmid \varnothing} = \br{U \in \O(X) : U \text{ is dense in } X}$.
    Thus, \Cref{lem: codense filter is filter} tells us the intersection of two open dense subsets is again open dense.
\end{example}

We can generalize this observation to any codense filter in $\O(X)$.

\begin{prop} \label{prop: describe codense in X da}
    Let $X$ be a topological space and $V \in \O(X)$.
    Then
    \[
    \F_{\nmid V} = \br{U \in \O(X) : U \cap V^c \text{ is dense in } V^c}.
    \]
\end{prop}

\begin{proof}
    $\sub$: pick $U \in \F_{\nmid V}$.
    If $U \cap (X - V)$ is not dense in $X - V$, then there exists $W \in \O(X)$ such that $(U \cap V^c) \cap (W \cap V^c) = \varnothing$ and $W \cap V^c \ne \varnothing$.
    Then $(U \cap W) \cap V^c = \varnothing$, hence $U \cap W \sub V$.
    However, $W \cap V^c \ne \varnothing$ tells us $W \not\sub V$, so $U \not\in \F_{\nmid V}$, a contradiction.

    \Skip $\supset$: if $U \cap W \sub V$, then $(U \cap V^c) \cap (W \cap V^c) \sub V \cap V^c = \varnothing$, so because $U \cap V^c$ is dense in $V^c$ we see $W \cap V^c = \varnothing$, or $W \sub V$.
\end{proof}

\smallskip

\subsection{Localization}

After defining the multiplicative filter, we could give the definition of localization. Before it, we first give the relation of local order:

\begin{definition}
    For $a, b \in M$, we say $a$ is \emph{locally less than} $b$ \emph{in one step}, or $a \preceq_{\F}^1 b$, if there exists index set $I$, $a_i \in Q, s_i \in \F$ for $i \in I$, such that $a \le \sum_{i \in I} a_i$ and $s_ia_i \le b$ for all $i \in I$. 

    We say $a$ is \emph{locally less than} $b$ \emph{in $n$ step}, or $a \preceq_\F^n b$, if there exists $c_1, \cdots, c_{n-1} \in Q$ such that $a \preceq_\F^1 c_1 \preceq_\F^1 \cdots \preceq_\F^1 c_{n-1} \preceq_\F^1 b$, and $a$ is \emph{locally less than} $b$, or $a \preceq_{\F} b$, if there exists such $n$. When there's no confusion, we could omit the $\F$ subscript. 
\end{definition}

\begin{prop}
    The local order is a preorder. 
\end{prop}

\begin{proof}
    Clearly we have $a \preceq a$. Suppose $a \preceq b$ and $b \preceq c$, we have $a \preceq^n b$ and $b \preceq^m c$ for some $n, m$, thus $a \preceq^{n+m} c$ and $a \preceq c$. 
\end{proof}

\begin{definition}
    The \emph{localization} of $M$ at $\F$, written $M_{\F}$ is the set $M / \sim$, where $a \sim b$ if and only if $a \preceq b$ and $b \preceq a$.
\end{definition}

\begin{lemma} \label{lem: 1 step local preorder closed under addition}
    Suppose $a_i \preceq^1 b_i$ for $i \in I$, then $\sum_{i \in I} a_i \preceq^1 \sum_{i \in I} b_i$. 
\end{lemma}

\begin{proof}
    By definition of local order, we have $a_i \le \sum_{j \in J_i} a_{ij}, s_{ij}a_{ij} \le b_i$ for index set $J_i$ and $s_{ij} \in \F$. Thus, we have $\sum_{i \in I} a_i \le \sum_{i \in I}\sum_{j \in J_i} a_{ij}, s_{ij}a_{ij} \le b_i \le \sum_{i \in I} b_i$, which shows $\sum_{i \in I} a_i \preceq^1 \sum_{i \in I} b_i$. 
\end{proof}

\begin{lemma} \label{lem: 1 step local preorder closed under multiplication}
    Suppose $b \preceq^1 c$, then $ab \preceq^1 ac$ for any $a \in Q$. 
\end{lemma}

\begin{proof}
     We have $b \le \sum_{i \in I} b_i$ with $s_i b_i \le c$ for some $s_i \in \F$. Thus, we have $ab \le \sum_{i \in I} ab_i$ with $s_i(ab_i) \le ac$. 
\end{proof}

\begin{lemma} \label{lem: MF is a pre module}
    We have
    \begin{enum}
        \item $Q_\F$ is a prequantale under addition $\=p + \=q = \={p + q}$ and multiplication $\=p \cdot \=q = \={p q}$.
        \item $M_\F$ is a $Q_\F$-premodule under addition $\=a + \=b = \={a + b}$ and multiplication $\=q \cdot \=a = \={q a}$.
    \end{enum}
\end{lemma}

\begin{proof}
    Part (a): for addition in both case, suppose $x \preceq z$ and $y \preceq z$, we have $x \preceq^1 x_1 \preceq^1 \cdots \preceq^1 x_{n-1} \preceq^1 z$ and $y \preceq^1 y_1 \preceq^1 \cdots \preceq^1 y_{m-1} \preceq^1 z$, thus we have $x + y \preceq^1 x_1 + y_1 \preceq^1 \cdots \preceq^1 z + z = z$ in at most $\max(m, n)$ step, thus $\=z \ge \={x + y}$ is equivalent to $\=z \ge \=x$ and $\=z \ge \=y$, so $\=x + \=y = \={x + y}$. 

    For multiplication, it suffices to show $a b \preceq a c$ provided that $b \preceq c$, since distribution comes from definition. 
    After writing $b \preceq^1 b_1 \preceq^1 b_2 \preceq^1 \cdots \preceq^1 b_{n - 1} \preceq^1 c$, it suffices to show $a b \preceq^1 a c$ provided that $b \preceq^1 c$, which is fulfilled in \Cref{lem: 1 step local preorder closed under multiplication}.
\end{proof}

Since this paper's main focus is quantale/module, we wonder when $M_\F$ is complete.

\begin{prop} \label{prop: chan of localizable}
    Let $M$ be a $Q$-module and $\F \sub Q$ be a m-filter.
    The following are equivalent:
    \begin{enum}
        \item For all $b \in M, \br{a_i}_{i \in I} \sub M$, if $a_i \preceq b$ for all $i \in I$, then $\sum_{i \in I} a_i \preceq b$.
        \item For all $b \in M$, there exists $n_b \in \NN$ such that $a \preceq_\F b$ implies $a \preceq_\F^{n_b} b$ for all $a \in M$.
    \end{enum}
\end{prop}

\begin{proof}
    (a) $\Rightarrow$ (b): Assume for an element $b \in M$, there is no such $n_b$, then there exists a sequence of elements $(a_i)_{i \in \NN} \sub M$, such that $a_i \preceq_\F b$, but $a_i \not\preceq_\F^{i-1} b$ for all $i$. 

    Consider $a = \sum_{i \in \NN} a_i$, then $a \preceq_\F b$ by assumption. Suppose $a \preceq_\F^n b$, then $a_{n+1} \le a$, thus $a_{n+1} \preceq_\F^n b$, a contradiction. 

    (b) $\Rightarrow$ (a): Suppose for $i \in I$, we have $a_i \preceq_\F^1 x_{i1} \preceq_\F^1 \cdots \preceq_\F^1 x_{i(n_b-1)} \preceq_\F^1 b$, then $\sum_{i \in I} a_i \preceq_\F^1 \sum_{i \in I} x_{i1} \preceq_\F^1 \cdots \preceq_\F^1 \sum_{i \in I} x_{i(n_b-1)} \preceq_\F^1 \sum_{i \in I} b = b$. 
\end{proof}

\begin{definition}
    When a m-filter $\F \sub Q$ satisfies any of the equivalent condition in \Cref{prop: chan of localizable} we say $\F$ is \emph{localizable} over $M$.
    In case of $M = Q$ we omit ``over $M$" and simply say $\F$ is localizable.
\end{definition}

In fact, although most of multiplicative filters we meet would be localizable, there exist non-localizable filters. 

\begin{example}
    Let $Q$ be the quantale $L_0 \sqcup L_1 \sqcup L_2$, where $L_0 \cong \LL_\omega$, $L_1$ be the collection of all ordinals at most $\omega^2$, and $L_2 = \br{\mathbf{0}}$ being the bottom. The addition is defined by a total order, with $x > y$ for $x \in L_i, y \in L_j, i < j$, and the order in each layer is naturally defined. The multiplication is defined as $xy \in L_{\min(i+j, 2)}$ for $x \in L_i, y \in L_j$, and for $0 \ge x \in L_0$, $y = a\omega^2 + b\omega + c \in L_1$ (either $a = 0$ or $a = 1, b = c = 0$), we have $xy = a\omega^2 + b\omega + \max(c+x, 0)$. 
\end{example}

\begin{prop} \label{prop: example of nonlocalizable}
    Let $\F = L_0$ be a multiplicative filter inside the above quantale $Q$, then $\F$ is not localizable over $Q$. 
\end{prop}

\begin{proof}
    Consider $x_n = n\omega \in L_1$, then suppose $x_n \le \sum_{i \in I} x_{ni}, s_{ni}x_{ni} \le y$ for $y \le x_n$, then we must have $x_{ni} \in L_1 \cup L_2$, thus there exists $i \in I$ such that $x_{ni} \ge (n-1)\omega$. (otherwise, the supremum would be at most $(n-1)\omega$.) However, $s\big((n-1)\omega\big) = (n-1)\omega$ for all $s \in \F$, thus $y \ge x_{n-1}$, and we indeed have $x_n \preceq_\F^1 x_{n-1}$. 

    Thus, by writing out a chain, we have $x_n \preceq^n x_0$ and $x_n \not\preceq^{n-1} x_0$, thus $\F$ is not localizable over $Q$ by \Cref{prop: chan of localizable}. 
\end{proof}

We would now show that in case of localizable $Q_\F$ is indeed a quantale and $M_\F$ its module.

\begin{lemma} \label{lem: QF is quantale}
    We have
    \begin{enum}
        \item If $\F$ is localizable, then $Q_\F$ is a quantale under addition $\sum_{i \in I} \={q_i} = \={\sum_{i \in I} q_i}$ and multiplication $\=p \cdot \=q = \={p q}$.
        \item If $\F$ is localizable over $M$, then $M_\F$ is a $Q$-module under addition $\sum_{i \in I} \={a_i} = \={\sum_{i \in I} a_i}$ and action $q \cdot \=a = \={q a}$.
        \item If $\F$ is localizable over $Q$ and $M$, then $M_\F$ is a $Q_\F$-module under action $\=q \cdot \=a = \={q a}$.
    \end{enum}
\end{lemma}

\begin{proof}
    For addition in all cases, suppose the elements be $x_i$. With $x_i \preceq y$, we have $x_i \preceq^{n_{y}} y$, and so $\sum_{i} x_i \preceq^{n_y} y$ by using \Cref{lem: 1 step local preorder closed under addition} for $n_y$ times. Thus, the condition $\={\sum_{i \in I} x_i} \le \=y$ is equivalent to the condition all $\={x_i} \le \=y$ holds, and $\sum_{i \in I} \={x_i} = \={\sum_{i \in I} x_i}$. 

    For multiplication, we just need to check $ab \preceq ac$, since the distribution comes from definition. Similar with the prequantale case, we could break it into a chain and apply \Cref{lem: 1 step local preorder closed under multiplication}. 
\end{proof}

\begin{example}
    We have $\br{1} \in \mF(Q)$ is localizable over any $Q$-module $M$, and $M_{\br{1}} \cong M$ (as $Q$-modules) via isomorphism $M \iso M_{\br{1}}, x \mapsto \=x$.
\end{example}

\begin{example}
    We will see later that when $R$ is a ring, $M$ is an $R$-module, and $S \sub R$ a multiplicatively closed subset, then $\N = \br{J \in \Id(R) : J \cap S \ne \varnothing} \in \mF(\Id(R))$ is localizable over $\Id(R)$ and $\Sub_R(M)$, and  $\Id(R_S) \cong \Id(R)_\N$ and $\Sub_{R_S}(M_S) \cong \Sub_R(M)_\N$ (as $\Id(R_S) = \Id(R)_\N$-module).
    In particular, for prime ideal $\pp \in \Id(R)$, we have $\Id(R_\pp) \cong \Id(R)_{\nmid \pp}$ and $\Sub_{R_\pp}(M_\pp) \cong \Sub_R(M)_{\nmid \pp}$.
    This is one motivation to define localization of quantale.
\end{example}

\begin{lemma} \label{lem: local order covered by larger mf}
    Let $M$ be a $Q$-module and pick $\F, \G \in \mF(Q)$ such that $\F \sub \G$.
    For $x, y \in M$ such that $x \preceq_\F^1 y$, we have $x \preceq_\G^1 y$.
\end{lemma}

\begin{proof}
    We have $x \le \sum_{i \in I} x_i$ with $s_i x_i \le y$ for some $s_i \in \F$.
    Thus, $x \le \sum_{i \in I} x_i$ and $s_i x_i \le y$, where $s_i \in \G$, showing $x \preceq_\G^1 y$, as desired.
\end{proof}

\begin{prop} \label{prop: MF to MG}
    Let $M$ be a $Q$-module and pick $\F, \G \in \mF(Q)$ such that $\F \sub \G$.
    Then
    \begin{enum}
        \item There is a $Q$-premodule homomorphism $\vp : M_\F \to M_\G$.
        \item If both $\F$ and $\G$ are localizable over $M$, then $\vp$ is a $Q$-module homomorphism.
    \end{enum}
\end{prop}

\begin{proof}
    Part (a): we first show $\vp$ is well-defined.
    To see this, pick $a, b \in M$ such that $a \preceq_\F b$.
    Then $a \preceq_\F^1 x_1 \preceq_\F^1 x_2 \preceq_\F^1 \cdots \preceq_\F^1 x_{n - 1} \preceq_\F^1 b$.
    Thus, $a \preceq_\G^1 x_1 \preceq_\G^1 x_2 \preceq_\G^1 \cdots \preceq_\G^1 x_{n - 1} \preceq_\G^1 b$ by \Cref{lem: local order covered by larger mf}, showing $a \preceq_\G b$, hence $\vp$ is well-defined (as map between sets).

    To check that $\vp$ preserves finite addition, just note that for $a, b \in M$, we have $\vp(\=a + \=b) = \vp(\={a + b}) = \={a + b} = \=a + \=b = \vp(\=a) + \vp(\=b)$, as desired.

    To check that $\vp$ preserves multiplication, just note that for $a \in M, q \in Q$, we have $\vp(q \=a) = \vp(\={q a}) = \={q a} = q \=a = q \vp(\=a)$, as desired.

    \Skip Part (b): by (the proof of) Part (a) $\vp$ is well defined and preserves multiplication.
    
    To check that $\vp$ preserves addition, just note that for $a_i \in M$, we have $\vp(\sum_{i \in I} \={a_i}) = \vp(\={\sum_{i \in I} a_i}) = \={\sum_{i \in I} a_i} = \sum_{i \in I} \={a_i} = \sum_{i \in I} \vp(\={a_i})$.
\end{proof}

Let $\F, \G \in \mF(Q)$.
If $\F$ is localizable over $M$, then we obtain a $Q$-module $M_\F$, so we can localize it at $\G$ to obtain $(M_\F)_\G$.
Similarly, if $\G$ is localizable over $M$, then we obtain $(M_\G)_\F$.
A natural question to ask is when they are $Q$-module and if so do they agree?
We are now going to show that when all $\F, \G, \F + \G$ are localizable over $M$, then we have an isomorphism between $Q$-modules $(M_\F)_\G \cong (M_\G)_\F$.

\begin{lemma} \label{lem: factors through localization}
    Let $M, N$ be $Q$-modules with $Q$-linear map $\vp : M \to N$.
    Pick $\F \in \mF(Q)$ such that $s v = v$ for all $s \in \F$ and $v \in N$.
    Then
    \begin{enum}
        \item $\vp$ factors through $M_\F$ to give a homomorphism between $Q$-premodules $\~\vp : M_\F \to N, \=x \mapsto \vp(x)$.
        \item If $\F$ is localizable over $M$, then $\~\vp$ is a $Q$-module homomorphism.
    \end{enum}
\end{lemma}

\begin{proof}
    It suffices to show that when $\=x \le \=y$ in $M_\F$, we have $\vp(x) \le \vp(y)$.
    It suffices to show this when $x \preceq_\F^1 y$.
    However, then $x \le \sum_{i \in I} x_i$ and $s_i x_i \le y$ for some $s_i \in \F$, so $\vp(y) \ge \sum_{i \in I} s_i \vp(x_i) = \sum_{i \in I} \vp(x_i) \ge \vp(x)$, as desired.
\end{proof}

\begin{lemma} \label{lem: repr in MF is G local <= implies F+G local <=}
    Let $\F, \G \in \mF(Q)$ such that $\F$ is localizable over $M$.
    For $x, y \in M$, if $\=x \preceq_\G \=y$ in $M_\F$, then $x \preceq_{\F + \G} y$ in $M$.
\end{lemma}

\begin{proof}
    We have $\=x \preceq_\G^1 \={x_1} \preceq_\G^1 \cdots \preceq_\G^1 \={x_{n - 1}} \preceq_\G^1 \=y$ in $M_\F$ for some $x_i \in M$.
    Thus, it suffices to show that for $a, b \in M$ satisfy $\=a \preceq_\G^1 \=b$ in $M_\F$, we have $a \preceq_{\F + \G} b$.
    We have $\=a \le \sum_{i \in I} \={a_i}$ and $\={t_i a_i} \le \=b$ in $M_\F$ for some $s_i \in \G$.
    Thus, $a \preceq_\F \sum_{i \in I} a_i$ and $t_i a_i \preceq_\F b$.
    Then since $\F$ is localizable, we see $\sum_{i \in I} t_i a_i \preceq_\F b$.
    
    Now we have $a \preceq_\F^1 z_1 \preceq_\F^1 z_2 \preceq_\F^1 \cdots \preceq_\F^1 z_{m - 1} \preceq_\F^1 \sum_{i \in I} a_i \preceq_\G^1 \sum_{i \in I} t_i a_i \preceq_\F^1 z_1' \preceq_\F^1 z_2' \preceq_\F^1 \cdots \preceq_\F^1 z_{m' - 1}' \preceq_\F^1 b$ for some $m, m' \in \NN, z_j, z_j' \in M$.
    Now by \Cref{lem: local order covered by larger mf} we have $a \preceq_{\F + \G}^1 z_1 \preceq_{\F + \G}^1 z_2 \preceq_{\F + \G}^1 \cdots \preceq_{\F + \G}^1 z_{m - 1} \preceq_{\F + \G}^1 \sum_{i \in I} a_i \preceq_{\F + \G}^1 \sum_{i \in I} t_i a_i \preceq_{\F + \G}^1 z_1' \preceq_{\F + \G}^1 z_2' \preceq_{\F + \G}^1 \cdots \preceq_{\F + \G}^1 z_{m' - 1}' \preceq_{\F + \G}^1 b$, hence $a \preceq_{\F + \G} b$, as desired.
\end{proof}

\begin{prop} \label{prop: MFG = MGF}
    Let $\F, \G \in \mF(Q)$.
    If both $\F$ and $\F + \G$ are localizable over $M$, then $\G$ is localizable over $M_\F$ and we have isomorphism between $Q$-modules $M_{\F + \G} \iso (M_\F)_\G, \=x \mapsto \={\=x}$.
\end{prop}

\begin{proof}
    Using the notation as in \Cref{prop: chan of localizable}, for $y \in M$ we will show we can pick $n_{\=y}^\G = n_y^{\F + \G}$, where the left hand side corresponds to $\=x \preceq_\G \={y}$ in $M_\F$ and right hand side corresponds to $x \preceq_{\F + \G} y$ in $M$.
    To see this, pick $x, y \in M$ such that $\=x \preceq_\G \=y$ in $M_\F$.
    Then by \Cref{lem: repr in MF is G local <= implies F+G local <=} we have $x \preceq_{\F + \G} y$.

    Thus, we have $x \preceq_{\F + \G}^1 w_1 \preceq_{\F + \G}^1 w_2 \preceq_{\F + \G}^1 \cdots \preceq_{\F + \G}^1 w_{n_0 - 1} \preceq_{\F + \G}^1 y$, where $n_0 = n_y^{\F + \G}$.
    We now claim that for $a, b \in M$, we have $a \preceq_{\F + \G}^1 b$ implies $\=a \preceq_\G^1 \=b$ in $M_\F$, which finishes the first part of the proof that $\G$ is localizable over $M_\F$.
    However, this claim is true because we have $a \le \sum_{i \in I} a_i$ with $s_i t_i a_i \le b$ for some $s_i \in \F, t_i \in \G$.
    Then in $M_\F$, we have $\=a = \sum_{i \in I} \={s_i a_i}$ and $t_i \={s_i a_i} \le \=b$, so $\=a \preceq_\G \=b$ in $M_\F$, as desired.

    We now prove the second part of this proposition.
    Let $Q$-linear map $\vp : M \to (M_\F)_\G$ be the composition $M \to M_\F \to (M_\F)_\G$ (valid by \Cref{prop: MF to MG}).
    Then for any $q \in \F + \G$ and $x \in M$, we claim that $q \={\=x} = \={\=x}$ in $(M_\F)_\G$.
    It suffices to show $\=x \preceq_\G \={q x}$ in $M_\F$.
    However, $q \ge s t$ for some $s \in \F, t \in \G$, so $\=x = \={s x}$ and $t \={s x} \le \={q x}$ in $M_\F$, as desired.

    Therefore, by \Cref{lem: factors through localization} we obtain a $Q$-linear map $\~\vp : M_{\F + \G} \to (M_\F)_\G, \=x \mapsto \={\=x}$.
    We have $\~\vp$ is surjective because both $M \to M_\F$ and $M_\F \to (M_\F)_\G$ are surjective (so that $\vp$ is surjective).
    Thus, to conclude the proof, it suffices to show $\vp$ is injective.

    Now pick any $x, y \in M$ such that $\={\=x} = \={\=y}$ in $(M_\F)_\G$, then $\=x \preceq_\G \=y$ in $M_\F$, so by \Cref{lem: repr in MF is G local <= implies F+G local <=} we have $x \preceq_{\F + \G} y$.
    Similarly, $y \preceq_{\F + \G} x$, hence $\=x = \=y$ in $M_{\F + \G}$.
    As a result, $\vp$ is injective, as desired.
\end{proof}

Now, we exhibit a special kind of multiplicative filters that will be useful later. 

\begin{definition}
    A localizable multiplicative filter $\F$ is \emph{1-step over $M$} if $a \preceq_\F b$ implies $a \preceq_\F^1 b$, or equivalently, $a \preceq_\F^2 b$ implies $a \preceq_\F^1 b$.
\end{definition}

\bigskip

\section{Shrinkable and Suspension}
\label{sec: shrinkable and susp}

In this section, $Q$ is a quantale and $M$ a $Q$-module.

\smallskip

\subsection{Definitions and Basic Properties}

In many places we have inequalities like $x \le \sum_{i \in I} x_i$.
We wonder when we can ``shrink" $x_i$'s to make it become an equality.
We want to say that in a ``shrinkable" module such inequality can be made into an equality, and we want the usual modules to be shrinkable in order for our theory to be general.

\begin{example}
    Let $X$ be a topological space and $U, U_i (i \in I)$ be open sets with $U \sub \bigcup_{i \in I} U_i$, then $U = \bigcup_{i \in I} (U \cap U_i)$.
\end{example}

\begin{definition}
    A map $f : A \to B$ between posets is \textit{order-preserving} if $f(a_1) \le f(a_2)$ for all $a_1, a_2 \in A$ such that $a_1 \le a_2$.
\end{definition}

\begin{definition}
    We say a surjective order-preserving map $f : A \to B$ between posets is \textit{shrinkable} if for all $a \in A, b \in B$ such that $b \le f(a)$, there exists $a' \in f\inv(b)$ such that $a' \le a$.
\end{definition}

This condition can be explained by the following diagram:

\[\begin{tikzcd}
	a & {f(a)} \\
	a' & {b} \\
	\arrow["f", from=1-1, to=1-2]
	\arrow["\ge"', no head, from=1-1, to=2-1]
	\arrow["\ge", no head, from=1-2, to=2-2]
	\arrow["{f\inv}", dashed, from=2-2, to=2-1]
\end{tikzcd}\]

We now introduce the suspension of a semilattice in order to define what is a shrinkable complete semilattice.

\begin{lemma} \label{lem: preorder on power set of L}
    Let $L$ be a semilattice.
    Then there is a preorder on the collection of all nonempty subsets of $L$, defined by $S \le T$ if and only if for all $s \in S$, there exists finite $T_s \sub T$ such that $s \le \sum T_s$.
\end{lemma}

\begin{proof}
    The only nontrivial part is transitivity.
    Pick nonempty susbsets $S, W, T$ of $L$ such that $S \le W$ and $W \le T$.
    Then for all $s \in S$, there exists $w_1, \ldots, w_n \in W$ such that $s \le \sum_{i = 1}^n w_i$, and for each $1 \le i \le n$, there exists finite $T_i \sub T$ such that $w_i \le \sum T_i$.
    Now take $T_s = \bigcup_{i = 1}^n T_i$, then $s \le \sum_{i = 1}^n w_i \le \sum_{i = 1}^n \sum T_i = \sum T_s$, so $S \le T$, as desired.
\end{proof}

\begin{definition}
    Let $L$ be a semilattice.
    The \emph{suspension} of $L$, denoted by $\Sigma L$, is the quotient $\br{\text{nonempty subsets of } L} / \sim$, where the equivalence is defined by $S \sim T$ if and only if $S \le T$ and $T \le S$, where $\le$ is as in \Cref{lem: preorder on power set of L}. 
\end{definition}

\begin{prop} \label{prop: basic properties of suspension}
    We have
    \begin{enum}
        \item If $L$ is a semilattice, then $\Sigma L$ is a complete semilattice under addition $\sum_{i \in I} S_i = \bigcup_{i \in I} S_i$.
        \item If $L$ is a complete semilattice, then we have a surjective homomorphism between complete semilattices $\sigma : \Sigma L \to L, S \mapsto \sum_{s \in S} s$.
        \item If $Q$ is a prequantale, then $\Sigma Q$ admits a quantale structure, where multiplication is $S \cdot T = \br{s t : s \in S, t \in T}$.
        \item If $Q$ is a quantale, then $\sigma$ is a homomorphism between quantales.
        \item If $M$ is a $Q$-module, then $\Sigma M$ is a $\Sigma Q$-module via $S \cdot A = \br{s a : s \in S, a \in A} (S \in \Sigma Q, A \in \Sigma M)$.
    \end{enum}
\end{prop}

\begin{proof}
    Part (a): Since the partial order in $\Sigma L$ is element wise, we have $\bigcup_{i \in I} S_i \le T$ equivalent to $S_i \le T$ for all $T$. 

    Part (b): We need to check $\sigma(\sum_{i \in I}) = \sum_{i \in I} \sigma(S_i)$. This follows by $\sum_{i \in I}\sum_{s \in S_i} s = \sum_{s \in \bigcup_{i \in I} S_i} s$. 
    
    Part (c): Consider $\sum_{i \in I} (S_i \cdot T)$ and $\left(\sum_{i \in I} S_i\right) \cdot T$, they are both in shape $\br{st: s \in \bigcup_{i \in I} S_i, t \in T}$, thus the multiplication is distributive, and $\Sigma Q$ is a quantale. 

    Part (d): As $\sigma(ST) = \sum_{\substack{s \in S\\ t \in T}} st = \sum_{s \in S}\sum_{t \in T} st = \sigma(S)\sigma(T)$, $\sigma$ is indeed a quantale homomorphism given it is a complete semilattice homomorphism. 

    Part (e): We only need to check the distribution, which is clear by similar proof with part (c). 
\end{proof}

\begin{definition}
    Let $L$ be a complete semilattice. Then we use $\sigma_L: \Sigma L \to L$ to denote the canonical map from $\Sigma L$ to $L$.
    We would omit the subscript if there is no confusion.  
\end{definition}

For a complete semilattice $L$, note we have a homomorphism between (non-complete) semilattices $\iota : L \to \Sigma L, x \mapsto \br{x}$.
When $L$ is a quantale, then $\iota$ is a homomorphism between prequantales.
Also, we have $\sigma \circ \iota = \id$ and $\iota \circ \sigma(S) \ge S$ for all $S \in \Sigma L$.

Finally, let us see the definition for a complete semilattice to be shrinkable.

\begin{definition}
    We say a complete semilattice $L$ is \textit{shrinkable} if $\sigma_L : \Sigma L \to L$ is shrinkable.
\end{definition}

\begin{remark}
    Unpacking the definition we see a complete semilattice $L$ is shrinkable if for $x, x_i \in L$ with $x \le \sum_{i \in I} x_i$, there exist $\br{y_j}_{j \in J} \sub L$ such that $x = \sum_{j \in J} y_j \le^* \sum_{i \in I} x_i$.
\end{remark}

Most of the quantales we encounter would be shrinkable.
The rule of thumb is, when you see a ``naturally arising" quantale/module, then with high probability it is shrinkable.

\begin{example}
    Let $Q$ be an idempotent quantale, then $Q$ is shrinkable.
    This is because if $x \le \sum_{i \in I} x_i$, then $x = x^2 = \sum_{i \in I} (x x_i)$ and $x x_i \le x_i$.
    In particular, $\O(X)$ is shrinkable for all topological space $X$.
\end{example}

\begin{example}
    Let $P$ be an $R$-module.
    Then $\Sub_R(P)$ is shrinkable.
    This is because if $N \sub \sum_{i \in I} N_i$, then $N = \sum_{n \in N} (n)$ and for all $n \in N$, there exists finite $I_n \sub I$ such that $n \in \sum_{i \in I_n} N_i$, so that $(n) \sub \sum_{i \in I_n} N_i$.
\end{example}

\begin{example}
    Let $L$ be a shrinkable semilattice and $x \in L$. 
    Then $L_{\ge x}$ is shrinkable. 
\end{example}

In fact, the shrinkable property could simplify the condition for localization.

\begin{lemma} \label{lem: shrinkable 1 step equality}
    Let $M$ be a shrinkable $Q$-module and $\F \in \mF(Q)$.
    For $a, b \in M$, if $a \preceq_\F^1 b$, then $a = \sum_{i \in I} a_i$ and $s_i a_i \le b$ for some $a_i \in M, s_i \in \F$.
\end{lemma}

\begin{proof}
    We have $a \le \sum_{j \in J} a_j'$ with $s_j' a_j' \le b$ for some $a_j' \in M, s_j' \in \F$.
    Now since $M$ is shrinkable $a = \sum_{i \in I} a_i$ and $a_i \le^* \sum_J a_j'$ for some $a_i \in M$.
    For $i \in I$, pick finite subset $J_i \sub J$ such that $a_i \le \sum_{j \in J_i} a_j'$ and let $s_i = \prod_{j \in J_i} s_j'$.
    Then $a = \sum_{i \in I} a_i$ and $s_i a_i \le (\prod_{j \in J_i} s_j')(\sum_{j \in J_i} a_j') = \sum_{j \in J_i} (s_j' a_j' \prod_{j' \in J_i - \br{j}} s_{j'}') \le \sum_{j \in J_i} b = b$, as desired.
\end{proof}

\begin{prop} \label{prop: shrinkable descending}
    Let $M$ be a shrinkable $Q$-module and $\F \in \mF(Q)$.
    For $a, b \in M, n \in \NN$, if $a \preceq_\F^n b$, then there exists $x_1, \ldots, x_n \in M$ such that $a \ge x_1 \ge x_2 \ge \cdots \ge x_n, a \preceq_\F^1 x_1 \preceq_\F^1 x_2 \preceq_\F^1 \cdots \preceq_\F^1 x_n$, and $x_n \le b$.
\end{prop}

\begin{proof}
    We use induction on $n$.
    If $n = 1$, then by \Cref{lem: shrinkable 1 step equality} $a = \sum_{i \in I} a_i$ with $s_i a_i \le b$ for some $a_i \in M, s_i \in \F$, so we can take $x_1 = \sum_{i \in I} s_i a_i$.
    Thus, we assume $n > 1$.

    We have $a \preceq^1 y_1 \preceq^1 y_2 \preceq^1 \cdots \preceq^1 y_{n - 1} \preceq^1 b$ for some $y_k \in M$.
    Then by \Cref{lem: shrinkable 1 step equality} $a = \sum_{i \in I} a_i$ with $s_i a_i \le y_1$ for some $a_i \in M, s_i \in \F$.
    We let $x_1 = \sum_{i \in I} s_i a_i$, then $a \ge x_1$ and $a \preceq x_1$.
    Also, $x_1 \le y_1$, so $x_1 \preceq y_2$.
    Thus, we can apply induction hypothesis to $x_1 \preceq^{n - 1} b$ to produce $x_2, x_3, \ldots, x_n$.
\end{proof}

\begin{prop} \label{prop: equality once by once}
    Let $M$ be a shrinkable $Q$-module and $\F \in \mF(Q)$.
    For $a, a' \in M$, if $\=a = \={a'}$ in $M_\F$, then there exists $n \in \NN$ and $a_0 = a, a_1, \ldots, a_{n - 1}, a_n = a'$ in $M$ such that $a_k \preceq_\F^1 a_{k + 1} \preceq_\F^1 a_k$ for all $0 \le k \le n - 1$.
\end{prop}

\begin{proof}
    By \Cref{prop: shrinkable descending} $a \ge x_1 \ge \cdots \ge x_m, a \preceq^1 x_1 \preceq^1 \cdots \preceq^1 x_m, x_m \le a'$ for some $x_k$.
    Also, $a' \ge y_1 \ge \cdots \ge y_m, a' \preceq^1 y_1 \preceq^1 \cdots \preceq^1 y_m, y_m \le a$ for some $y_k$ (we can possibly enlarge $m$ to make the two $m$'s agree).
    Then $a_0 = a, a_1 = x_1 + y_m, a_2 = x_2 + y_{m - 1}, \ldots, a_m = x_m + y_1, a_{m + 1} = a'$ works.
\end{proof}

The next proposition shows shrinkability is preserved under localization.

\begin{prop} \label{prop: shrinkable preserving}
    Let $M$ be a shrinkable $Q$-module.
    If $\F \sub Q$ is a localizable m-filter over $M$, then $M_\F$ is shrinkable.
\end{prop}

\begin{proof}
    Suppose $\=x \le \sum_{i \in I} \={x_i}$ in $M_\F$, then $x \preceq_\F \sum_{i \in I} x_i$. By \Cref{prop: shrinkable descending}, we have some $y$ satisfying $x \ge y$, $x \preceq_\F y$ and $y \le \sum x_i$. Thus, $y \preceq_\F x$ and so $\=y = \=x$. As $y \le \sum_{i \in I} x_i$ and $M$ is shrinkable, $y = \sum_{j \in J} y_j$ and $y_j \le^* \sum_{i \in I} x_i$, giving $\=x = \=y = \sum_{j \in J} \={y_j}$. Thus $M_\F$ is shrinkable. 
\end{proof}

\smallskip

\subsection{Applications in Localization}

It turns out that shrinkable complete semilattice has a nice characterization.

\begin{theorem} \label{thm: Suspension Shrinkable Meet}
    Let $L$ be a complete semilattice.
    Then $L$ is shrinkable if and only if for $S, T \in \Sigma L$ such that $\sigma(S) \wedge \sigma(T)$ exists in $L$, we have $S \wedge T$ exists and $\sigma(S \wedge T) = \sigma(S) \wedge \sigma(T)$.
\end{theorem}

\begin{proof}
    $\Leftarrow$: for $t \le s$ and $S \in \sigma\inv(s)$, we have $S \wedge \iota(t) \in \sigma\inv(t)$ and $S \wedge \iota(t) \le S$.

    \Skip $\Rightarrow$: we finish the proof in 2 steps.
    
    \Skip Step 1: if $\sigma(S) = \sigma(T) = x \in L$, then $S \wedge T$ exists and $\sigma(S \wedge T) = x$.
    
    Let $S = \br{s_i}_{i \in I}$, then for each $i \in I$, there exists $S_i \in \sigma\inv(s_i)$ with $S_i \le T$ (by shrinkability).
    Also, we have $S_i \le \iota(s_i) \le S$, so after letting $\Sigma = \sum_{i \in I} S_i$ we have $\Sigma \le S$ and $\Sigma \le T$, so $\sigma(S \wedge T) = x$ as $\Sigma \le S \wedge T \le S$ and $\sigma(\Sigma) = \sum_{i \in I} s_i = x = \sigma(S)$.

    \Skip Step 2: we conclude the general case by the following diagrams:
    \[\begin{tikzcd} [column sep = small]
    S & s & T & t \\
    {S'} & {s \wedge t} & {T'} & {s \wedge t} \\
    S & s & T & t \\
    & {S \wedge T} & {s \wedge t} \\
    & {S' \wedge T'} & {s \wedge t}
    \arrow[from=1-1, to=1-2]
    \arrow[no head, from=1-1, to=2-1]
    \arrow[no head, from=1-2, to=2-2]
    \arrow[from=1-3, to=1-4]
    \arrow[no head, from=1-3, to=2-3]
    \arrow[no head, from=1-4, to=2-4]
    \arrow[dashed, from=2-2, to=2-1]
    \arrow[dashed, from=2-4, to=2-3]
    \arrow[from=3-1, to=3-2]
    \arrow[no head, from=3-1, to=4-2]
    \arrow[no head, from=3-2, to=4-3]
    \arrow[from=3-3, to=3-4]
    \arrow[no head, from=3-3, to=4-2]
    \arrow[no head, from=3-4, to=4-3]
    \arrow[dashed, from=4-2, to=4-3]
    \arrow[no head, from=4-2, to=5-2]
    \arrow[equals, from=4-3, to=5-3]
    \arrow[from=5-2, to=5-3]
    \end{tikzcd}\]
\end{proof}

\Cref{thm: Suspension Shrinkable Meet} illuminates some natures of localization.

\begin{lemma} \label{lem: 1-step wedge}
    Let $M$ be a shrinkable $Q$-module and $\F \in \mF(Q)$.
    For $x, y, x', y' \in M$, if $x \wedge y$ exists in $M$ and $x \preceq^1 x', y \preceq^1 y'$, then $x' \wedge y'$ exists and $x \wedge y \preceq^1 x' \wedge y'$.
\end{lemma}

\begin{proof}
    We have $x = \sum_I x_i$ with $s_i x_i \le x'$ for some $s_i \in \F$.
    Also, $y = \sum_J y_j$ with $t_j y_j \le y'$ for some $t_j \in \F$.
    
    Let $S = \br{x_i}_{i \in I} \in \Sigma M, T = \br{y_j}_{j \in J} \in \Sigma M$.
    Then $S \wedge T \in \Sigma M$ exists and $\sigma(S \wedge T) = x \wedge y$ by \Cref{thm: Suspension Shrinkable Meet}.
    Let $S \wedge T = \br{z_k}_{k \in K}$, then for all $k \in K$, we have $z_k \le^* \sum_I x_i$, so there exists finite $I_k \sub I$ such that $z_k \le \sum_{i \in I_k} x_i$; let $s_k' = \prod_{i \in I_k} s_i$.
    Similarly, $z_k \le^* \sum_J y_j$, so there exists finite $J_k \sub J$ such that $z_k \le \sum_{j \in J_k} y_j$; let $t_k' = \prod_{j \in J_k} t_j$.

    Then $x \wedge y = \sum_{k \in K} z_k$ and $s_k' t_k' z_k \le t_k' (\prod_{i \in I_k} s_i) \sum_{i \in I_k} x_i \le x'$.
    Similarly, $s_k' t_k' z_k \le y'$.
    Thus, $x' \wedge y'$ exists and $x \wedge y \preceq^1 x' \wedge y'$, which proves the claim.
\end{proof}

\begin{prop} \label{prop: localization preserves wedge}
    Let $M$ be shrinkable and $\F \in \mF(Q)$.
    For $a, b \in M$, the following are equivalent:
    \begin{enum}
        \item $a \wedge b$ exists in $M$.
        \item $\=a \wedge \=b$ exists in $M_\F$.
        \item There exists $c \in M$ such that $\=c \le \=a$ and $\=c \le \=b$ in $M_\F$.
    \end{enum}
    Moreover, when such condition is met, we have $\=a \wedge \=b = \={a \wedge b}$ in $M_\F$.
\end{prop}

\begin{proof}
    (a) $\Rightarrow$ (b): we have $\={a \wedge b} \le \=a, \={a \wedge b} \le \=b$ in $M_\F$.
    Now take any $c \in M$ such that $\=c \le \=a$ and $\=c \le \=b$ in $M_\F$.
    Then $c \preceq^n a$ and $c \preceq^n b$ for some $n \in \NN$.
    Thus, $c \preceq^1 x_1 \preceq^1 x_2 \preceq^1 \cdots \preceq^1 x_{n - 1} \preceq^1 a$ and $c \preceq^1 y_1 \preceq^1 y_2 \preceq^1 \cdots \preceq^1 y_{n - 1} \preceq^1 b$ for some $x_i, y_i \in M$.

    Therefore, by \Cref{lem: 1-step wedge} all $x_i \wedge y_i$ exist and $c \preceq^1 x_1 \wedge y_1 \preceq^1 x_2 \wedge y_2 \preceq^1 \cdots \preceq^1 x_{n - 1} \wedge y_{n - 1} \preceq^1 a \wedge b$, showing $\=c \le \={a \wedge b}$, hence $\=a \wedge \=b = \={a \wedge b}$ when $a \wedge b$ exists.

    \Skip (b) $\Rightarrow$ (c): this is immediate.

    \Skip (c) $\Rightarrow$ (a): we have $c \preceq^n a, c \preceq^m b$ for some $n, m \in \NN$.
    We use induction on $m + n$.
    If $m = n = 0$, then $c \le a, c \le b$, so $a \wedge b$ exists.
    If $m + n > 0$, then without loss of generality $n > 0$.
    By \Cref{prop: shrinkable descending} $c \preceq^1 x \preceq^{n - 1} a$ for some $x \le c$.
    Then we have $x \preceq^{n - 1} a$ and $x \preceq^m b$, so by induction hypothesis $a \wedge b$ exists, as desired.
\end{proof}

\begin{cor} \label{cor: MF -> MG is shrinkable}
    Let $M$ be shrinkable and $\F, \G \in \mF(Q)$ such that $\F \sub \G$.
    The map (between $Q$-premodules) $M_\F \to M_\G$ is shrinkable.
\end{cor}

\begin{proof}
    Denote the map $M_\F \to M_\G$ by $\vp$.
    Pick $\al, \beta \in M_\F$ such that $\vp(\al) \le \vp(\beta)$.
    Pick $x, y \in M$ such that $\=x = \al$ and $\=y = \beta$ in $M_\F$.
    Then $z = x \wedge y$ exists in $M$ by \Cref{prop: localization preserves wedge}, so $\=z = \=x \wedge \=y = \=x$ in $M_\G$ by \Cref{prop: localization preserves wedge}.
    We now have $\al \wedge \beta = \=z$ exists in $M_\F$ by \Cref{prop: localization preserves wedge} and $\vp(\al \wedge \beta) = \=x = \vp(\al)$, showing $\vp$ is shrinkable.
\end{proof}

\begin{lemma} \label{lem: Filter Merge}
    Let $M$ be a shrinkable $Q$-module and $\F, \G \in \mF(Q)$.
    For $a, b \in M, n, m \in \NN$, if $a \preceq_\F^n b$ and $a \preceq_\G^m b$, then $a \preceq_{\F \G}^{n + m - 1} b$.
\end{lemma}

\begin{proof}
    Step 1: pre-processing.
    
    Let $a \preceq_\F^1 x_1 \preceq_\F^1 x_2 \preceq_\F^1 \cdots \preceq_\F^1 x_{n - 1} \preceq_\F^1 b$ and $a \preceq_\G^1 y_1 \preceq_\G^1 y_2 \preceq_\G^1 \cdots \preceq_\G^1 y_{m - 1} \preceq_\G^1 b$.
    For simplicity, let $x_0 = y_0 = a$ and $x_n = y_m = b$.
    By \Cref{prop: shrinkable descending} we can assume $a \ge x_1 \ge x_2 \ge \cdots \ge x_{n - 1}$ and $a \ge y_1 \ge y_2 \ge \cdots \ge y_{m - 1}$.

    Then observe that $a + b \preceq_\F^1 x_1 + b \preceq_\F^1 \cdots \preceq_\F^1 x_{n - 1} + b \preceq_\F^1 b$ and $a + b \preceq_\G^1 y_1 + b \preceq_\G^1 \cdots \preceq_\G^1 y_{m - 1} + b \preceq_\G^1 b$, so after adding $b$ to all $x_k$'s and $y_k$'s we can assume $a \ge x_1 \ge \cdots \ge x_{n - 1} \ge b$ and $a \ge y_1 \ge \cdots \ge y_{m - 1} \ge b$.
    
    \Skip Step 2: we claim $x_u \wedge y_v \preceq_{\F \G}^1 x_{u + 1} \wedge y_v + x_u \wedge y_{v + 1}$ for $0 \le u \le n - 1, 0 \le v \le m - 1$.

    By \Cref{lem: 1-step wedge} we have $x_u \wedge y_v \preceq_\F^1 x_{u + 1} \wedge y_v$ and $x_u \wedge y_v \preceq_\G^1 x_u \wedge y_{v + 1}$.
    Thus, $x_u \wedge y_v = \sum_I \al_i = \sum_J \beta_j$ with $s_i \al_i \le x_{u + 1} \wedge y_v$ and $t_j \beta_j \le x_u \wedge y_{v + 1}$ for some $\al_i, \beta_j \in M, s_i \in \F, t_j \in \G$.
    
    Let $A = \br{\al_i}_{i \in I} \in \Sigma M, B = \br{\beta_j}_{j \in J} \in \Sigma M$, then $C = A \wedge B$ exists in $\Sigma M$ and $\sigma C = x_u \wedge y_v$ by \Cref{thm: Suspension Shrinkable Meet}.
    Let $C = \br{\gamma_k}_{k \in K}$.
    For each $k \in K$, there exist finite $I_k \sub I$ and $J_k \sub J$ such that $\gamma_k \le \sum_{i \in I_k} \al_i$ and $\gamma_k \le \sum_{j \in J_k} \beta_j$ (follows from $C \le A$ and $C \le B$ in $\Sigma M$).
    Let $s_k' = \prod_{i \in I_k} s_i$ and $t_k' = \prod_{j \in J_k} t_j$.

    However, then $x_u \wedge y_v = \sum_K \gamma_k$ and $(s_k' + t_k') \gamma_k \le x_{u + 1} \wedge y_v + x_u \wedge y_{v + 1}$, showing $x_u \wedge y_v \preceq_{\F \G}^1 x_{u + 1} \wedge y_v + x_u \wedge y_{v + 1}$, as desired.

    \Skip Step 3: conclude the proof.

    By Step 2 after letting $w_k = \sum \br{x_u + y_v : u + v = k, 0 \le u \le n, 0 \le v \le m}$ we have $w_k \preceq_{\F \G}^1 w_{k + 1}$, so we are done by the fact that $w_0 = a$ and $w_{n + m - 1} = x_n \wedge y_{m - 1} + x_{n - 1} \wedge y_m = b$.
\end{proof}

\Cref{lem: Filter Merge} can be used to show that localizable/1-step filters are closed under multiplication.

\begin{cor} \label{cor: localizable closed under multiplication}
    Let $M$ be a shrinkable $Q$-module.
    If $\F, \G \sub Q$ are localizable m-filters over $M$, then $\F\G$ is also localizable over $M$.
\end{cor}

\begin{proof}
    Let $a \preceq_{\F \G} b$, then $a \preceq_\F b$, so $a \preceq_\F^{n_b^\F} b$ by \Cref{prop: chan of localizable}.
    Similarly, $a \preceq_\G^{n_b^\G} b$, so by \Cref{lem: Filter Merge} $a \preceq_{\F \G}^{n_b^\F + n_b^\G - 1} b$, showing we can take $n_b^{\F\G} = n_b^\F + n_b^\G - 1$, hence $\F \G$ is localizable over $M$ by \Cref{prop: chan of localizable}.
\end{proof}

\begin{cor} \label{cor: 1-step closed under multiplication}
    Let $M$ be a shrinkable $Q$-module.
    If $\F, \G \sub Q$ are 1-step m-filters overe $M$, then $\F\G$ is also 1-step over $M$.
\end{cor}

\begin{proof}
    Let $a \preceq_{\F \G} b$, then $a \preceq_\F b$, so $a \preceq_\F^1 b$.
    Similarly, $a \preceq_\G^1 b$, so by \Cref{lem: Filter Merge} we see $a \preceq_{\F \G}^1 b$, as desired.
\end{proof}

\begin{lemma} \label{lem: glue together finite mfilters}
    Suppose $M$ is shrinkable.
    Pick $\F_1, \ldots, \F_n \in \mF(Q)$ along with $x_1, \ldots, x_n \in M$ such that $x_i \preceq_{\F_i + \F_j}^1 x_j$ for all $i \ne j$. 
    Then there exists $x \in M$ such that $x \preceq_{\F_i}^1 x_i \preceq_{\F_i}^1 x$ for all $1 \le i \le n$. 
\end{lemma}

\begin{proof}
    We can write (when $i \ne j$) (after possibly merge the indices) $x_i = \sum_{k \in K} x_k^{ij}$ and $s_k^{ij} t_k^{ij} x_k^{ij} \le x_j$ for some $x_k^{ij} \in M, s_k^{ij} \in \F_i, t_k^{ij} \in \F_j$ by \Cref{lem: shrinkable 1 step equality}.

    For each $i$, let $A^{ij} = \br{x_k^{ij}}_{k \in K} \in \Sigma M$, then $\sigma A^{ij} = x_i$, so by \Cref{thm: Suspension Shrinkable Meet} $A^i = \bigwedge_{j \ne i} A^{ij}$ exists and $\sigma A^i = x_i$.
    After possibly merging indices we write $A^i = \br{x_\ell^i}_{\ell \in L}$.

    Then for each $i \ne j$ and $\ell \in L$, there exists finite $K_\ell^{ij} \sub K$ such that $x_\ell^i \le \sum_{k \in K_\ell^{ij}} x_k^{ij}$.
    Let $s_\ell^{ij} = \prod_{k \in K_\ell^{ij}} s_k^{ij}, t_\ell^{ij} = \prod_{k \in K_\ell^{ij}} t_k^{ij}$, then $x_i = \sum_{\ell \in L} x_\ell^i$ and $s_\ell^{ij} t_\ell^{ij} x_\ell^i \le x_j$.

    Now consider $x\up{i} = \sum_{\ell \in L} (\prod_{j \ne i} s_\ell^{ij}) x_\ell^i$ and we claim that $x = \sum_{i = 1}^n x\up{i}$ is a desired element.
    To see this, note that $x_i \preceq_{\F_i}^1 x\up{i}$ from the construction of $x\up{i}$.
    Also, $x\up{i} \le x_i$, and for $j \ne i$, we have $x\up{j} \preceq_{\F_i}^1 x_i$ since $t_\ell^{ji} (\prod_{j' \ne j} s_\ell^{jj'}) x_\ell^j \le x_i$, hence $x_i \preceq_{\F_i}^1 x \preceq_{\F_i}^1 x_i$, as desired.
\end{proof}

The next theorem should remind you of the fact that $\Spec R$ is a sheaf.

\begin{theorem} \label{thm: gluing axioms for localization}
    Let $M$ be a shrinkable $Q$-module and $\F_1, \ldots, \F_n \sub \mF(Q)$.
    Let $\F = \prod_{k = 1}^n \F_k \in \mF(Q)$ and consider map (between $Q$-premodules) $\vp : M_\F \to \prod_{k = 1}^n M_{\F_k}$.
    Then
    \begin{enum}
        \item $\vp$ is injective.
        \item If $\F_i + \F_j$ is 1-step over $M$ for all $1 \le i < j \le n$, then for $x_1, \ldots, x_n \in M$, we have $(\={x_1}, \ldots, \={x_n}) \in \im \vp$ if and only if the image of $\={x_i}$ under map $M_{\F_i} \to M_{\F_i + \F_j}$ and the image of $\={x_j}$ under map $M_{\F_j} \to M_{\F_i + \F_j}$ agree for all $1 \le i < j \le n$.
    \end{enum}
\end{theorem}

\begin{proof}
    Part (a): it suffices to show for $a, b \in M$, we have $\=a = \=b$ in all $M_{\F_k}$ implies $\=a = \=b$ in $M_\F$.
    We have $a \preceq_{\F_k}^{m_k} b$ for some $m_k \in \NN$, so by \Cref{lem: Filter Merge} (and easy induction) we see $a \preceq_\F^m b$, where $m = (\sum_{k = 1}^n m_k) - n + 1$.

    Thus, $a \preceq_\F b$, and similarly $b \preceq_\F a$, hence $\=a = \=b$ in $M_\F$, as desired.

    \Skip Part (b): this follows directly from \Cref{lem: glue together finite mfilters}.
\end{proof}

\begin{example} [$1$-stepness is needed]
    Let $Q$ be the quantale of open sets in $[0, 1]$.
    Let $a = [0, 1 / 2)$, then $\F_a\F_{\perp a} = \br{1}, Q_{\F_a+\F_{\perp a}} = \br{\ast}$.
    However, $Q \to Q_{\F_a} \times Q_{\F_{\perp a}}$ is not an isomorphism.
    To see this, just note that $(\=\varnothing, \={[0, 1]})$ does not lie in the image.
    The reason this happens is that $\F_a \F_{\perp a}$ is not 1-step relative to $Q$. In fact, it is 2-step.
\end{example}

As an application of \Cref{thm: gluing axioms for localization}, we will show that two m-filters $\F$ and $\G$ are both localizable provided that $\F \G$ is localizable and $\F + \G$ is 1-step.

\begin{theorem} \label{thm: FG localizable F+G 1-step imply F and G localizable}
    Let $M$ be shrinkable and $\F, \G \in \mF(Q)$ such that $\F + \G$ is 1-step over $M$.
    Then
    \begin{enum}
        \item If there exists $\H \in \mF(Q)$ localizable over $M$ such that $\F \G \sub \H \sub \F$, then $\F$ is localizable over $M$.
        \item If $s m = m$ for all $s \in \F \G$ and $m \in M$, then both $\F$ and $\G$ are 1-step over $M$.
    \end{enum}
\end{theorem}

\begin{proof}
    Pick any $y \in M$.
    Let $z = \sum \br{z' \in M : z' \preceq_\F y}$.
    Then we claim the image of $z$ under map $M \to M_{\F + \G}$ is the image of $\=y$ under map $M_\F \to M_{\F + \G}$.
    This is because we have commutative diagram
    
    \[\begin{tikzcd}
    	M \\
    	{M_\F} & {M_{\F + \G}}
    	\arrow[from=1-1, to=2-1]
    	\arrow[from=1-1, to=2-2]
    	\arrow[from=2-1, to=2-2]
    \end{tikzcd}\]
    thus, the image of $z$ under map $M \to M_{\F + \G}$ is $\={\sum \br{z' \in M : z' \preceq_\F y}} = \sum \br{\={z'} : z' \preceq_\F y}\\ \le \sum \br{\={z'} : z' \preceq_{\F + \G} y} = \=y$ since $\F + \G$ is 1-step (in particular, localizable), as desired (the other direction is trivial).

    Now from commutative diagram
    
    \[\begin{tikzcd}
    	M & {M_\G} \\
    	{M_\F} & {M_{\F + \G}}
    	\arrow[from=1-1, to=1-2]
    	\arrow[from=1-1, to=2-1]
    	\arrow[from=1-2, to=2-2]
    	\arrow[from=2-1, to=2-2]
    \end{tikzcd}\]
    we see the image of $\=y$ under map $M_\F \to M_{\F + \G}$ agree with the image of $\=z$ under map $M_\G \to M_{\F + \G}$.
    Thus, there exists $w \in M$ such that the image of $\=w$ under map $M_{\F \G} \to M_\F$ is $\=y$ and the image of $\=w$ under map $M_{\F \G} \to M_\G$ is $\=z$ by \Cref{lem: glue together finite mfilters}, and $w \preceq_\F^1 y$.

    Now pick $x \in M$ such that $x \preceq_\F y$.
    Then $x \le z$, so the image of $x$ under map $M \to M_\G$ is less or equal to the image of $z$ under map $M \to M_\G$.
    Also, the image of $x$ under map $M \to M_\F$ is less or equal to the image of $y$ under map $M \to M_\F$.
    Thus, the image of $x + w$ and $w$ under both maps $M \to M_\F$ and $M \to M_\G$ agree, showing $\=x \le \=w$ in $M_{\F \G}$ by \Cref{thm: gluing axioms for localization}.

    Thus, using the notation as in \Cref{prop: chan of localizable} we have $x \preceq_{\F \G} w \preceq_\F^1 y$, which implies $x \preceq_\H w \preceq_\F^1 y$, so $x \preceq_\H^{n_w^\H} w \preceq_\F^1 y$, hence $x \preceq_\F^{n_w^\H} w \preceq_\F^1 y$.
    As a result, we see we can take $n_y^\F = n_w^\H + 1$, hence $\F$ is localizable over $M$.

    Now if $s m = m$ for all $s \in \F \G$ and $m \in M$, then $x \preceq_{\F \G} w$ means $x \le w$, so $x \preceq_\F^1 y$, showing $\F$ is 1-step over $M$ (the statement that $\G$ is 1-step over $M$ is completely symmetric to that of $\F$).
\end{proof}

\bigskip

\section{Useful Structures on Quantales, Modules, and Multiplicative Filters}
\label{sec: useful structures}

In this section $Q$ is a quantale.

\smallskip

\subsection{Precoherence}

Let us recall the definition of precoherent quantales from \cite{K90}.

\begin{definition}
    For a complete semilattice $L$ and $c \in L$, we say $c$ is a \textit{compact element} if $c \le \sum_{i \in I} x_i$ implies $c \le^* \sum_{i \in I} x_i$ for all $x_i \in L$.
    The collection of compact elements in $L$ form a subset called $K(L)$.
\end{definition}

\begin{example}
    We have $1 \in K(L)$ if and only if $L$ is compact.
\end{example}

\begin{lemma} \label{lem: addition of compact is compact}
    Let $L$ be a complete semilattice and $a, b \in K(L)$.
    Then $a + b \in K(L)$.
\end{lemma}

\begin{proof}
    Suppose $a + b \le \sum_{i \le I} c_i$, then $a \le \sum_{i \in I} c_i$ and $b \le \sum_{i \in I} c_i$. Thus, $a \le^* \sum_{i \in I} c_i$ and $b \le^* \sum_{i \in I} c_i$.
    By taking the union of the finite sets, we get $a + b \le^* \sum_{i \in I} c_i$, thus $a + b$ is compact. 
\end{proof}

\begin{definition}
    A complete semilattice $L$ is \textit{algebraic} if it is compactly generated, i.e., we have $x = \sum \br{c \in K(L) : c \le x}$ for all $x \in L$.
\end{definition}

\begin{definition}
    A quantale $Q$ is \textit{precoherent} if
    \begin{enum}
        \item $Q$ is algebraic.
        \item $K(Q)$ is closed under multiplication, i.e., $a b \in K(Q)$ for all $a, b \in K(Q)$.
    \end{enum}
\end{definition}

\begin{definition}
    A quantale $Q$ is coherent if $Q$ is both precoherent and compact.
\end{definition}

We can generalize precoherence to $Q$-modules.

\begin{definition}
    Let $Q$ be a precoherent quantale.
    A $Q$-module $M$ is \textit{precoherent} if
    \begin{enum}
        \item $M$ is algebraic.
        \item $c m \in K(M)$ for all $c \in K(Q)$ and $m \in K(M)$.
    \end{enum}
\end{definition}

As the next proposition shows, we do not have to verify multiplicative closedness for all elements in $K(Q)$ in order to show $Q$ is precoherent.

\begin{prop} \label{prop: selective base for precoh}
    Suppose $A \sub K(Q)$ satisfies:
    \begin{enum}
        \item $x = \sum \br{a \in A : a \le x}$ for all $x \in Q$, and
        \item $a a' \in A$ for all $a, a' \in A$,
    \end{enum}
    then $Q$ is precoherent and $c \in K(Q)$ if and only if $c = a_1 + a_2 + \cdots + a_n$ for some $a_i \in A$.
    Let $M$ be a $Q$-module with $B \sub K(M)$ such that
    \begin{enum}
        \item $x = \sum \br{b \in B : b \le x}$ for all $x \in M$, and
        \item $a b \in B$ for all $a \in A, b \in B$,
    \end{enum}
    then $M$ is a precoherent $Q$-module and $m \in K(M)$ if and only if $m = b_1 + \cdots + b_n$ for some $b_i \in B$.
\end{prop}

\begin{proof}
    We first prove the first part.
    First pick $c \in K(Q)$.
    Then $c = \sum_{i \in I} a_i$ for some $a_i \in A$.
    As $c \in K(Q)$, we have $c =^* \sum_{i \in I} a_i$, so $c = \sum_{i = 1}^n a_i$ for some $a_i \in A$.
    Conversely, by \Cref{lem: addition of compact is compact} we see $\sum_{i = 1}^n a_i \in K(Q)$ provided that $a_i \in K(Q)$.

    To show $Q$ is precoherent, it suffices to show $K(Q)$ is closed under multiplication.
    However, for $c, c' \in K(Q)$, we have $c = \sum_{i = 1}^n a_i$ and $c' = \sum_{i = 1}^m a_i'$ for some $a_i, a_i' \in A$, so \\$c c' = \sum \br{a_i a_j' : 1 \le i \le n, 1 \le j \le m} \in K(Q)$ by \Cref{lem: addition of compact is compact}.

    We now prove the second part.
    By the same reason we have $m \in K(M)$ if and only if $m$ is the summation of finitely many elements in $B$.
    Now for $c \in K(Q), m \in K(M)$, we have $c = \sum_{i = 1}^n a_i$ and $m = \sum_{j = 1}^m b_j$ for some $a_i \in A, b_j \in B$.
    Then $c m = \sum \br{a_i b_j : 1 \le i \le n, 1 \le j \le m} \in K(M)$ by \Cref{lem: addition of compact is compact}.
\end{proof}

\begin{example}
    Let $R$ be a ring and $N$ an $R$-module.
    Then $\Id(R)$ is coherent since in \Cref{prop: selective base for precoh} we can take $A$ to be the set of principal ideals, and $\Sub_R(N)$ is a precoherent $\Id(R)$-module since in \Cref{prop: selective base for precoh} we can take $B$ to be $\br{R n : n \in N}$.
\end{example}

\begin{example}
    $\mF(Q)$ is a precoherent quantale since $\F_q \in K(\mF(Q))$ for all $q \in Q$, so we can take $A = \br{\F_q : q \in Q}$ in \Cref{prop: selective base for precoh} (note that $\F_a \F_b = \F_{a + b}$ by \Cref{prop: Ff cap Fg and Ff Fg}).
    When $Q$ has a minimal element $0$, we have $Q = \F_0 \in \mF(Q)$ is compact, so $\mF(Q)$ is coherent.
\end{example}

\begin{example}
    $\Sigma Q$ is a coherent quantale since $\br{q} \in K(\Sigma Q)$ for all $q \in Q$ (and $1 = \br{1}$), so we can take $A = \br{\br{q} : q \in Q}$ in \Cref{prop: selective base for precoh} (note that $\br{a} \cdot \br{b} = \br{a b}$).
\end{example}

\smallskip

\subsection{Blooming}

Let us recall the definition of left adjoint for a order-preserving map.

\begin{definition}
    Let $A, B$ be posets and $L : A \to B, R : B \to A$ be order-preserving maps.
    Then we say $L$ is \textit{left adjoint} to $R$ (or $R$ is \textit{right adjoint} to $L$) if for all $a \in A, b \in B$, we have $L(a) \le b$ if and only if $a \le R(b)$.
\end{definition}

We record the following theorem from \cite{KL17}.

\begin{theorem} \cite[Proposition~1.8]{KL17}
    Left adjoint preserves all joins and right adjoint preserves all meets.
\end{theorem}

\begin{definition}
    Let $A, B$ be posets with order-preserving map $f : B \to A$.
    If $f$ admits a left adjoint, then we will denote its left adjoint by $f^\flat$.
\end{definition}

\begin{lemma} \label{lem: left adjoint when surjective}
    Let $A, B$ be posets and $f : B \to A$ be a surjective order-preserving map.
    The following are equivalent:
    \begin{enum}
        \item $f$ admits a left adjoint.
        \item $\min f\inv(A_{\ge a})$ exists for all $a \in A$ and $f(\min f\inv(A_{\ge a})) = a$.
        \item $f$ is shrinkable and $\min f\inv(a)$ exists for all $a \in A$.
    \end{enum}
    Moreover, if any of these condition is met, then $f^\flat(a) = \min f\inv(a)$.
\end{lemma}

\begin{proof}
    (a) $\Rightarrow$ (b): Consider the adjoint $f^\flat$, with $f^\flat(a) \le b$ if and only if $a \le f(b)$. Thus, we have the minimal element inside $f\inv(A_{\ge a})$ is exactly $f^\flat(a)$, and it would map to $a$ since $f$ is surjective. 

    (b) $\Rightarrow$ (c): Given $f(\min f\inv(A_{\ge a})) = a$, $\min f\inv(A_{\ge a})$ is also the minimal in $f\inv(a)$. Moreover, since for all $a' \le f(b)$ we have $b' = \min f\inv(a)$, $b' \le b$, we would have $f$ shrinkable. 

    (c) $\Rightarrow$ (a): We define $f^\flat$ as $\min f\inv(a)$, thus by the shrinkable property we know for $a' \le f(b)$ we have $f^\flat(a') \le b$, and by taking $b = f^\flat(a)$ we have the order preserving property. Moreover, for $f^\flat(a) \le b$ we have $a \le f(b)$ given that $f\circ f^\flat$ is identity, thus $f^\flat$ is indeed a left adjoint. 
\end{proof}

\begin{definition} \cite[Page 135]{JA12}
    A complete semilattice $L$ is said to be \textit{continuous} if $\sigma_L : \Sigma L \to L$ has a left adjoint.
\end{definition}

When $M$ is continuous, we are able to generalize \Cref{lem: glue together finite mfilters} to arbitrarily many m-filters.

\begin{lemma} \label{lem: arbitrary merge}
    Let $M$ be a continuous $Q$-module.
    For $a, b \in M, \F_i \in \mF(Q)$, if $a \preceq_{\F_i}^1 b$ for all $i \in I$, then $a \preceq_\F^1 b$, where $\F = \bigcap_{i \in I} \F_i$.
\end{lemma}

\begin{proof}
    We have (after possibly merge indices) $a = \sum_{j \in J} a_j^i$ with $s_j^i a_j^i \le b$ for some $s_j^i \in \F_i$.
    Let $\sigma_M^\flat(a) = \br{a_k}_{k \in K}$, then for all $i \in I, k \in K$, there exists finite $J_k^i \sub J$ such that $a_k \le \sum_{j \in J_k^i} a_j^i$.
    Now let $s_k^i = \prod_{j \in J_k^i} s_j^i$, then $s_k^i a_k \le b$ and $s_k^i \in \F_i$.
    As a result, we see $(\sum_{i \in I} s_k^i) a_k \le b$ and $\sum_{i \in I} s_k^i \in \F$, as desired.
\end{proof}

\begin{cor} \label{prop: arbi inter 1-step is 1-step}
    Let $M$ be a continuous $Q$-module.
    For $\F_i \in \mF(Q)$, if each $\F_i$ is 1-step over $M$, then $\bigcap_{i \in I} \F_i$ is 1-step over $M$.
\end{cor}

\begin{proof}
    Let $\F = \bigcap_{i \in I} \F_i$.
    Suppose $a \preceq_\F b$ for some $a, b \in M$.
    Then $a \preceq_{\F_i} b$ for all $i \in I$, so $a \preceq_{\F_i}^1 b$ because $\F_i$ is 1-step over $M$.
    Thus, $a \preceq_\F^1 b$ by \Cref{lem: arbitrary merge}, so $\F$ is 1-step over $M$, as desired.
\end{proof}

\begin{theorem} \label{thm: arbitrary merge}
    Let $M$ be a continuous $Q$-module.
    For $\F_i \in \mF(Q)$ and $\F = \bigcap_{i \in I} \F_i$, if each $\F_i$ is 1-step over $M$, then we have injective $Q$-linear map
    \[
    M_\F \to \prod_{i \in I} M_{\F_i}.
    \]
\end{theorem}

\begin{proof}
    This is a direct consequence of \Cref{lem: arbitrary merge}.
\end{proof}

We are now able to define what is a blooming quantale/module.

\begin{definition}
    We say a quantale $Q$ is \emph{blooming} if $\sigma_Q$ has a left adjoint and the following diagram commutes:
    \[\begin{tikzcd}
        {Q \times Q} & Q \\
        {\Sigma Q \times \Sigma Q} & {\Sigma Q}
        \arrow["\cdot", from=1-1, to=1-2]
        \arrow["{\sigma_Q^\flat \times \sigma_Q^\flat}"', from=1-1, to=2-1]
        \arrow["{\sigma_Q^\flat}", from=1-2, to=2-2]
        \arrow["\cdot", from=2-1, to=2-2]
    \end{tikzcd}\]
\end{definition}

\begin{definition}
    Let $Q$ be a blooming quantale and $M$ a $Q$-module.
    Then $M$ is \emph{blooming} if $\sigma_M$ has a left adjoint and the following diagram commutes:
    \[\begin{tikzcd}
        {Q \times M} & M \\
        {\Sigma Q \times \Sigma M} & {\Sigma M}
        \arrow["\cdot", from=1-1, to=1-2]
        \arrow["{\sigma_Q^\flat \times \sigma_M^\flat}"', from=1-1, to=2-1]
        \arrow["{\sigma_M^\flat}", from=1-2, to=2-2]
        \arrow["\cdot", from=2-1, to=2-2]
    \end{tikzcd}\]
\end{definition}

\begin{remark}
    Even if $Q$ is blooming, $\sigma_Q^\flat$ may fail to be a quantale homomorphism since it may not preserve 1.
\end{remark}

We now show that blooming is a more general notion than precoherence.

\begin{lemma} \label{lem: algebraic admits left adjoint}
    Let $L$ be an algebraic complete semilattice.
    Then $\sigma_L$ admits a left adjoint $\sigma_L^\flat : L \to \Sigma L, x \mapsto \br{c \in K(L) : c \le x}$.
\end{lemma}

\begin{proof}
    For $x \in L$ let $C = \br{c \in K(L) : c \le x}$.
    Then $\sigma_L(C) = x$ since $L$ is algebraic.
    Also, for all $D \in \Sigma L, c \in C$ such that $\sigma_L(D) \ge x$, we have $c \le^* \sum D$ by the definition of $K(L)$, showing $C \le D$ in $\Sigma L$.
    Thus, we conclude the proof by \Cref{lem: left adjoint when surjective}.
\end{proof}

\begin{prop} \label{prop: precoh implies blooming}
    Let $Q$ be a precoherent quantale and $M$ a precoherent $Q$-module.
    Then both $Q$ and $M$ are blooming.
\end{prop}

\begin{proof}
    We know $\sigma$ has left adjoint $\sigma^\flat : x \mapsto \br{y : y \text{ compact, } y \le x}$ by \Cref{lem: algebraic admits left adjoint}.
    For the blooming condition of $Q$, by the precoherent property, for $y_1 \in \sigma^\flat(x_1)$, $y_2 \in \sigma^\flat(x_2)$ we have $y_1 y_2 = \sum_{i \in I} y_i'$ for finite set of compact element.
    Thus, we would have $\sigma^\flat(x_1) \sigma^\flat(x_2) \le \sigma^\flat(x_1x_2)$, and they would be the same by minimality.
    Hence $Q$ is blooming. 

    For the blooming condition of $M$, the procedure is exactly the same, but with $x_1 \in Q, x_2 \in M$.
    This would give the blooming property of $M$. 
\end{proof}

Thus, we have implications

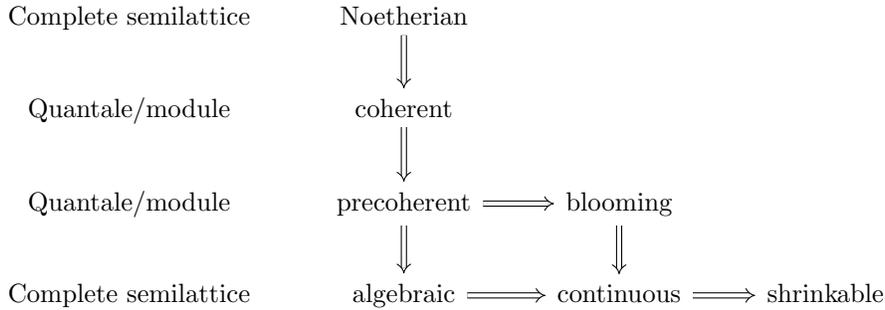
\[\begin{tikzcd}
	{\text{Complete semilattice}} & {\text{Noetherian}} \\
	{\text{Quantale/module}} & {\text{coherent}} \\
	{\text{Quantale/module}} & {\text{precoherent}} & {\text{blooming}} \\
	{\text{Complete semilattice}} & {\text{algebraic}} & {\text{continuous}} & {\text{shrinkable}}
	\arrow[Rightarrow, from=1-2, to=2-2]
	\arrow[Rightarrow, from=2-2, to=3-2]
	\arrow[Rightarrow, from=3-2, to=3-3]
	\arrow[Rightarrow, from=3-2, to=4-2]
	\arrow[Rightarrow, from=3-3, to=4-3]
	\arrow[Rightarrow, from=4-2, to=4-3]
	\arrow[Rightarrow, from=4-3, to=4-4]
\end{tikzcd}\]
where the texts at the leftmost indicate under which setting each terminology can be defined.

From the diagram we see that precoherent implies both algebraic and blooming.
It turns out that the converse is also true.

\begin{prop} \label{prop: quasi coherent is algebraic and blooming}
    Let $M$ be a $Q$-module.
    Then
    \begin{enum}
        \item $Q$ is precoherent if and only if $Q$ is both algebraic and blooming.
        \item When $Q$ is precoherent, we have $M$ is precoherent if and only if $M$ is both algebraic and blooming.
    \end{enum}
\end{prop}

\begin{proof}
    Precoherent implies algebraic by definition, and precoherent implies blooming from \Cref{prop: precoh implies blooming}. This holds for both quantale or module, with the module case assumed the quantale being precoherant too. 

    Suppose $Q$ is algebraic and blooming. Notice the collection of compact elements at most $x$ is a choice of $\sigma^\flat(x)$, thus we have $\sigma^\flat(x_1)\sigma^\flat(x_2) = \sigma^\flat(x_1x_2)$ for all $x_1, x_2 \in Q$, which is equivalent to $y_1y_2$ being finite sum of compact elements for all compact $y_1, y_2$, which matches the definition of precoherent. 

    Now let $Q$ be precoherent. Suppose $M$ is algebraic and blooming. Similarly, we have $\sigma^\flat(q)\sigma^\flat(x) = \sigma^\flat(qx)$, thus $qx$ is finite sum of compact elements in $M$, for all compact $q \in Q$ and $x \in M$. This matches the definition of precoherent module. 
\end{proof}

\smallskip

\subsection{Solid Multiplicative Filters}

\begin{prop} \label{prop: equivalent def of solid filter}
    For $\F \in \mF(Q)$, the following are equivalent:
    \begin{enum}
        \item For all $S \in \sigma\inv(\F) \sub \Sigma Q$, there exists $t \in \F$ such that $\iota(t) \le S$ in $\Sigma Q$.
        \item For all $x_i \in Q$ such that $\sum_{i \in I} x_i \in \F$, there exists finite subset $I_0 \sub I$ such that $\sum_{i \in I_0} x_i \in \F$.
    \end{enum}
\end{prop}

\begin{proof}
    (a) $\Rightarrow$ (b): let $x = \sum_{i \in I} x_i \in \F$ and $S = \br{x_i}_{i \in I} \in \Sigma Q$.
    Then there exists $x' \in \F$ such that $\br{x'} \le S$ in $\Sigma Q$, so there exists finite $I_0 \sub I$ such that $x' \le \sum_{i \in I_0} x_i$, which shows $\sum_{i \in I_0} x_i \in \F$.

    \Skip (b) $\Rightarrow$ (a): pick any $s \in \F$ and $S \in \sigma\inv(s)$.
    Then there exists finite $S_0 \sub S$ such that $t = \sigma(S_0) \in \F$, so we conclude by the fact that $\br{t} = S_0 \le S$ in $\Sigma Q$.
\end{proof}

\begin{definition}
    If $\F \in \mF(Q)$ satisfies any of the condition in \Cref{prop: equivalent def of solid filter}, then we say $\F$ is \emph{solid}.
\end{definition}

\begin{example}
    Let $p \in Q$ be a prime element.
    Then $\F_{\nmid p}$ is solid by chasing definition.
\end{example}

\begin{prop} \label{prop: compact element control is solid}
    Let $\F \in \mF(Q)$.
    \begin{enum}
        \item If for all $x \in \F$, there exists compact element $y \in \F$ such that $y \le x$, then $\F$ is solid.
        \item If $Q$ is algebraic, then the converse holds.
    \end{enum}
\end{prop}

\begin{proof}
    Part (a): follows directly from definition.

    \Skip Part (b): let $\F$ be solid.
    For all $x \in \F$, we have $x = \sum_{i \in I} a_i$ where all $a_i$'s are compact elements.
    Then by assumption there exists some finite $I_0 \sub I$ such that $y = \sum_{i \in I_0} a_i \in \F$.
    Then $y \le x$ and $y$ is a compact element by \Cref{lem: addition of compact is compact}.
\end{proof}

\begin{prop} \label{prop: solid is quantale}
    Let $Q$ be a blooming quantale.
    Then the collection of solid m-filters in $Q$ form a subquantale in $\mF(Q)$.
\end{prop}

\begin{proof}
    Clearly the identity $Q$ is solid, and the distribution law follows from the quantale structure of $\mF(Q)$. We just need to check the closedness of addition and product. 

    Let $\F, \G$ be solid quantales, then for $S \in \sigma\inv(\F\G)$, we have $S \in \sigma\inv(\F)$ and $S \in \sigma\inv(\G)$, thus there's $t_1 \in \F, t_2 \in \G$ with $\iota(t_1) \le S$, $\iota(t_2) \le S$, thus $\iota(t_1 + t_2) = \iota(t_1) + \iota(t_2) \le S$. 

    Let $\F_i$ be solid quantales for $i \in I$. Suppose $S \in \sigma\inv(\sum_{i \in I} \F_i)$, then $S \in \sigma^{-1}(\sum_{i \in I_0} \F_i)$ for some finite $I_0$, since $\sigma(S)$ is at least finite product of elements in $\F_i$. Thus, we only need to check the case there's two filters, called $\F, \G$. 

    Suppose $\sigma(S) \ge s_1s_2$ for $s_1 \in \F$, $s_2 \in \G$, then we have $t_1, t_2$ such that $\iota(t_1) \le \sigma^\flat(s_1), \iota(t_2) \le \sigma^\flat(s_2)$. Since $S \ge \sigma^\flat(s_1s_2) = \sigma^\flat(s_1)\sigma^\flat(s_2)$ given $Q$ is blooming, we have $S \ge \iota(t_1t_2)$. 
\end{proof}

\begin{theorem} \label{thm: prime contains and avoid}
    Let $\F \in \mF(Q)$ be solid and $q \in Q - \F$.
    Then there exists a prime element $p \in Q - \F$ such that $p \ge q$.
\end{theorem}

\begin{proof}
    We use Zorn's lemma to pick a maximal $p \in Q - \F$ such that $p \ge q$.
    To see this is doable, pick a chain $C \sub Q$ such that $C \cap \F = \varnothing$ and we want to show $\sum C \not\in \F$.
    Assume the contrary, then there exists finite $C_0 \sub C$ such that $\sum C_0 \in \F$ by Proposition~\ref{prop: equivalent def of solid filter}.
    However, there exists a maximal element $c \in C_0$, from which we see $c = \sum C_0 \in \F$, a contradiction.

    Now we show such an element $p$ is prime.
    Assume the contrary, then there exists $a, b > p$ such that $a b \le p$.
    However, then from maximality of $p$ we see $a, b \in \F$, hence $a b \in \F$, so $p \in \F$, a contradiction.
\end{proof}

\smallskip

\subsection{Locally Solid Multiplicative Filters}

We are now going to introduce a local version of solid m-filters.

\begin{definition}
    We say $\F \in \mF(Q)$ is \emph{locally solid} if there exists $W \in \sigma\inv(1)$ such that for all $w \in W$ and $S \in \sigma\inv(\F)$, we have $\iota(t w) \le S$ (in $\Sigma Q$) for some $t \in \F$.
\end{definition}

\begin{lemma} \label{lem: locally solid lower witness}
    For locally solid $\F \in \mF(Q)$, if $W \in \sigma\inv(1)$ meets the condition in the definition of locally solid and $W' \le W$ in $\Sigma Q$ (such that $\sigma W' = 1$), then $W'$ also meets the condition in the definition of locally solid.
\end{lemma}

\begin{proof}
    For each $w' \in W'$, we have $w' \le w_1 + \cdots + w_n$ for some $w_i \in W$.
    Pick any $S \in \sigma\inv(\F)$, then $\iota(t_i w_i) \le S$ for some $t_i \in \F$, meaning $t_i w_i \le \sum S_i$ for some finite $S_i \sub S$.
    Thus, if we take $t = \prod_{i = 1}^n t_i$, then $t w' \le \sum_{i = 1}^n (\prod_{j \ne i} t_j) t_i w_i \le \sum_{i = 1}^n t_i w_i \le \sum \bigcup_{i = 1}^n S_i$, whence $\iota(t w') \le S$ in $\Sigma Q$, as desired.
\end{proof}

Thus, when $Q$ is continuous, in the definition of locally solid it is safe to take $W = \sigma_Q^\flat(1)$.
Obviously a solid m-filter is locally solid, and the next corollary says the converse holds for compact quantale.

\begin{cor} \label{cor: compact locally solid is solid}
    Let $Q$ be a compact quantale.
    Then a locally solid m-filter of $Q$ is solid.
\end{cor}

\begin{proof}
    Let $\F \in \mF(Q)$ be locally solid and $W \in \sigma\inv(1)$ witnesses the locally solidity of $\F$.
    Then $\br{1} \le W$ since $Q$ is compact, hence by \Cref{lem: locally solid lower witness} we can take $W = \br{1}$.
    However, then the condition for locally solid becomes solid.
\end{proof}

\begin{prop} \label{prop: 1 is locally solid iff min sigma inv 1 exists}
    We have $\br{1} \in \mF(Q)$ is locally solid if and only if $\min \sigma\inv(1)$ exists in $\Sigma Q$.
\end{prop}

\begin{proof}
    $\Rightarrow$: let $W$ witnesses the locally solidity of $\br{1} \in \mF(Q)$.
    We claim $W = \min \sigma\inv(1)$, which finishes the proof.
    To see this, pick any $W' \in \sigma\inv(1)$, then by the definition of locally solid we see $w \le^* \sum W'$ for all $w \in W$, showing $W \le W'$ in $\Sigma Q$.

    \Skip $\Leftarrow$: let $W = \min \sigma\inv(1)$.
    Then for all $\sigma W' = 1$, we have $W \le W'$ in $\Sigma Q$, so $w \le^* \sum W'$ for all $w \in W$, showing $W$ witnesses the locally solidity of $\br{1} \in \mF(Q)$.
\end{proof}

\begin{example}
    Let $X$ be a locally compact topological space, i.e., every point has a compact neighborhood.
    Then $\br{1} \in \mF(\O(X))$ is locally solid.
    This is because every point $x \in X$ admits open $U_x \ni x$ and compact $C_x \supset U_x$.
    Now $X = \bigcup_{x \in X} U_x$, and if $X = \bigcup_{i \in I} V_i$, then for each $x \in X$ we have $C_x \sub \bigcup_{i \in I} V_i$, so $C_x \sub \bigcup_{i \in I_x} V_i$ for some finite $I_x \sub I$.
    Thus, $\br{U_x}_{x \in X} \le \br{V_i}_{i \in I}$ in $\Sigma \O(X)$, as desired.
\end{example}

Furthermore, in the setting of regular space $X$ being locally compact is the same as $\br{1} \in \mF(\O(1))$ being locally solid, in which case $\O(X)$ is blooming.

\begin{theorem} \label{thm: locally compact hausdorff is blooming}
    Let $X$ be a regular topological space.
    The following are equivalent:
    \begin{enum}
        \item $X$ is locally compact.
        \item $\O(X)$ is blooming.
        \item $\br{1} \in \mF(\O(X))$ is locally solid.
    \end{enum}
\end{theorem}

\begin{proof}
    (a) $\Rightarrow$ (b): pick any $U \in \O(X)$.
    For all $x \in U$, there exists open neighborhood $V_x$ and closed compact neighborhood $C_x$ such that $x \in V_x \sub C_x \sub U$.
    Then we claim $\sigma^\flat(U) = \sum_{x \in U} V_x$.
    To this, let $U = \sum_{i \in I} U_i$, then for each $x \in U$, we have $C_x \sub \bigcup_{i \in I} U_i$, so $C_x \sub \bigcup_{i \in I_x} U_i$ for some finite $I_x \sub I$.
    Thus, $V_x \le \sum_{i \in I_x} U_i$, showing $V_x \le^* \sum_{i \in I} U_i$, as desired.

    We now argue that for $U^1, U^2 \in \O(X)$, we have $\sigma^\flat(U^1) \sigma^\flat(U^2) = \sigma^\flat(U^1 \cap U^2)$.
    For $k = 1, 2$, pick $V_x^k, C_x^k$ like above (corresponding to $U^k$).
    Then we claim $\sigma^\flat(U^1 \cap U^2) = \sum_{x \in U^1, y \in U^2} (V_x^1 \cap V_y^2) = \sigma^\flat(U^1) \sigma^\flat(U^2)$.

    To see this, first note that $V_x^1 \cap V_y^2 \sub U^1 \cap U^2$ for $x \in U^1, y \in U^2$, and for all $x \in U^1 \cap U^2$, we would have $x \in V_x^1 \cap V_x^2$, hence $U^1 \cap U^2 = \bigcup_{x \in U^1, y \in U^2} (V_x^1 \cap V_y^2)$.
    Now take any $U^1 \cap U^2 = \sum_{i \in I} U_i$.
    Then for each $x \in U^1, y \in U^2$, if $V_x^1 \cap V_y^2 = \varnothing$, then obviously $V_x^1 \cap V_y^2 \le^* \sum_{i \in I} U_i$.
    Thus, assume $V_x^1 \cap V_y^2 \ne \varnothing$, then $V_x^1 \cap V_y^2 \sub C_x^1 \cap C_y^2 \sub U^1 \cap U^2$, and $C_x^1 \cap C_y^2$ is compact since it is a closed subset in $C_x^1$.
    Thus, there exists finite $I_0 \sub I$ such that $C_x^1 \cap C_y^2 \sub \bigcup_{i \in I_0} U_i$, meaning $V_x^1 \cap V_y^2 \le^* \sum_{i \in I} U_i$, as desired.

    \Skip (b) $\Rightarrow$ (c): this is a direct consequence of \Cref{prop: 1 is locally solid iff min sigma inv 1 exists}.

    \Skip (c) $\Rightarrow$ (a): by \Cref{prop: 1 is locally solid iff min sigma inv 1 exists} we see $W = \min \sigma\inv(1)$ exists in $\Sigma \O(X)$.
    For each $x \in X$, there exists $U \in W$ such that $x \in U$.
    Then because $X$ is regular we see there exists open neighborhood $V$ and closed neighborhood $C$ such that $x \in V \sub C \sub U$ (since there are two open sets separating $\br{x}$ and $X - U$).
    We claim that $C$ is compact, so we are done.

    Pick any open cover $C \sub \bigcup_{i \in I} U_i$, then $X = (X - C) \cup \bigcup_{i \in I} U_i$, so by the choice of $W$ we see $U \sub (X - C) \cup \bigcup_{i \in I_0} U_i$ for some finite $I_0 \sub I$.
    Then we have $C \sub \bigcup_{i \in I_0} U_i$, showing $C$ is compact, as desired.
\end{proof}

The Spec like \Cref{thm: gluing axioms for localization} only works for finitely many m-filters.
Now under some additional assumptions we can generalize (the first part of) it to countably many m-filters.

\begin{lemma} \label{lem: countable filter merge}
    Let $Q$ be a quantale with m-filters $\br{\F_k}_{k = 1}^\oo$ such that $\F = \bigcap_{k = 1}^\oo \F_k$ is locally solid.
    Let $M$ be a shrinkable $Q$-module.
    For $a, b \in M$, if $a \preceq_{\F_k}^1 b$ for all $k \in \NN$, then $a \preceq_\F^1 b$.
\end{lemma}

\begin{proof}
    We have $a = \sum_{i_1 \in I_1} a_1(i_1)$ with $s_1(i_1) a_1(i_1) \le b$ for some $a_1(i_1) \in M$ and $s_1(i_1) \in \F_1$.
    Now $a_1(i_1) \le a \preceq_{\F_2}^1 b$ for each $i_1 \in I_1$, so (after possibly merge indices) $a_1(i_1) = \sum_{i_2 \in I_2} a_2(i_1, i_2)$ with $s_2(i_1, i_2) a_2(i_1, i_2) \le b$ for some $s_2(i_1, i_2) \in \F_2$.
    Repeat this process, for each $k \in \NN$ we get $a_k(i_1, \ldots, i_k) \in M, s_k(i_1, \ldots, i_k) \in \F$, and index set $I_k$ such that $a_{k - 1}(i_1, \ldots, i_{k - 1}) = \sum_{i_k \in I_k} a_k(i_1, \ldots, i_k)$ and $s_k(i_1, \ldots, i_k) a_k(i_1, \ldots, i_k) \le b$ for all $k \in \NN$ (where $a_0 = a$).

    Let $W \in \sigma\inv(1)$ witnesses the locally solidity of $\F$.
    It suffices to show that $w a \preceq_\F^1 b$ for each $w \in W$ by \Cref{lem: 1 step local preorder closed under addition}.
    Let $I = \prod_{k = 1}^\oo I_k$.
    For each $i = (i_1, i_2, \ldots) \in I$ and $k \in \NN$, let $i_{\le k} = (i_1, \ldots, i_k) \in \prod_{t = 1}^k I_t$.
    Also, for $j_1 \in I_1, \ldots, j_k \in I_k$, we define $I(j_1, \ldots, j_k) = \br{i \in I : i_1 = j_1, \ldots, i_k = j_k} \sub I$.
    For each $i \in I$, we have $\sum_{k = 1}^\oo s_k(i_{\le k}) \in \F$, so because $\F$ is locally solid we see there exists $m_i \in \NN$ and $t_i \in \F$ such that $t_i w \le \sum_{k = 1}^{m_i} s_k(i_{\le k})$.
    We take $m_i$ to be minimal such positive integer.

    Our key claim is that $a = \sum_{i \in I} a_{m_i}(i_{\le m_i})$.
    Obviously $\sum_{i \in I} a_{m_i}(i_{\le m_i}) \le a$, so we now show $a \le \sum_{i \in I} a_{m_i}(i_{\le m_i})$.
    Assume the contrary, $a \not\le \sum_{i \in I} a_{m_i}(i_{\le m_i})$.
    Then we have $\sum_{i \in I} a_{m_i}(i_{\le m_i}) = \sum_{j_1 \in I_1} \sum_{i \in I(j_1)} a_{m_i}(i_{\le m_i})$.
    If $a_1(j_1) \le \sum_{i \in I(j_1)} a_{m_i}(i_{\le m_i})$ for all $j_1 \in I_1$, then $\sum_{i \in I} a_{m_i}(i_{\le m_i}) \ge \sum_{j_1 \in I_1} a_1(j_1) = a$, a contradiction.
    Thus, there exists $j_1 \in I_1$ such that $a_1(j_1) \not\le \sum_{i \in I(j_1)} a_{m_i}(i_{\le m_i})$.

    We now have $\sum_{i \in I(j_1)} a_{m_i}(i_{\le m_i}) = \sum_{j_2 \in I_2} \sum_{i \in I(j_1, j_2)} a_{m_i}(i_{\le m_i})$.
    If $a_2(j_1, j_2) \le \sum_{i \in I(j_1, j_2)} a_{m_i}(i_{\le m_i})$ for all $j_2 \in I_2$, then $\sum_{i \in I(j_1)} a_{m_i}(i_{\le m_i}) \ge \sum_{j_2 \in I_2} a_2(j_1, j_2) = a_1(j_1)$, a contradiction.
    Thus, there exists $j_2 \in I_2$ such that $a_2(j_1, j_2) \not\le \sum_{i \in I(j_1, j_2)} a_{m_i}(i_{\le m_i})$.

    Repeat this procedure, we get $j \in I$ such that $a_k(j_{\le k}) \not\le \sum_{i \in I(j_1, \ldots, j_k)} a_{m_i}(i_{\le m_i})$ for each $k \in \NN$.
    Now let $m = m_j$, then $a_m(j_{\le m}) \not\le \sum_{i \in I(j_1, \ldots, j_m)} a_{m_i}(i_{\le m_i})$.
    However, for $i \in I(j_1, \ldots, j_m)$, we have $t_j w \le \sum_{k = 1}^m s_k(j_{\le k}) = \sum_{k = 1}^m s_k(i_{\le k})$, so $m_i \le m$ by the minimality of $m_i$, hence $a_{m_i}(i_{\le m_i}) \ge a_m(j_{\le m})$.
    Thus, $a_m(j_{\le m}) \le \sum_{i \in I(j_1, \ldots, j_m)} a_{m_i}(i_{\le m_i})$, a contradiction.

    As a result, we see $a = \sum_{i \in I} a_{m_i}(i_{\le m_i})$, so $w a = \sum_{i \in I} w a_{m_i}(i_{\le m_i})$.
    However, we have $t_i w a_{m_i}(i_{\le m_i}) \le \sum_{k = 1}^{m_i} s_k(i_{\le k}) a_{m_i}(i_{\le m_i}) \le \sum_{k = 1}^{m_i} s_k(i_{\le k}) a_k(i_{\le k}) \le b$, whence $w a \preceq_\F^1 b$, as desired.
\end{proof}

\begin{theorem} \label{thm: countable solid inj}
    Let $Q$ be a quantale with m-filters $\br{\F_k}_{k = 1}^\oo$ such that $\F = \bigcap_{k = 1}^\oo \F_k$ is locally solid.
    Let $M$ be a shrinkable $Q$-module such that all $\F_k$'s are 1-step over $M$.
    Then we have injective $Q$-linear map
    \[
    M_\F \to \prod_{k = 1}^\oo M_{\F_k}.
    \]
\end{theorem}

\begin{proof}
    Follows easily from \Cref{lem: countable filter merge}.
\end{proof}

\bigskip

\section{Normal and Conormal Filters}
\label{sec: normal and conormal}

Throughout this section $Q$ is a quantale and $M$ a $Q$-module.

We have seen we need to assume m-filters are 1-step in order for some statements hold.
However, a problem with 1-stepness is that even if both $\F$ and $\G$ are 1-step m-filters relative to $M$, it is not guaranteed that $\F + \G$ is a 1-step m-filter over $M$ (we will see an example later in this chapter).

This motivates us to define the notion of normal and conormal filters, which will be preserved under arbitrary addition (assuming $M$ is shrinkable).

\smallskip

\subsection{Diagram Language}

Before dive into the details of normal and conormal, we need to introduce some diagram languages in order to make the notions and proofs intuitive.

\begin{definition}
    For $S, T \in \Sigma M$, if $\sigma S \le \sigma T$, then we draw
    
    \[\begin{tikzcd}
    	{S} & {T}
    	\arrow[no head, from=1-1, to=1-2]
    \end{tikzcd}\]
\end{definition}

\begin{definition}
    Let $\F \in \mF(Q)$.
    For $S, T \in \Sigma M$, if for all $x \in S$, we have $\br{s x} \le T$ in $\Sigma M$ (i.e., $s x \le^* \sum T$) for some $s \in \F$, then we draw
    
    \[\begin{tikzcd}
    	{S} \\
    	{T}
    	\arrow["\F"', no head, from=1-1, to=2-1]
    \end{tikzcd}\]
    or simply
    
    \[\begin{tikzcd}
    	S & {\text{or}} & S & {\text{or}} & S \\
    	T && T && T
    	\arrow[no head, from=1-1, to=2-1]
    	\arrow[squiggly, no head, from=1-3, to=2-3]
    	\arrow[dashed, no head, from=1-5, to=2-5]
    \end{tikzcd}\]
    depending on our convention on whether $\F$ corresponds to straight line or curly line or dashed line.
\end{definition}

\begin{remark}
    This also works if $S$ or $T$ is in $M$ because we can embed $M$ into $\Sigma M$ using $\iota$.
\end{remark}

We will use $\sum$ to represent an arbitrary element in $\Sigma M$, so different $\sum$'s might not agree with each other.
Then we see for $x, y \in M$, we have $\=x \le \=y$ in $M_\F$ if (here $\F$ corresponds to straight vertical lines)

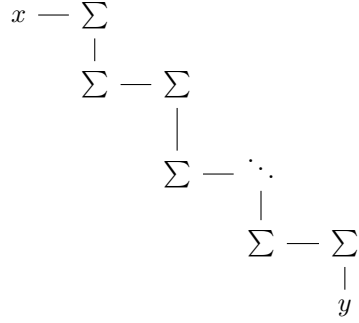
\[\begin{tikzcd} [row sep = small, column sep = small]
	x & \sum \\
	& \sum & \sum \\
	&& \sum & \ddots \\
	&&& \sum & \sum \\
	&&&& y
	\arrow[no head, from=1-1, to=1-2]
	\arrow[no head, from=1-2, to=2-2]
	\arrow[no head, from=2-2, to=2-3]
	\arrow[no head, from=2-3, to=3-3]
	\arrow[no head, from=3-3, to=3-4]
	\arrow[no head, from=3-4, to=4-4]
	\arrow[no head, from=4-4, to=4-5]
	\arrow[no head, from=4-5, to=5-5]
\end{tikzcd}\]
or, for simplicity, we omit the $\sum$'s and write

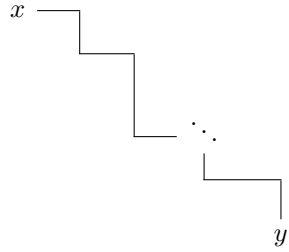
\[\begin{tikzcd} [row sep = small, column sep = small]
	x & {} \\
	& {} & {} \\
	&& {} & \ddots \\
	&&& {} & {} \\
	&&&& y
	\arrow[shorten <=0pt, shorten >=-4.7pt, no head, from=1-1, to=1-2]
	\arrow[shorten <=-6.1pt, shorten >=-1.5pt, no head, from=1-2, to=2-2]
	\arrow[shorten <=-4.7pt, shorten >=-4.7pt, no head, from=2-2, to=2-3]
	\arrow[shorten <=-6.1pt, shorten >=-1.5pt, no head, from=2-3, to=3-3]
	\arrow[shorten <=-4.7pt, shorten >=0pt, no head, from=3-3, to=3-4]
	\arrow[shorten <=0pt, shorten >=-1.5pt, no head, from=3-4, to=4-4]
	\arrow[shorten <=-4.7pt, shorten >=-4.7pt, no head, from=4-4, to=4-5]
	\arrow[shorten <=-6.1pt, shorten >=0pt, no head, from=4-5, to=5-5]
\end{tikzcd}\]

Now let us make some observations about when we can simplify a diagram.

\begin{prop} \label{prop: diagram merge consecutive lines}
    For $S, T, W \in \Sigma M, \F, \G \in \mF(Q)$, we have diagram implications:
    
    \[\begin{tikzcd}
    	S & W & T & \rightsquigarrow & S && T \\
    	&& S && S \\
    	&& W & \rightsquigarrow \\
    	&& T && T \\
    	&& S && S \\
    	&& W & \rightsquigarrow & {W'} \\
    	&& T && T
    	\arrow[no head, from=1-1, to=1-2]
    	\arrow[no head, from=1-2, to=1-3]
    	\arrow[no head, from=1-5, to=1-7]
    	\arrow["\F"', no head, from=2-3, to=3-3]
    	\arrow["\F"', no head, from=2-5, to=4-5]
    	\arrow["\F"', no head, from=3-3, to=4-3]
    	\arrow["\F"', no head, from=5-3, to=6-3]
    	\arrow["\G"', squiggly, no head, from=5-5, to=6-5]
    	\arrow["\G"', squiggly, no head, from=6-3, to=7-3]
    	\arrow["\F"', no head, from=6-5, to=7-5]
    \end{tikzcd}\]
    where the last diagram means the left hand side admits some $W' \in \Sigma M$ such that the right hand side holds.
\end{prop}

\begin{proof}
    The first diagram: follows from the fact that $\sigma S \le \sigma W \le \sigma T$.

    \Skip The second diagram: for each $s \in S$, there exists $w_1, \ldots, w_n \in W$ and $u \in \F$ such that $u s \le \sum_{i = 1}^n w_i$.
    Now for each $i$, there exists (after possibly enlarge index set) $t_{i1}, t_{i2}, \ldots, t_{im} \in T$ and $v_i \in \F$ such that $v_i w_i \le \sum_{j = 1}^m t_{ij}$.
    Thus, we have $(\prod_{i = 1}^m v_i) u s \le (\prod_{i = 1}^m v_i) \sum_{i = 1}^n w_i \le \sum \br{t_{ij} : 1 \le i \le n, 1 \le j \le m}$.

    \Skip The third diagram: for each $s \in S$, there exists $w_1^s, \ldots, w_{n_s}^s \in W$ and $u^s \in \F$ such that $u^s s \le \sum_{i = 1}^n w_i^s$.
    Now for each $1 \le i \le n_s$, there exists (after possibly enlarge index set) $t_{i1}^s, t_{i2}^s, \ldots, t_{im_s}^s \in T$ and $v_i^s \in \G$ such that $v_i^s w_i^s \le \sum_{j = 1}^{m_s} t_{ij}^s$.
    Thus, we have $(\prod_{i = 1}^{m_s} v_i^s) u^s s \le \sum \br{t_{ij}^s : 1 \le i \le n_s, 1 \le j \le m_s}$.
    Now we can simply take $W' = \br{(\prod_{i = 1}^{m_s} v_i^s) s : s \in S}$.
\end{proof}

Thanks to \Cref{prop: diagram merge consecutive lines}, when drawing diagrams we can safely omit the $\sum$'s, and we can use

\[\begin{tikzcd}
	S \\
        {} \\
	T
	\arrow["{\F + \G}"', no head, from=1-1, to=3-1]
\end{tikzcd}\]
to represent

\[\begin{tikzcd}
	S \\
	\sum \\
	T
	\arrow["\F"', no head, from=1-1, to=2-1]
	\arrow["\G"', no head, from=2-1, to=3-1]
\end{tikzcd}\]

\smallskip

\subsection{Normal Filters}

\begin{definition}
    Let $\N$ be a m-filter for $Q$ and $M$ a $Q$-qmodule.
    We say $\N$ is \emph{normal} over $M$ if for all $s \in \N, m \in M, \br{m_i}_{i \in I} \sub M$ with $s m \le \sum m_i$, there exist $\br{m_j'}_{j \in J} \sub M, \br{s_j}_{j \in J} \sub \N$ such that $m \le \sum m_j'$ and $s_j m_j' \le^* \sum_{i \in I} m_i$.
\end{definition}

In other words, a normal m-filter $\F$ satisfies the following diagram implication for all $s \in \F$:

\[\begin{tikzcd}
	x && \rightsquigarrow & x & {} \\
	sx & T &&& T
	\arrow["\F"', no head, from=1-1, to=2-1]
	\arrow[no head, shorten >= -4.7pt, from=1-4, to=1-5]
	\arrow["\F"', no head, shorten <= -6.1pt, from=1-5, to=2-5]
	\arrow[no head, from=2-1, to=2-2]
\end{tikzcd}\]

It turns out that this diagram implication is in a broader form.

\begin{lemma} \label{lem: suspension discription for normal}
    Let $\F \in \mF(Q)$ be normal, then for $S, T \in \Sigma M$ we have diagram implication:
    
    \[\begin{tikzcd}
    	S && \rightsquigarrow & S & {} \\
    	{} & T &&& T
    	\arrow["\F"', no head, shorten >= -1.5pt, from=1-1, to=2-1]
    	\arrow[no head, shorten >= -4.7pt, from=1-4, to=1-5]
    	\arrow["\F"', no head, shorten <= -6.1pt, from=1-5, to=2-5]
    	\arrow[no head, shorten <= -4.7pt, from=2-1, to=2-2]
    \end{tikzcd}\]
\end{lemma}

\begin{proof}
    For $x \in S$, there exists $s^x \in \F$ such that $s^x x \le \sigma T$.
    Thus, there exists $W^x \in \Sigma M$ such that
    
    \[\begin{tikzcd} [row sep = small, column sep = small]
    	x & {W^x} \\
    	& T
    	\arrow[no head, from=1-1, to=1-2]
    	\arrow["\F"', no head, from=1-2, to=2-2]
    \end{tikzcd}\]

    Now if we take $W = \sum W^x$ (summation inside $\Sigma M$) then we have
    
    \[\begin{tikzcd} [row sep = small, column sep = small]
    	S & W \\
    	& T
    	\arrow[no head, from=1-1, to=1-2]
    	\arrow["\F"', no head, from=1-2, to=2-2]
    \end{tikzcd}\]
\end{proof}

\begin{example}
    The m-filter $\br{1}$ is normal over any $Q$-qmodule.
\end{example}

\begin{example}
    Let $X$ be a regular topological space, then $\F_{\perp U} \in \mF(\O(X))$ is normal over $\O(X)$ for all $U \in \O(X)$. 
    Suppose we have $V \supset U^c$ such that $V \cap W \subset \bigcup \br{W_i}_{i \in I}$. We consider points in $W$: For point $w \in W \cap U$, we could separate it with $U^c$ by separation axiom of regular space, thus we take some disjoint neighborhood $W' \ni w$ and $V' \supset U^c$. For point $w \in W \cap U^c$, we have $w \in V \cap W$, thus $w \in W_i$ for some $i \in I$. Thus, by taking all such $W'$ above and all $W_i$, the collection satisfies the normal condition. 
\end{example}

As promised, normal filters are closed under arbitrary addition.

\begin{prop} \label{prop: normal preserved under addition}
    Suppose $\N_i \in \mF(Q)$ are normal over $M$.
    Then $\sum_{i \in I} \N_i$ is normal over $M$.
\end{prop}

\begin{proof}
    The next diagram explains the proof.
    
    \[\begin{tikzcd} [column sep = small]
    	x &&& x &&& x &&& x & \sum \\
    	{s_1 x} &&& {s_1 x} &&& {s_1 x} & \sum &&& \sum \\
    	{s_2 s_1 x} && \rightsquigarrow & {s_2 s_1 x} & \sum & \rightsquigarrow && \sum & \rightsquigarrow && \sum \\
    	{s_3 s_2 s_1 x} & W &&& W &&& W &&& W
    	\arrow["{\F_1}"', no head, from=1-1, to=2-1]
    	\arrow["{\F_1}"', no head, from=1-4, to=2-4]
    	\arrow["{\F_1}"', no head, from=1-7, to=2-7]
    	\arrow[no head, from=1-10, to=1-11]
    	\arrow["{\F_1}", no head, from=1-11, to=2-11]
    	\arrow["{\F_2}"', no head, from=2-1, to=3-1]
    	\arrow["{\F_2}"', no head, from=2-4, to=3-4]
    	\arrow[no head, from=2-7, to=2-8]
    	\arrow["{\F_2}", no head, from=2-8, to=3-8]
    	\arrow["{\F_2}", no head, from=2-11, to=3-11]
    	\arrow["{\F_3}"', no head, from=3-1, to=4-1]
    	\arrow[no head, from=3-4, to=3-5]
    	\arrow["{\F_3}", no head, from=3-5, to=4-5]
    	\arrow["{\F_3}", no head, from=3-8, to=4-8]
    	\arrow["{\F_3}", no head, from=3-11, to=4-11]
    	\arrow[no head, from=4-1, to=4-2]
    \end{tikzcd}\]
\end{proof}

Since the normal property is closed under arbitrary addition, and since $\br{1}$ is normal over any module, we could define the notion of greatest normal filter contained in a fixed filter: 

\begin{definition}
    Let $\F$ be a m-filter of $Q$ and $M$ be a $Q$-module. Then we define $\gnf_M(\F)$ be the sum of all normal m-filters contained in $\F$. When $M = Q$, we omit the $Q$ and write $\gnf(\F)$. 
\end{definition}

\begin{example}
    Let $X$ be a regular topological space and $Q = \O(X)$. Then $\gnf(\F_{\perp U}) = \F_{\text{int}(U^C)}$. 
\end{example}

When $M$ is shrinkable, it turns out that the collection of normal filters form a subquantale in $\mF(Q)$.

\begin{prop}
    Suppose $M$ is shrinkable.
    Let $\F, \G \in \mF(Q)$ be normal over $M$, then $\F \G$ is normal over $M$.
\end{prop}

\begin{proof}
    The next diagram explains the proof.
    Here the $\wedge$ is taken inside $\Sigma M$, and is doable by \Cref{thm: Suspension Shrinkable Meet}.
    \[\begin{tikzcd} [column sep = small]
    	x && \rightsquigarrow & x && {\text{and}} & x \\
    	{(s + t) x} & W && {s x} & W && {t x} & W \\
    	&& \rightsquigarrow & x & {S} & {\text{and}} & x & {T} \\
    	&&&& W &&& W \\
    	&& \rightsquigarrow & x & {S \wedge T} & {\text{and}} & x & {S \wedge T} \\
    	&&&& W &&& W \\
    	&& \rightsquigarrow & x & {S \wedge T} \\
    	&&&& W
    	\arrow["{\F \G}"', no head, from=1-1, to=2-1]
    	\arrow["\F"', no head, from=1-4, to=2-4]
    	\arrow["\G"', no head, from=1-7, to=2-7]
    	\arrow[no head, from=2-1, to=2-2]
    	\arrow[no head, from=2-4, to=2-5]
    	\arrow[no head, from=2-7, to=2-8]
    	\arrow[no head, from=3-4, to=3-5]
    	\arrow["\F", no head, from=3-5, to=4-5]
    	\arrow[no head, from=3-7, to=3-8]
    	\arrow["\G", no head, from=3-8, to=4-8]
    	\arrow[no head, from=5-4, to=5-5]
    	\arrow["\F", no head, from=5-5, to=6-5]
    	\arrow[no head, from=5-7, to=5-8]
    	\arrow["\G", no head, from=5-8, to=6-8]
    	\arrow[no head, from=7-4, to=7-5]
    	\arrow["{\F \G}", no head, from=7-5, to=8-5]
    \end{tikzcd}\]
\end{proof}

For blooming quantales/modules, it is surprising that normality over quantale implies normality over modules.

\begin{theorem} \label{thm: normal in blooming}
    For blooming $Q$ and $\F \in \mF(Q)$, the following are equivalent:
    \begin{enum}
        \item $\F$ is normal over $Q$.
        \item $\F$ is locally solid.
        \item $\F$ is normal over all blooming $Q$-modules.
    \end{enum}
\end{theorem}

\begin{proof}
    (a) $\Rightarrow$ (b): let $W = \sigma^\flat(1)$.
    For each $s \in \F$ and $S \in \sigma\inv(s)$, we have diagram
    \[\begin{tikzcd} [row sep = small, column sep = small]
    	1 \\
    	s & S
    	\arrow[no head, from=1-1, to=2-1]
    	\arrow[no head, from=2-1, to=2-2]
    \end{tikzcd}\]

    Thus, since $\F$ is normal, we see there exists $W' \in \sigma\inv(1)$ fulfilling the diagram
    \[\begin{tikzcd} [row sep = small, column sep = small]
    	1 & W' \\
    	& S
    	\arrow[no head, from=1-1, to=1-2]
    	\arrow[no head, from=1-2, to=2-2]
    \end{tikzcd}\]

    Now from the choice of $W$ we have
    \[\begin{tikzcd} [row sep = small, column sep = small]
    	1 & W \\
    	& {W'} \\
    	& S
    	\arrow[no head, from=1-1, to=1-2]
    	\arrow[no head, from=1-2, to=2-2]
    	\arrow[no head, from=2-2, to=3-2]
    \end{tikzcd}\]

    Thus, by \Cref{prop: diagram merge consecutive lines} we have
    \[\begin{tikzcd} [row sep = small, column sep = small]
    	1 & W \\
    	& S
    	\arrow[no head, from=1-1, to=1-2]
    	\arrow[no head, from=1-2, to=2-2]
    \end{tikzcd}\]
    which is the condition for $\F$ being locally solid.

    \Skip (b) $\Rightarrow$ (c): let $M$ be a blooming $Q$-module.
    For $x \in M, s \in \F$, suppose $\sigma A \ge s x$.
    Let $S = \sigma^\flat(s), W = \sigma^\flat(1), B = \sigma^\flat(x)$.
    Then for each $w \in W$, there exists $t_w \in \F$ such that $\br{t_w w} \le S$ by \Cref{lem: locally solid lower witness}.

    We have $x = \sum \br{w b : w \in W, b \in B}$ and $\sigma^\flat(s x) = \sigma^\flat(s) \sigma^\flat(x) = \sum_{b \in B} (S \cdot \br{b})$.
    Now for $b \in B$ we have $\br{t_w w b} = \br{t_w w} \cdot \br{b} \le S \cdot \br{b}$ in $\Sigma M$, meaning we have diagram
    \[\begin{tikzcd} [row sep = small, column sep = small]
    	x & {W B} \\
    	& {S B} \\
    	& A
    	\arrow[no head, from=1-1, to=1-2]
    	\arrow[no head, from=1-2, to=2-2]
    	\arrow[no head, from=2-2, to=3-2]
    \end{tikzcd}\]

    Thus, by \Cref{prop: diagram merge consecutive lines} we get
    \[\begin{tikzcd} [row sep = small, column sep = small]
    	x & {W B} \\
    	& A
    	\arrow[no head, from=1-1, to=1-2]
    	\arrow[no head, from=1-2, to=2-2]
    \end{tikzcd}\]
    hence $\F$ is normal over $M$.

    \Skip (c) $\Rightarrow$ (a): just note that $Q$ is a blooming $Q$-module.
\end{proof}

\begin{cor} \label{cor: normal in coherent}
    Let $Q$ be a coherent quantale.
    Then a m-filter $\F \in \mF(Q)$ is normal over $Q$ if and only if all $x \in \F$, there exists a compact element $y \in \F$ such that $y \le x$.
\end{cor}

\begin{proof}
    Follows from \Cref{cor: compact locally solid is solid}, \Cref{prop: compact element control is solid}, and \Cref{thm: normal in blooming}.
\end{proof}

\begin{example}
    Let $R$ be a ring and $S \sub R$ be a multiplicatively closed subset.
    Consider m-filter $\N = \br{J \in \Id(R) : J \cap S \ne \varnothing}$.
    Then $\N$ is normal since for all $J \in \N$, there exists $s \in S \cap J$, so that $(s) \sub J$ and $(s) \in \N$, hence done by \Cref{cor: normal in coherent}.
    
    Moreover, we have $\Id(R)_\N \cong \Id(R_S)$ via quantale isomorphism $\Id(R)_\N \iso \Id(R_S), \=J \mapsto J R_S$. This is because the one-stepness gives $j_1s_1 \in J_2$ and $j_2s_2 \in J_1$ for $j_i \in J_i$, $s_i \in S$ equivalent to $\=J_1 = \=J_2$, which is exactly the case for $J_1R_S = J_2R_S$ based on \cite[Proposition~2.2]{E95}.
\end{example}

\begin{example}
    Similarly, let $R$ be a ring, $M$ be an $R$-module, and $S \sub R$ be a multiplicatively closed subset.
    Consider m-filter $\N = \br{J \in \Id(R) : J \cap S \ne \varnothing}$.
    Then we have $\Id(R)_\N$-module ($= \Id(R_S)$-module) isomorphism $\Sub_R(M)_\N \iso \Sub_{R_S}(M_S)$. This is due to the same reason, saying $n_1s_1 \in N_2$, $n_2s_2 \in N_1$ for $n_i \in N_i$, $s_i \in S$ and submodules $\=N_1 = \=N_2$, equivalent with $N_1R_S = N_2R_S$, also based on \cite[Proposition~2.2]{E95}.
\end{example}

\begin{prop} \label{prop: normal implies 1-step}
    Let $\N$ be a normal m-filter over $M$, then $\N$ is 1-step over $M$.
\end{prop}

\begin{proof}
    It suffices to show that for $x, y \in M$, we have $x \preceq^2 y$ implies $x \preceq^1 y$.
    However, this is guaranteed by the following diagram implication (here $\N$ corresponds to straight vertical lines):
    
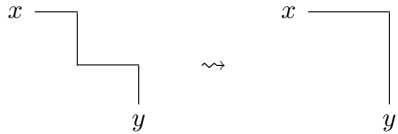
\[\begin{tikzcd} [row sep = small, column sep = small]
    	x & {} &&& x && {} \\
    	& {} & {} & \rightsquigarrow \\
    	&& y &&&& y
    	\arrow[shorten >= -4.7pt, no head, from=1-1, to=1-2]
    	\arrow[shorten <= -6.1pt, shorten >= -1.5pt, no head, from=1-2, to=2-2]
    	\arrow[shorten >= -4.7pt, no head, from=1-5, to=1-7]
    	\arrow[shorten <= -6.1pt, no head, from=1-7, to=3-7]
    	\arrow[shorten <= -4.7pt, shorten >= -4.7pt, no head, from=2-2, to=2-3]
    	\arrow[shorten <= -6.1pt, no head, from=2-3, to=3-3]
    \end{tikzcd}\]
\end{proof}

\begin{prop} \label{prop: shrinkable normal}
    Let $\N$ be a normal m-filter over a shrinkable $M$, then for any $m, m_i \in M$, if $\=m \le \sum_{i \in I} \={m_i}$ in $M_\N$, then there exist $\br{n_j}_{j \in J} \sub M, \br{s_j}_{j \in J} \sub \N$ such that $m = \sum_{j \in J} n_j$ and $s_j n_j \le^* \sum_{i \in I} m_i$.
\end{prop}

\begin{proof}
    By \Cref{prop: normal implies 1-step}, after letting $A = \br{m_i}_{i \in I}$ we have diagram
    \[\begin{tikzcd} [row sep = small, column sep = small]
    	m & {} \\
    	& {} & A
    	\arrow[shorten >=-4.7pt, no head, from=1-1, to=1-2]
    	\arrow[shorten <=-6.1pt, shorten >=-1.5pt, no head, from=1-2, to=2-2]
    	\arrow[shorten <=-4.7pt, no head, from=2-2, to=2-3]
    \end{tikzcd}\]

    Thus, by the definition of normality and \Cref{prop: diagram merge consecutive lines} we get diagram
    
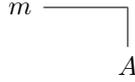
\[\begin{tikzcd} [row sep = small, column sep = small]
    	m && {} \\
    	&& A
    	\arrow[shorten >=-4.7pt, no head, from=1-1, to=1-3]
    	\arrow[shorten <=-6.1pt, no head, from=1-3, to=2-3]
    \end{tikzcd}\]
    which corresponds exactly to the statement of this proposition.
\end{proof}

\begin{lemma} \label{lem: normal + mf local preorder}
    Let $\N \in \mF(Q)$ be normal over $M$ and $\F \in \mF(Q)$.
    For $x, y \in M$, if $x \preceq_{\N + \F}^n y$ for some $n \in \NN$, then there exists $z \in M$ such that $x \preceq_\F^n z \preceq_\N^1 y$.
\end{lemma}

\begin{proof}
    Explained by the following diagram (the diagram implications are valid by \Cref{prop: diagram merge consecutive lines}) (here $\N$ corresponds to straight vertical lines and $\F$ corresponds to dashed vertical lines):
    \[\begin{tikzcd} [row sep = small, column sep = small]
    	& x & {} &&&& x & {} &&&& x & {} \\
    	&&&&&&& {} &&&&& {} & {} \\
    	&& {} & {} &&&& {} & {} &&&&& {} \\
    	&&&&& \rightsquigarrow &&& {} && \rightsquigarrow &&& {} & {} \\
    	&&& {} & {} &&&& {} & {} &&&&& {} \\
    	&&&&&&&&& {} &&&&& {} \\
    	&&&& y &&&&& y &&&&& y \\
    	& x & {} &&&& x & {} &&&& x & {} \\
    	&& {} & {} &&&& {} & {} &&&& {} & {} \\
    	&&& {} &&&&& {} & {} &&&& {} & {} \\
    	\rightsquigarrow &&& {} & {} & \rightsquigarrow &&&& {} & \rightsquigarrow &&&& {} \\
    	&&&& {} &&&&& {} \\
    	\\
    	&&&& y &&&&& y &&&&& y
    	\arrow[shorten >=-4.7pt, no head, from=1-2, to=1-3]
    	\arrow["{\N + \F}"', shorten <=-6.1pt, shorten >=-1.5pt, squiggly, no head, from=1-3, to=3-3]
    	\arrow[shorten >=-4.7pt, no head, from=1-7, to=1-8]
    	\arrow[shorten <=-6.1pt, shorten >=-1.5pt, dashed, no head, from=1-8, to=2-8]
    	\arrow[shorten >=-4.7pt, no head, from=1-12, to=1-13]
    	\arrow[shorten <=-6.1pt, shorten >=-1.5pt, dashed, no head, from=1-13, to=2-13]
    	\arrow[shorten <=-6.1pt, shorten >=-1.5pt, no head, from=2-8, to=3-8]
    	\arrow[shorten <=-4.7pt, shorten >=-4.7pt, no head, from=2-13, to=2-14]
    	\arrow[shorten <=-6.1pt, shorten >=-1.5pt, no head, from=2-14, to=3-14]
    	\arrow[shorten <=-4.7pt, shorten >=-4.7pt, no head, from=3-3, to=3-4]
    	\arrow["{\N + \F}"', shorten <=-6.1pt, shorten >=-1.5pt, squiggly, no head, from=3-4, to=5-4]
    	\arrow[shorten <=-4.7pt, shorten >=-4.7pt, no head, from=3-8, to=3-9]
    	\arrow[shorten <=-6.1pt, shorten >=-1.5pt, dashed, no head, from=3-9, to=4-9]
    	\arrow[shorten <=-6.1pt, shorten >=-1.5pt, dashed, no head, from=3-14, to=4-14]
    	\arrow[shorten <=-6.1pt, shorten >=-1.5pt, no head, from=4-9, to=5-9]
    	\arrow[shorten <=-4.7pt, shorten >=-4.7pt, no head, from=4-14, to=4-15]
    	\arrow[shorten <=-6.1pt, shorten >=-1.5pt, no head, from=4-15, to=5-15]
    	\arrow[shorten <=-4.7pt, shorten >=-4.7pt, no head, from=5-4, to=5-5]
    	\arrow["{\N + \F}"', shorten <=-6.1pt, squiggly, no head, from=5-5, to=7-5]
    	\arrow[shorten <=-4.7pt, shorten >=-4.7pt, no head, from=5-9, to=5-10]
    	\arrow[shorten <=-6.1pt, shorten >=-1.5pt, dashed, no head, from=5-10, to=6-10]
    	\arrow[shorten <=-6.1pt, shorten >=-1.5pt, dashed, no head, from=5-15, to=6-15]
    	\arrow[shorten <=-6.1pt, shorten >=-1.5pt, no head, from=6-10, to=7-10]
    	\arrow[shorten <=-6.1pt, no head, from=6-15, to=7-15]
    	\arrow[shorten >=-4.7pt, no head, from=8-2, to=8-3]
    	\arrow[shorten <=-6.1pt, shorten >=-1.5pt, dashed, no head, from=8-3, to=9-3]
    	\arrow[shorten >=-4.7pt, no head, from=8-7, to=8-8]
    	\arrow[shorten <=-6.1pt, shorten >=-1.5pt, dashed, no head, from=8-8, to=9-8]
    	\arrow[shorten >=-4.7pt, no head, from=8-12, to=8-13]
    	\arrow[shorten <=-6.1pt, shorten >=-1.5pt, dashed, no head, from=8-13, to=9-13]
    	\arrow[shorten <=-4.7pt, shorten >=-4.7pt, no head, from=9-3, to=9-4]
    	\arrow[shorten <=-6.1pt, shorten >=-1.5pt, dashed, no head, from=9-4, to=10-4]
    	\arrow[shorten <=-4.7pt, shorten >=-4.7pt, no head, from=9-8, to=9-9]
    	\arrow[shorten <=-6.1pt, shorten >=-1.5pt, dashed, no head, from=9-9, to=10-9]
    	\arrow[shorten <=-4.7pt, shorten >=-4.7pt, no head, from=9-13, to=9-14]
    	\arrow[shorten <=-6.1pt, shorten >=-1.5pt, dashed, no head, from=9-14, to=10-14]
    	\arrow[shorten <=-6.1pt, shorten >=-1.5pt, no head, from=10-4, to=11-4]
    	\arrow[shorten <=-4.7pt, shorten >=-4.7pt, no head, from=10-9, to=10-10]
    	\arrow[shorten <=-6.1pt, shorten >=-1.5pt, no head, from=10-10, to=11-10]
    	\arrow[shorten <=-4.7pt, shorten >=-4.7pt, no head, from=10-14, to=10-15]
    	\arrow[shorten <=-6.1pt, shorten >=-1.5pt, dashed, no head, from=10-15, to=11-15]
    	\arrow[shorten <=-4.7pt, shorten >=-4.7pt, no head, from=11-4, to=11-5]
    	\arrow[shorten <=-6.1pt, shorten >=-1.5pt, dashed, no head, from=11-5, to=12-5]
    	\arrow[shorten <=-6.1pt, shorten >=-1.5pt, dashed, no head, from=11-10, to=12-10]
    	\arrow[shorten <=-6.1pt, no head, from=11-15, to=14-15]
    	\arrow[shorten <=-6.1pt, no head, from=12-5, to=14-5]
    	\arrow[shorten <=-6.1pt, no head, from=12-10, to=14-10]
    \end{tikzcd}\]
\end{proof}

\begin{prop} \label{prop: normal add localizable is localizable}
    Let $\N \in \mF(Q)$ be a normal over $M$ and $\F \in \mF(Q)$ be localizable over $M$.
    Then $\N + \F$ is localizable over $M$.
\end{prop}

\begin{proof}
    Suppose $x \preceq_{\N + \F} y$, then we have $x \preceq_\F^1 z \preceq_\N^1 y$ for some $z \in M$ by \Cref{lem: normal + mf local preorder}.
    Let $\^z = \sum \br{z \in M : z \preceq_\N^1 y}$, then $\^z \preceq_\N^1 y$ and $x \preceq_\F \^z \preceq_\N^1 y$.
    Now in \Cref{prop: chan of localizable} we can take $n_y^{\N + \F} = n_{\^z}^\F + 1$.
\end{proof}

\begin{example}
    Thus, if $Q$ is the quantale of open sets in $[0, 1]$ and $a = [0, 1 / 2) \in Q$, then $\F_a$ is not normal, since otherwise $\F_a \F_{\perp a}$ would be normal, hence 1-step, which is false.
\end{example}

\begin{example}
    Let $Q$ be the quantale of open sets in a regular topological space and $a \in Q$.
    Generalize the above arguments we have $\F_a$ is normal if and only if $a$ is also closed.
\end{example}

\Cref{prop: normal add localizable is localizable} combined with \Cref{prop: MFG = MGF} tell us that when $\N$ is normal over $M$ and $\F$ is localizable over $M$, then $\N$ is localizable over $M_\F$ and $\F$ is localizable over $M_\N$.
Moreover, we have isomorphisms (between $Q$-modules) $M_{\N + \F} \iso (M_\F)_\N, \=x \mapsto \={\=x}$ and $M_{\N + \F} \iso (M_\N)_\F, \=x \mapsto \={\=x}$.
The next proposition shows that in this case we have the stronger statement that $\N$ is normal over $M_\F$.

\begin{prop} \label{prop: normal inheritance}
    Let $\N$ be a normal m-filter over $M$ and $\F$ be a localizable (resp., 1-step) m-filter over $M$.
    Then
    \begin{enum}
        \item $\N$ is normal over $M_\F$.
        \item $\F$ is localizable (resp., 1-step) over $M_\N$.
    \end{enum}
\end{prop}

\begin{proof}
    Part (a): let $x \in M, s \in \N, \br{y_i}_{i \in I} \sub M$ satisfy $s \={x} \le \sum_{i \in I} \={y_i}$ in $M_\F$.
    Then we have $s x \preceq_\F \sum_{i \in I} y_i$, so from diagram (the diagram implications are valid by \Cref{prop: diagram merge consecutive lines}) ($\N$ is straight and $\F$ is dashed vertical lines)
    
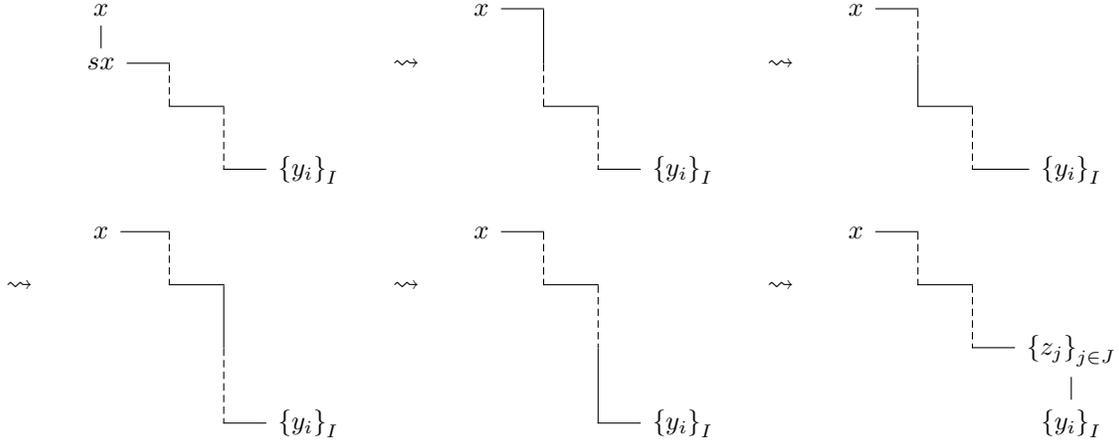
\[\begin{tikzcd} [row sep = small, column sep = small]
    	& x &&&&& x & {} &&&& x & {} \\
    	& sx & {} &&& \rightsquigarrow && {} &&& \rightsquigarrow && {} \\
    	&& {} & {} &&&& {} & {} &&&& {} & {} \\
    	&&& {} & {\br{y_i}_I} &&&& {} & {\br{y_i}_I} &&&& {} & {\br{y_i}_I} \\
    	& x & {} &&&& x & {} &&&& x & {} \\
    	\rightsquigarrow && {} & {} && \rightsquigarrow && {} & {} && \rightsquigarrow && {} & {} \\
    	&&& {} &&&&& {} &&&&& {} & {\br{z_j}_{j \in J}} \\
    	&&& {} & {\br{y_i}_I} &&&& {} & {\br{y_i}_I} &&&&& {\br{y_i}_I}
    	\arrow[no head, from=1-2, to=2-2]
    	\arrow[shorten >=-4.7pt, no head, from=1-7, to=1-8]
    	\arrow[shorten <=-6.1pt, shorten >=-1.5pt, no head, from=1-8, to=2-8]
    	\arrow[shorten >=-4.7pt, no head, from=1-12, to=1-13]
    	\arrow[shorten <=-6.1pt, shorten >=-1.5pt, dashed, no head, from=1-13, to=2-13]
    	\arrow[shorten >=-4.7pt, no head, from=2-2, to=2-3]
    	\arrow[shorten <=-6.1pt, shorten >=-1.5pt, dashed, no head, from=2-3, to=3-3]
    	\arrow[shorten <=-6.1pt, shorten >=-1.5pt, dashed, no head, from=2-8, to=3-8]
    	\arrow[shorten <=-6.1pt, shorten >=-1.5pt, no head, from=2-13, to=3-13]
    	\arrow[shorten <=-4.7pt, shorten >=-4.7pt, no head, from=3-3, to=3-4]
    	\arrow[shorten <=-6.1pt, shorten >=-1.5pt, dashed, no head, from=3-4, to=4-4]
    	\arrow[shorten <=-4.7pt, shorten >=-4.7pt, no head, from=3-8, to=3-9]
    	\arrow[shorten <=-6.1pt, shorten >=-1.5pt, dashed, no head, from=3-9, to=4-9]
    	\arrow[shorten <=-4.7pt, shorten >=-4.7pt, no head, from=3-13, to=3-14]
    	\arrow[shorten <=-6.1pt, shorten >=-1.5pt, dashed, no head, from=3-14, to=4-14]
    	\arrow[shorten <=-4.7pt, no head, from=4-4, to=4-5]
    	\arrow[shorten <=-4.7pt, no head, from=4-9, to=4-10]
    	\arrow[shorten <=-4.7pt, no head, from=4-14, to=4-15]
    	\arrow[shorten >=-4.7pt, no head, from=5-2, to=5-3]
    	\arrow[shorten <=-6.1pt, shorten >=-1.5pt, dashed, no head, from=5-3, to=6-3]
    	\arrow[shorten >=-4.7pt, no head, from=5-7, to=5-8]
    	\arrow[shorten <=-6.1pt, shorten >=-1.5pt, dashed, no head, from=5-8, to=6-8]
    	\arrow[shorten >=-4.7pt, no head, from=5-12, to=5-13]
    	\arrow[shorten <=-6.1pt, shorten >=-1.5pt, dashed, no head, from=5-13, to=6-13]
    	\arrow[shorten <=-4.7pt, shorten >=-4.7pt, no head, from=6-3, to=6-4]
    	\arrow[shorten <=-6.1pt, shorten >=-1.5pt, no head, from=6-4, to=7-4]
    	\arrow[shorten <=-4.7pt, shorten >=-4.7pt, no head, from=6-8, to=6-9]
    	\arrow[shorten <=-6.1pt, shorten >=-1.5pt, dashed, no head, from=6-9, to=7-9]
    	\arrow[shorten <=-4.7pt, shorten >=-4.7pt, no head, from=6-13, to=6-14]
    	\arrow[shorten <=-6.1pt, shorten >=-1.5pt, dashed, no head, from=6-14, to=7-14]
    	\arrow[shorten <=-6.1pt, shorten >=-1.5pt, dashed, no head, from=7-4, to=8-4]
    	\arrow[shorten <=-6.1pt, shorten >=-1.5pt, no head, from=7-9, to=8-9]
    	\arrow[shorten <=-4.7pt, no head, from=7-14, to=7-15]
    	\arrow[no head, from=7-15, to=8-15]
    	\arrow[shorten <=-4.7pt, no head, from=8-4, to=8-5]
    	\arrow[shorten <=-4.7pt, no head, from=8-9, to=8-10]
    \end{tikzcd}\]
    we see $x \preceq_\F \sum_J z_j$ and $t_j z_j \le^* \sum_I y_i$ for some $t_j \in \N$.
    As a result, inside $M_\F$ we have $\=x \le \sum_{j \in J} \={z_j}$ and $t_j \={z_j} \le^* \sum_{i \in I} \={y_i}$, showing $\N$ is normal over $M_\F$, as desired.

    \Skip Part (b): if $\F$ is localizable over $M$, then $\F$ is localizable over $M_\N$ by \Cref{prop: MFG = MGF} and \Cref{prop: normal add localizable is localizable}.
    Now assume $\F$ is 1-step over $M$ and we want to show $\F$ is 1-step over $M_\N$.
    Pick $x, y \in M$ such that $\=x \preceq_\F \=y$ in $M_\N$.
    Then \Cref{lem: repr in MF is G local <= implies F+G local <=} shows $x \preceq_{\F + \N} y$, so \Cref{lem: normal + mf local preorder} exhibits $z \in M$ such that $x \preceq_\F z \preceq_\N^1 y$.
    Then because $\F$ is 1-step over $M$, we see $x \preceq_\F^1 z$, so $\=x \preceq_\F^1 \=z$ in $M_\N$ and $\=z \le \=y$ in $M_\N$, hence $\=x \preceq_\F^1 \=y$ in $M_\N$, as desired.
\end{proof}

Now let us show that bloomingness is preserved under localization at normal m-filters.

\begin{prop} \label{prop: blooming preserved under normal}
    Let $Q$ be a blooming quantale and $M$ a blooming $Q$-module.
    Then for $\N \in \mF(Q)$ normal over $M$, we have
    \begin{enum}
        \item $M_\N$ is a blooming $Q$-module with $\sigma_{M_\N}^\flat(\=x) = \br{\=y : y \in \sigma_M^\flat(x)}$ for $x \in M$.
        \item If $\N$ is locally solid, then $M_\N$ is a blooming $Q_\N$-module.
    \end{enum}
\end{prop}

\begin{proof}
    Let $\sigma = \sigma_{M_\N}$.
    For $x \in M$, let $T = \br{\=y : y \in \sigma_M^\flat(x)} \in \Sigma M_\N$.
    Then $\sigma T = \=x$ by definition.
    Now pick any $\br{\={\al_i}}_{i \in I} \in \Sigma M_\N$ (for some $\al_i \in M$) such that $\sigma \br{\={\al_i}}_{i \in I} \ge \=x$ in $M_\N$.
    Then $x \preceq_\N \sum_i \al_i$, so since $\N$ is normal over $M$ we see we have diagram implications
    \[\begin{tikzcd}
    	x & \sum && x & {\sigma^\flat(x)} && x & {\sigma^\flat(x)} \\
    	&& \rightsquigarrow && \sum & \rightsquigarrow \\
    	& {\br{\al_i}_{i \in I}} &&& {\br{\al}_{i \in I}} &&& {\br{\al}_{i \in I}}
    	\arrow[no head, from=1-1, to=1-2]
    	\arrow[no head, from=1-2, to=3-2]
    	\arrow[no head, from=1-4, to=1-5]
    	\arrow[no head, from=1-5, to=2-5]
    	\arrow[no head, from=1-7, to=1-8]
    	\arrow[no head, from=1-8, to=3-8]
    	\arrow[no head, from=2-5, to=3-5]
    \end{tikzcd}\]
    from which we see $T \le \br{\={\al_i}}_{i \in I}$ in $\Sigma M_\N$, as desired.

    The multiplicative closedness is easy to check (for Part (b) we should use \Cref{thm: normal in blooming}).
\end{proof}

\smallskip

\subsection{Conormal Filters}

There's also another special kind of 1-step m-filter, which is the dual of normal m-filters.

\begin{definition}
    Let $\C$ be a m-filter for $Q$ and $M$ a $Q$-module. 
    We say $\C$ is \emph{conormal} over $M$ if for all $m, n \in M$ with $m \preceq_{\C}^1 n$, there exist $s \in \C$ such that $sm \le n$. 
\end{definition}

In other words, a conormal m-filter satisfies the following diagram implication:

\[\begin{tikzcd}
	x & {} & \rightsquigarrow & x \\
	& y && {} & y
	\arrow[shorten >=-4.7pt, no head, from=1-1, to=1-2]
	\arrow["\C"', shorten <=-6.1pt, no head, from=1-2, to=2-2]
	\arrow["\C"', shorten >=-1.5pt, no head, from=1-4, to=2-4]
	\arrow[shorten <=-4.7pt, no head, from=2-4, to=2-5]
\end{tikzcd}\]

Similarly to \Cref{lem: suspension discription for normal}, by merging sets in $\Sigma M$ in the above diagram we can replace $x$ and $y$ by elements in $\Sigma M$.

We also give a definition of binormal, which satisfies both normal and conormal property: 

\begin{definition}
    Let $\B$ be a m-filter for $Q$ and $M$ a $Q$-module. We say $\B$ is \emph{binormal} over $M$ if $\B$ is both normal and conormal over $M$. 
\end{definition}

\begin{example}
    A minimal m-filter on an idempotent quantale $Q$ is conormal over $Q$. 
\end{example}

\begin{example}
    Let $\B$ be a m-filter over $Q$ and $M$ be a Noetherian $Q$-module over $Q$. Then $\B$ is binormal over $M$. The normal part is due to $sm \le^* \sum_{i \in I} m_i$ and we take $\br{m_j'} = \br{m}$, and the conormal part is due to $m \le \sum_{i \in I} m_i$ gives $m \le^* \sum_{i \in I} m_i$, and by taking $m \le \sum_{i \in I_0} m_i$ for finite $I_0$, we have $\pr{\prod_{i \in I_0} s_i}m \le n$ with $s_im_i \le n$. 
\end{example}

Similar to the addition preservation \Cref{prop: normal preserved under addition}, we have a reversed property for conormal filter: 

\begin{prop} \label{prop: conormal preserved under product}
    Let $\C, \br{\F_i}_{i \in I}$ be m-filters. Suppose $\C = \bigcap_{i \in I} \F_i$ and all $\F_i$'s are conormal over $M$, then $\C$ is conormal over $M$. 
\end{prop}

\begin{proof}
    The diagram shows the proof: 
    \[\begin{tikzcd}
	x & \sum & \rightsquigarrow & x & \sum \\
	& y &&& y & {\forall \ i} \\
	&& \rightsquigarrow & x \\
	&&& {s_ix} & y & {\forall \ i} \\
	&& \rightsquigarrow & x \\
	&&& {\left(\sum s_i\right)x} & y
	\arrow[no head, from=1-1, to=1-2]
	\arrow["{\bigcap \F_i}", no head, from=1-2, to=2-2]
	\arrow[no head, from=1-4, to=1-5]
	\arrow["{\F_i}", no head, from=1-5, to=2-5]
	\arrow["{\F_i}", no head, from=3-4, to=4-4]
	\arrow[no head, from=4-4, to=4-5]
	\arrow["{\bigcap \F_i}", no head, from=5-4, to=6-4]
	\arrow[no head, from=6-4, to=6-5]
    \end{tikzcd}\]
\end{proof}

Thus, we could define the least conormal filter containing $\F$, but notice that not all filters are contained in some conormal filter: 

\begin{definition}
    Let $\F$ be a m-filter of $Q$, and $M$ be a $Q$-module. Then we define $\lcf_M(\F)$ be the intersection of all conormal m-filter over $M$ containing $\F$, given that $\F$ is contained in some conormal m-filter over $M$. Specially, when $M = Q$, $\lcf_Q(\F)$ always exists, as $Q$ is conormal over itself. In this case we omit the $Q$ and write $\lcf(\F)$. 
\end{definition}

\begin{prop} \label{prop: conormal preserved under addition}
    Let $\F, \G$ be conormal m-filters over $M$, then $\F + \G$ is a conormal m-filter over $M$. 
\end{prop}

\begin{proof}
    The diagram shows the proof: 
    \[\begin{tikzcd}
	x & \sum & \rightsquigarrow & x && \rightsquigarrow & x \\
	& \sum && sx & \sum && sx \\
	& y &&& y && stx & y
	\arrow[no head, from=1-1, to=1-2]
	\arrow["\F", no head, from=1-2, to=2-2]
	\arrow["\F", no head, from=1-4, to=2-4]
	\arrow["\F", no head, from=1-7, to=2-7]
	\arrow["\G", no head, from=2-2, to=3-2]
	\arrow[no head, from=2-4, to=2-5]
	\arrow["\G", no head, from=2-5, to=3-5]
	\arrow["\G", no head, from=2-7, to=3-7]
	\arrow[no head, from=3-7, to=3-8]
    \end{tikzcd}\]
\end{proof}

\begin{lemma} \label{lem: conormal + mf local preorder}
    Let $\C \in \mF(Q)$ be conormal over $M$ and $\F \in \mF(Q)$.
    For $x, y \in M$, if $x \preceq_{\C + \F}^n y$ for some $n \in \NN$, then there exists $z \in M$ such that $x \preceq_\C^1 z \preceq_\F^n y$.
\end{lemma}

\begin{proof}
    Explained by the following diagram (the diagram implications are valid by \Cref{prop: diagram merge consecutive lines}) (here $\C$ corresponds to straight vertical lines and $\F$ corresponds to dashed vertical lines):
     
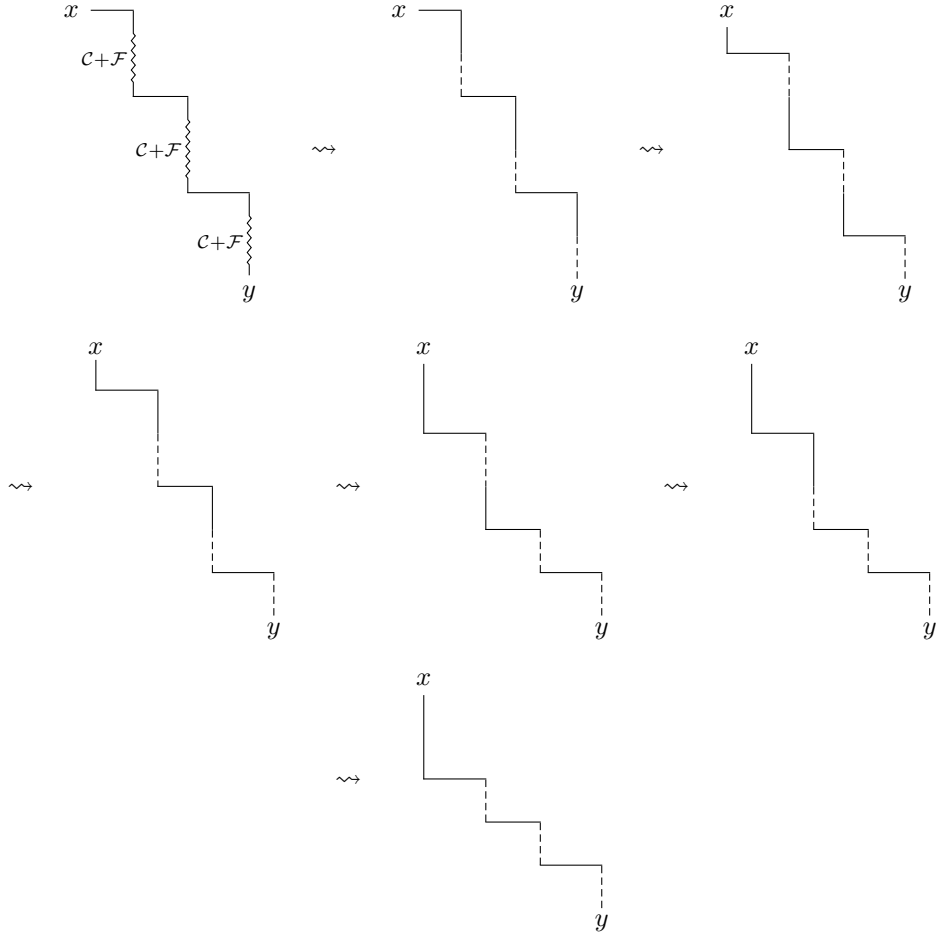
\[\begin{tikzcd}[row sep = small, column sep = small]
        & x & {} &&&& x & {} &&&& x \\
        &&&&&&& {} &&&& {} & {} \\
        && {} & {} &&&& {} & {} &&&& {} \\
        &&&&& \rightsquigarrow &&& {} && \rightsquigarrow && {} & {} \\
        &&& {} & {} &&&& {} & {} &&&& {} \\
        &&&&&&&&& {} &&&& {} & {} \\
        &&&& y &&&&& y &&&&& y 
        \arrow[shorten >=-4.7pt, no head, from=1-2, to=1-3]
        \arrow["{\C + \F}"', shorten <=-6.1pt, shorten >=-1.5pt, squiggly, no head, from=1-3, to=3-3]
        \arrow[shorten >=-4.7pt, no head, from=1-7, to=1-8]
        \arrow[shorten <=-6.1pt, shorten >=-1.5pt, no head, from=1-8, to=2-8]
        \arrow[shorten >=-1.5pt, no head, from=1-12, to=2-12]
        \arrow[shorten <=-6.1pt, shorten >=-1.5pt, dashed, no head, from=2-8, to=3-8]
        \arrow[shorten <=-4.7pt, shorten >=-4.7pt, no head, from=2-12, to=2-13]
        \arrow[shorten <=-6.1pt, shorten >=-1.5pt, dashed, no head, from=2-13, to=3-13]
        \arrow[shorten <=-4.7pt, shorten >=-4.7pt, no head, from=3-3, to=3-4]
        \arrow["{\C + \F}"', shorten <=-6.1pt, shorten >=-1.5pt, squiggly, no head, from=3-4, to=5-4]
        \arrow[shorten <=-4.7pt, shorten >=-4.7pt, no head, from=3-8, to=3-9]
        \arrow[shorten <=-6.1pt, shorten >=-1.5pt, no head, from=3-9, to=4-9]
        \arrow[shorten <=-6.1pt, shorten >=-1.5pt, no head, from=3-13, to=4-13]
        \arrow[shorten <=-6.1pt, shorten >=-1.5pt, dashed, no head, from=4-9, to=5-9]
        \arrow[shorten <=-4.7pt, shorten >=-4.7pt, no head, from=4-13, to=4-14]
        \arrow[shorten <=-6.1pt, shorten >=-1.5pt, dashed, no head, from=4-14, to=5-14]
        \arrow[shorten <=-4.7pt, shorten >=-4.7pt, no head, from=5-4, to=5-5]
        \arrow["{\C + \F}"', shorten <=-6.1pt, squiggly, no head, from=5-5, to=7-5]
        \arrow[shorten <=-4.7pt, shorten >=-4.7pt, no head, from=5-9, to=5-10]
        \arrow[shorten <=-6.1pt, shorten >=-1.5pt, no head, from=5-10, to=6-10]
        \arrow[shorten <=-6.1pt, shorten >=-1.5pt, no head, from=5-14, to=6-14]
        \arrow[shorten <=-6.1pt, shorten >=-1.5pt, dashed, no head, from=6-10, to=7-10]
        \arrow[shorten <=-4.7pt, shorten >=-4.7pt, no head, from=6-14, to=6-15]
        \arrow[shorten <=-6.1pt, shorten >=-1.5pt, dashed, no head, from=6-15, to=7-15]
    \end{tikzcd}\]
    
    \[\begin{tikzcd}[row sep = small, column sep = small]
        & x &&&&& x &&&&& x \\
        & {} & {} \\
        && {} &&&& {} & {} &&&& {} & {} \\
        \rightsquigarrow && {} & {} && \rightsquigarrow && {} &&& \rightsquigarrow && {} \\
        &&& {} &&&& {} & {} &&&& {} & {} \\
        &&& {} & {} &&&& {} & {} &&&& {} & {} \\
        &&&& y &&&&& y &&&&& y 
        \arrow[shorten <=-1.5pt, shorten >=-1.5pt, no head, from=1-2, to=2-2]
        \arrow[shorten >=-1.5pt, no head, from=1-7, to=3-7]
        \arrow[shorten >=-1.5pt, no head, from=1-12, to=3-12]
        \arrow[shorten <=-4.7pt, shorten >=-4.7pt, no head, from=2-2, to=2-3]
        \arrow[shorten <=-6.1pt, shorten >=-1.5pt, no head, from=2-3, to=3-3]
        \arrow[shorten <=-6.1pt, shorten >=-1.5pt, dashed, no head, from=3-3, to=4-3]
        \arrow[shorten <=-4.7pt, shorten >=-4.7pt, no head, from=3-7, to=3-8]
        \arrow[shorten <=-6.1pt, shorten >=-1.5pt, dashed, no head, from=3-8, to=4-8]
        \arrow[shorten <=-4.7pt, shorten >=-4.7pt, no head, from=3-12, to=3-13]
        \arrow[shorten <=-6.1pt, shorten >=-1.5pt, no head, from=3-13, to=4-13]
        \arrow[shorten <=-4.7pt, shorten >=-4.7pt, no head, from=4-3, to=4-4]
        \arrow[shorten <=-6.1pt, shorten >=-1.5pt, no head, from=4-4, to=5-4]
        \arrow[shorten <=-6.1pt, shorten >=-1.5pt, no head, from=4-8, to=5-8]
        \arrow[shorten <=-6.1pt, shorten >=-1.5pt, dashed, no head, from=4-13, to=5-13]
        \arrow[shorten <=-6.1pt, shorten >=-1.5pt, dashed, no head, from=5-4, to=6-4]
        \arrow[shorten <=-4.7pt, shorten >=-4.7pt, no head, from=5-8, to=5-9]
        \arrow[shorten <=-6.1pt, shorten >=-1.5pt, dashed, no head, from=5-9, to=6-9]
        \arrow[shorten <=-4.7pt, shorten >=-4.7pt, no head, from=5-13, to=5-14]
        \arrow[shorten <=-6.1pt, shorten >=-1.5pt, dashed, no head, from=5-14, to=6-14]
        \arrow[shorten <=-4.7pt, shorten >=-4.7pt, no head, from=6-4, to=6-5]
        \arrow[shorten <=-6.1pt, shorten >=-1.5pt, dashed, no head, from=6-5, to=7-5]
        \arrow[shorten <=-4.7pt, shorten >=-4.7pt, no head, from=6-9, to=6-10]
        \arrow[shorten <=-6.1pt, shorten >=-1.5pt, dashed, no head, from=6-10, to=7-10]
        \arrow[shorten <=-4.7pt, shorten >=-4.7pt, no head, from=6-14, to=6-15]
        \arrow[shorten <=-6.1pt, shorten >=-1.5pt, dashed, no head, from=6-15, to=7-15]
    \end{tikzcd}\]
    
    \[\begin{tikzcd}[row sep = small, column sep = small]
        & x\\
        \\
        \\
        \rightsquigarrow & {} & {} \\
        && {} & {} \\
        &&& {} & {} \\
        &&&& y
        \arrow[shorten >=-1.5pt, no head, from=1-2, to=4-2]
        \arrow[shorten <=-4.7pt, shorten >=-4.7pt, no head, from=4-2, to=4-3]
        \arrow[shorten <=-6.1pt, shorten >=-1.5pt, dashed, no head, from=4-3, to=5-3]
        \arrow[shorten <=-4.7pt, shorten >=-4.7pt, no head, from=5-3, to=5-4]
        \arrow[shorten <=-6.1pt, shorten >=-1.5pt, dashed, no head, from=5-4, to=6-4]
        \arrow[shorten <=-4.7pt, shorten >=-4.7pt, no head, from=6-4, to=6-5]
        \arrow[shorten <=-6.1pt, shorten >=-1.5pt, dashed, no head, from=6-5, to=7-5]
    \end{tikzcd}\]

\end{proof}

\begin{prop} \label{prop: conormal preserving}
    Let $\C$ be a conormal m-filter over $M$ and $\F$ be a m-filter.
    If $\F$ is localizable over $M$, then so is $\C + \F$. 
\end{prop}

\begin{proof}
    By \Cref{lem: conormal + mf local preorder}, we could take $n_y^{\C+\F} = n_y^\F + 1$. 
\end{proof}

\begin{cor} \label{cor: 1-step preserving}
    Let $\B$ be a binormal m-filter over $M$. If $\F$ is 1-step over $M$, then so is $\B+\F$. 
\end{cor}

\begin{proof}
    For $x, y \in M$, if $x \preceq_{\B + \F} y$, then we have diagram implication
    \[\begin{tikzcd}
	x &&& x & \sum \\
	sx & \sum & \rightsquigarrow && \sum \\
	& y &&& y
	\arrow[no head, "\B", from=1-1, to=2-1]
	\arrow[no head, from=1-4, to=1-5]
	\arrow[no head, "\B", from=1-5, to=2-5]
	\arrow[no head, from=2-1, to=2-2]
	\arrow[no head, "\F", from=2-2, to=3-2]
	\arrow[no head, "\F", from=2-5, to=3-5]
    \end{tikzcd}\]
    The first diagram comes from \Cref{lem: conormal + mf local preorder}, $\B$ is conormal and $\F$ is 1-step, while the second comes from the fact that $\B$ is normal. 
\end{proof}

\begin{prop}
    Let $\C$ be a conormal m-filter over $M$ and $\F$ be a localizable (resp., 1-step) m-filter over $M$. Then
    \begin{enum}
        \item $\C$ is a conormal m-filter over $M_\F$.
        \item $\F$ is a localizable (resp., 1-step) m-filter over $M_\C$. 
    \end{enum}
\end{prop}

\begin{proof}
    Part (a): Let $x, y \in M$, such that $\=x \preceq_\C \=y$ in $M_\F$. Then we have $\=x \le \sum \=x_i$, $\sum s_i\=x_i \le y$, thus $x \preceq_\F \sum x_i \preceq_\C \sum s_ix_i \preceq_\F y$, thus by \Cref{lem: conormal + mf local preorder}, we have $x \preceq_\C sx \preceq_\F y$ for some $s \in \C$, so $s\=x \le \=y$. Thus $\N$ is conormal over $M_\F$. 

    \Skip Part (b): if $\F$ is localizable over $M$ then $\F$ is localizable over $M_\C$ by \Cref{prop: MFG = MGF} and \Cref{prop: conormal preserving}.
    Now assume $\F$ is 1-step over $M$ and we want to show $\F$ is 1-step over $M_\C$.
    Pick $x, y \in M$ such that $\=x \preceq_\F \=y$ in $M_\C$.
    Then \Cref{lem: repr in MF is G local <= implies F+G local <=} shows $x \preceq_{\F + \C} y$, so \Cref{lem: conormal + mf local preorder} exhibits $z \in M$ such that $x \preceq_\C^1 z \preceq_\F y$.
    Then because $\F$ is 1-step over $M$, we see $z \preceq_\F^1 y$, so $\=z \preceq_\F^1 \=y$ in $M_\C$ and $\=x \le \=z$ in $M_\C$, hence $\=x \preceq_\F^1 \=y$ in $M_\C$, as desired.
\end{proof}

An important example of conormal filter is the dense filter in reduced quantale.

\begin{definition}
    Let $Q$ be a quantale with bottom $0$.
    Then we say $Q$ is \emph{reduced} if $a^2 = 0$ implies $a = 0$.
    Equivalently, $a^n = 0$ implies $a = 0$ for all $a \in Q, n \in \NN$.
\end{definition}

\begin{example}
    An idempotent quantale is reduced.
\end{example}

\begin{example}
    Let $R$ be a ring, then $\Id(R)$ is reduced if and only if $R$ is reduced.
\end{example}

\begin{theorem} \label{thm: dense filter is conormal}
    Let $Q$ be a reduced quantale with bottom $0$.
    Then the m-filter $\F_{\nmid 0}$ is conormal over $Q$.
    Moreover, $\=x = \=y$ in $Q_{\nmid 0}$ if and only if $\hbar(x) = \hbar(y)$, where $\hbar : Q \to Q, q \mapsto  \sum \br{q' \in Q : q q' = 0}$.
\end{theorem}

\begin{proof}
    First if $x \le \sum x_i$ and $s_i x_i \le y$ for some $s_i \in \F_{\nmid 0}$, then we claim $\hbar(y) \le \hbar(x)$.
    Pick $q \in Q$ such that $q y = 0$.
    Then $s_i q x_i = 0$, so $q x_i = 0$ from definition of $\F_{\nmid 0}$, so $q \sum x_i = 0$, so $q x = 0$, which shows $h(y) \le h(x)$.
    Thus, we see if $\=x = \=y$ in $Q_{\nmid 0}$, then $\hbar(x) = \hbar(y)$.

    Now pick $x, y \in Q$ such that $\hbar(x) = \hbar(y) = q$.
    Then we claim $x + q \in \F_{\nmid 0}$.
    If $p (x + q) = 0$, then $p x = 0$, so $p \le q$, so $p q = 0$ shows $p^2 = 0$, hence $p = 0$ since $Q$ is reduced.
    Thus, $x + q \in \F_{\nmid 0}$ and similarly $y + q \in \F_{\nmid 0}$.

    Now $xy + q \ge (x + q)(y + q) \in \F_{\nmid 0}$, and $x(x y + q) = x^2 y \le y$, which shows $\F_{\nmid 0}$ is conormal over $Q$ and $\=x = \=y$ if and only if $\hbar(x) = \hbar(y)$, as desired.
\end{proof}

\begin{cor} \label{cor: localize at dense is nontrivial}
    Let $Q \ne \br{\ast}$ be a reduced quantale with bottom $0$.
    Then $Q_{\nmid 0} \ne \br{\ast}$.
\end{cor}

\begin{proof}
    Assume the contrary, then $\=1 = \=0$, so by \Cref{thm: dense filter is conormal} $\hbar(0) = 1 = 0 = \hbar(1)$, a contradiction.
\end{proof}

\bigskip

\section{Comparisons, Applications, and Conjectures}
\label{sec: applications}

\subsection{Comparison with Existing Literatures}

The idea of localization of quantales has appeared in several literature.
In \cite{D76, CE95, BA19}, for precoherent quantale $Q$ (in \cite{D76} $Q$ is assumed to be an r-lattice, which is stronger than precoherent) and multiplicatively closed subset $S \sub K(Q)$, the authors define a map $Q \to Q, x \mapsto x_S$, where $x_S = \sum \br{c \in K(Q) : s c \le x \text{ for some } s \in S} = \sum \br{y \in Q : y \preceq_{\Ff(S)}^1 x}$.
Since $\Ff(S) = \sum_{s \in S} \F_s$ and each $\F_s$ is solid (by the assumption $S \sub K(Q)$ and \Cref{prop: compact element control is solid}), we see $\Ff(S)$ is solid by \Cref{prop: solid is quantale}, hence normal over $Q$ by \Cref{thm: normal in blooming}, so is 1-step over $Q$ by \Cref{prop: normal implies 1-step}.
Thus, $x_S = \sum \br{y \in Q : y \preceq_{\Ff(S)} x}$ and we see this is precisely the localization of $Q$ at $\Ff(S)$.
Similarly, in \cite{G25}, for coherent quantale $Q$ and prime element $p \in Q$, Georgescu define a map $Q \to Q, x \mapsto x_p$, where $x_p = \sum \br{y \in Q : y \preceq_{\F_{\nmid p}} x}$, so we see this is precisely the localization of $Q$ at $\F_{\nmid p}$ (which is again solid).

In summary, when defining localization people are mainly concerned with precoherent quantales (and localize at solid m-filters), which is a strong condition and can provide nice properties.
Also, their main focus is quantale rather than their modules.
Therefore, in this paper we generalize the idea of localization to any quantale and their modules.
Since most precoherent quantales arise from algebraic objects, our generalization allows us to apply the theory of localization to geometric objects like topological spaces, and produce some applications as we will immediately see.

\smallskip

\subsection{Applications}

\Cref{thm: locally compact hausdorff is blooming} tells us that $\O(X)$ is blooming provided that $X$ is a locally compact Hausdorff space.
This leads to some clean results.

\begin{prop} \label{prop: mf in Xda containing subset is 1-step}
    Let $X$ be a locally compact Hausdorff space and $Y \sub X$ be a subset.
    Consider m-filter $\F = \br{U \in \O(X) : Y \sub U}$.
    Then $\F$ is 1-step over $\O(X)$.
    Moreover, for $U, V \in \O(X)$, we have $\=U = \=V$ in $\O(X)_\F$ if and only if $U \cap Y = V \cap Y$.
\end{prop}

\begin{proof}
    We have $\F = \bigcap_{y \in Y} \F_{\nmid (X - \br{y})}$.
    Since each $\F_{\nmid (X - \br{y})}$ is solid (hence 1-step over $\O(X)$ by \Cref{thm: normal in blooming} and \Cref{prop: normal implies 1-step}), we see $\F$ is 1-step over $\O(X)$ by \Cref{prop: arbi inter 1-step is 1-step}.
    Now take $U_1, U_2 \in \O(X)$ such that $U_1 \preceq_\F^1 U_2$, then we claim that $U_1 \cap Y \sub U_2 \cap Y$.
    By assumption we see for each $x \in U_1 \cap Y$, there exists $U_x \in \O(X), V_x \in \F$ such that $U_x \le U_1$ and $U_x \cap V_x \sub U_2$.
    As $x \in U_x \cap V_x$, we see $x \in U_2 \cap Y$, showing $U_1 \cap Y \sub U_2 \cap Y$, as desired.

    Conversely, take $U_1, U_2 \in \O(X)$ such that $U_1 \cap Y = U_2 \cap Y$.
    Observe that for $y \in Y$, we have $\={U_1} = \={U_2}$ in $X_{\nmid (X - \br{y})}$ if and only if $U_1 \cap \br{y} = U_2 \cap \br{y}$, so we see the image of $\={U_1}$ and $\={U_2}$ under the map $\vp : \O(X)_\F \to \prod_{y \in Y} \O(X)_{\nmid (X - \br{y})}$ are the same.
    Since $\vp$ is injective by \Cref{thm: arbitrary merge}, we see $\={U_1} = \={U_2}$ in $\O(X)_\F$, as desired.
\end{proof}

\begin{remark}
    Thus, for locally compact Hausdorff space $X$ and $Y \sub X$, after letting\\ $\F = \br{U \in \O(X) : Y \sub U} \in \mF(\O(X))$ we get an idempotent quantale isomorphism $\O(X)_\F \iso \O(Y), \=U \mapsto U \cap Y$.
\end{remark}

\begin{cor} \label{cor: closed subspace of loc compact is loc compact}
    Let $X$ be a locally compact Hausdorff space.
    Then a closed subspace in $X$ is locally compact Hausdorff.
\end{cor}

\begin{proof}
    Let $Y \sub X$ be a closed subspace and let $\F = \br{U \in \O(X) : Y \sub U} \in \mF(\O(X))$.
    Then $\O(Y) \cong \O(X)_\F$ is blooming by \Cref{prop: blooming preserved under normal}.
    As a subspace of a regular space $X$, we have $Y$ is regular, so \Cref{thm: locally compact hausdorff is blooming} tells us $Y$ is locally compact.
\end{proof}

Reinterpret \Cref{prop: mf in Xda containing subset is 1-step} using the language of topology we obtain

\begin{cor} \label{cor: top interpret mf containing subset is 1-step}
    Let $X$ be a locally compact Hausdorff space and $Y \sub X$ be a subset.
    Pick open subsets $U, V \sub X$ such that $U \cap Y \sub V \cap Y$.
    Then for each $x \in U$, there exist open subsets $W_1 \ni x$ and $W_2 \supset Y$ such that $W_1 \cap W_2 \sub V$.
\end{cor}

Now let us see how results from commutative algebra (and algebraic geometry) can be proved using the language of quantale.

\begin{theorem} \label{thm: M -> prod Mfi is injective}
    Let $A$ be a ring and $f_1, \ldots, f_n \in A$ such that $(f_1, \ldots, f_n) = A$.
    Then for an $A$-module $M$, the $A$-linear map $M \to \prod_{i = 1}^n M_{f_i}$ is injective.
\end{theorem}

\begin{proof}
    Assume $m \in M$ is sent to 0, then $A m$ is sent to 0 under map $\Sub_A(M) \to \prod_{i = 1}^n \Sub_A(M_{f_i}) = \prod_{i = 1}^n (\Sub_A(M))_{(f_i)}$.
    However, we have $\prod_{i = 1}^n \F_{(f_i)} = \F_{\sum_{i = 1}^n (f_i)} = \br{1}$ by \Cref{prop: Ff cap Fg and Ff Fg}, so by \Cref{thm: gluing axioms for localization} we see $A m = 0$ in $\Sub_A(M)$, which shows $m = 0$, as desired.
\end{proof}

\begin{prop} \label{prop: compact blooming maximal merge}
    Let $Q$ be a compact blooming quantale and $M$ be a blooming $Q$-module.
    Then the $Q$-linear map
    \[
    M \to \prod_\mm M_{\nmid \mm}
    \]
    is injective, where the product runs over all maximal elements $\mm \in Q$.
\end{prop}

\begin{proof}
    By \Cref{thm: arbitrary merge} it suffices to show $\bigcap_\mm \F_{\nmid \mm} = \br{1}$ (since each $\F_{\nmid \mm}$ is solid, hence normal over $M$ by \Cref{thm: normal in blooming}).
    Pick any $x \in \bigcap_\mm \F_{\nmid \mm}$.
    If $x < 1$, then by (the proof of) \ref{thm: prime contains and avoid} there exists maximal $\mm_0 \in Q$ such that $x \le \mm_0$, so $x \not\in \F_{\nmid \mm_0}$, a contradiction.
\end{proof}

\begin{cor} \label{cor: M -> prod Mm is injective}
    Let $A$ be a ring and $M$ be an $A$-module.
    Then the $A$-linear map $M \to \prod_\mm M_\mm$ is injective, where $\mm$ runs over all maximal ideals in $A$.
\end{cor}

\begin{proof}
    Assume $m \in M$ is sent to 0, then $A m$ is sent to 0 under map $\Sub_A(M) \to \prod_\mm \Sub_A(M_\mm) = \prod_\mm (\Sub_A(M))_{\nmid \mm}$, so by \Cref{prop: compact blooming maximal merge} we see $A m = 0$ in $\Sub_A(M)$, hence $m = 0$, as desired.
\end{proof}

\begin{example}
    Let us see how \Cref{prop: compact blooming maximal merge} is applied to compact Hausdorff space $X$.
    In this case the maximal elements in $\O(X)$ corresponds to points (more precisely, complements of point), and \Cref{prop: compact blooming maximal merge} says two open sets in $\O(X)$ are the same if they contain the some set of points in $X$ (which is trivial).
\end{example}

Now let us see how (a version of) Baire Category Theorem can be proved using the language of quantale.

\begin{theorem} [Baire Category] \label{thm: Baire Category}
    Let $X$ be a locally compact regular space.
    Suppose $\br{C_i}_{i = 1}^\oo$ are nowhere dense closed subsets in $X$.
    Then $X \ne \bigcup_{i = 1}^\oo C_i$.
\end{theorem}

\begin{proof}
    Let $U_i = X - C_i \in Q = \O(X)$.
    Assume the contrary, $X = \bigcup_{i = 1}^\oo C_i$, then  $\bigcap_{i = 1}^\oo \F_{\perp U_i} = \br{1}$ is locally solid by \Cref{thm: locally compact hausdorff is blooming}.
    Then the $Q$-linear map $Q_{\nmid 0} \to \prod_{i = 1}^\oo (Q_{\nmid 0})_{\perp U_i}$ is injective by \Cref{thm: dense filter is conormal}, \Cref{prop: normal inheritance}, and \Cref{thm: countable solid inj}.
    For each $i \in \NN$, we have $X \preceq_{\F_{\nmid 0}} U_i \preceq_{\F_{\perp U_i}} \varnothing$, showing $(Q_{\nmid 0})_{\perp U_i} = \br{\ast}$, so $Q_{\nmid 0} = \br{\ast}$, which contradicts \Cref{cor: localize at dense is nontrivial}.
\end{proof}

We recall the non-rigorous ``proof" of Baire Category Theorem and show where each point corresponds in the above proof:
\begin{enum}
    \item We regard everything up to a nowhere dense subset: consider $\O(X)_{\nmid \varnothing}$.
    \item Then each nowhere dense subset becomes empty set: $(\O(X)_{\nmid \varnothing})_{\perp U_i} = \br{\ast}$.
    \item However, the whole space is not the empty set since the whole space is not nowhere dense: $\O(X)_{\nmid \varnothing} \ne \br{\ast}$.
    \item We conclude by the observation that a nonempty set is not the union of countably many empty sets: $\O(X)_{\nmid \varnothing} \to \prod_{i = 1}^\oo (\O(X)_{\nmid \varnothing})_{\perp U_i}$ is injective.
\end{enum}
We also address the concerns:
\begin{enum}
    \item Why can we regard everything up to a nowhere dense subset?
        Because $\F_{\nmid \varnothing}$ is conormal (in particular, localizable) over $\O(X)$ by \Cref{thm: dense filter is conormal}.
    \item A nonempty set is not the union of arbitrarily many empty sets, so why Baire Category Theorem does not work for arbitrarily many nowhere dense subsets?
        Because \Cref{thm: countable solid inj} only works for countably many m-filters.
    \item Where did we use the locally compact Hausdorff condition?
        Because we want $\br{1} \in \mF(\O(X))$ to be locally solid (so that the condition in \Cref{thm: countable solid inj} is met), so that $X$ should be locally compact.
        Moreover, we want each $(\O(X)_{\nmid \varnothing})_{\perp U_i}$ to be trivial, so we want $X$ to be regular (and locally compact Hausdorff space is locally compact regular).
\end{enum}

We now present two generalizations of Baire Category Theorem.

\begin{theorem} \label{thm: Baire Category regular subpsace}
    Let $X$ be a locally compact space (i.e., every point has a compact neighborhood) and $Y \sub X$ a regular subspace.
    Then for arbitrary subspaces $\br{Z_k}_{k = 1}^\oo$ in $X$, if $\bigcup_{k = 1}^\oo Z_k = X$, then there exists $\ell \in \NN$ such that $Z_\ell \cap Y$ is not nowhere dense in $Y$.
\end{theorem}

\begin{proof}
    The embedding $Y \hookrightarrow X$ induces a quantale homomorphism $\O(X) \to \O(Y), U \mapsto U \cap Y$.
    We have $\O(Y)_{\nmid \varnothing}$ is a $\O(Y)$-module, so can be pulled back to a $\O(X)$-module with action $U \cdot \=V = \={U \cap V \cap Y}$ for $U \in \O(X), V \in \O(Y)$.

    Now for $k \in \NN$, let $\F_k = \br{U \in \O(X) : Z_k \sub U} \in \mF(\O(X))$, then $\bigcap_{k = 1}^\oo \F_k = \br{1}$ is locally solid.
    Assume the contrary, then inside $\O(Y)$, we have $Y \preceq_{\F_{\nmid \varnothing}}^1 Y - \cl_Y(Z_k \cap Y) \preceq_{\F_k}^1 \varnothing$ since $Y$ is regular, so $Y \preceq_{\F_k}^1 \varnothing$ in $\O(Y)_{\nmid \varnothing}$ for all $k \in \NN$.
    Thus, inside $\O(Y)_{\nmid \varnothing}$ we have $\=Y \preceq_{\F_k} \=\varnothing$ for each $k \in \NN$, hence $\=Y \preceq_{\br{1}}^1 \=\varnothing$ by \Cref{lem: countable filter merge}, which contradicts \Cref{cor: localize at dense is nontrivial}.
\end{proof}

The proof of next theorem follows from very similar routine.

\begin{theorem} \label{thm: Baire Category map}
    Let $X$ be a locally compact Hausdorff space with closed subspaces $\br{C_k}_{k = 1}^\oo$ such that $X = \bigcup_{k = 1}^\oo C_k$.
    Let $f : Y \to X$ be a continuous map, then there exists $\ell \in \NN$ such that $f\inv(C_\ell) \sub Y$ is not nowhere dense.
\end{theorem}

\begin{proof}
    $f$ induces a quantale homomorphism $\O(f) : \O(X) \to \O(Y), U \mapsto f\inv(U)$, so we can view $\O(Y)_{\nmid \varnothing}$ as a $\O(X)$-module.
    For $k \in \NN$, let $\F_k = \br{U \in \O(X) : C_k \sub U} \in \mF(\O(X))$, then $\bigcap_{k = 1}^\oo \F_k = \br{1}$ is locally solid.
    
    Assume the contrary, then inside $\O(Y)$ we have $Y \preceq_{\F_{\nmid \varnothing}}^1 Y - f\inv(C_k)$, and we claim $Y - f\inv(C_k) \preceq_{\F_k}^1 \varnothing$.
    It suffices to show $\bigcap_{U \supset C_k} \cl f\inv(U) = f\inv(C_k)$, where in the intersection $U \sub X$ is open.
    Obviously $f\inv(C_k) \sub \bigcap_{U \supset C_k} \cl f\inv(U)$.
    For the other direction, just note that $\cl f\inv(U) \sub f\inv(\cl U)$ and $\bigcap_{U \supset C_k} \cl U = C_k$ since $X$ is regular.
    Thus, $\bigcap_{U \supset C_k} \cl f\inv(U) \sub \bigcap_{U \supset C_k} f\inv(\cl U) = f\inv(\bigcap_{U \supset C_k} \cl U) = f\inv(C_k)$, as desired.

    As a result, we see $\=Y \preceq_{\br{1}}^1 \={\varnothing}$ in $\O(Y)_{\nmid \varnothing}$ by \Cref{lem: countable filter merge}, which contradicts \Cref{cor: localize at dense is nontrivial}.
\end{proof}

Mimic \Cref{thm: Baire Category}, we can get an algebraic version of Baire Category Theorem.
Note that the proof is pretty much the same as \Cref{thm: Baire Category regular subpsace} and \Cref{thm: Baire Category map}.

\begin{theorem} \label{thm: algebraic Baire Category}
    Let $R$ be a ring with radical ideal $\bb \varsubsetneq R$ and countably many ideals $\br{I_k}_{k = 1}^\oo$.
    Suppose for all $r \in R - \bb$ and $k \in \NN$, there exists $x \in R$ such that $r x \in R - \bb$ and $(1 - a)x \in \bb$ for some $a \in I_k$, then there exists maximal ideal $\mm \varsubsetneq R$ such that $I_k \not\sub \mm$ for all $k \in \NN$.
\end{theorem}

\begin{proof}
    The ring homomorphism $R \to R / \bb$ induces quantale homomorphism $\Id(R) \to \Id(R / \bb), J \mapsto J (R / \bb)$, so we can view $\Id(R / \bb)_{\nmid (0)}$ (note that this is a $\Id(R / \bb)$-module by \Cref{thm: dense filter is conormal}) as a $\Id(R)$-module.
    
    Assume the contrary, then $\bigcap_{k = 1}^\oo \F_{\perp I_k} = \br{1} \in \mF(\Id(R))$ is locally solid (it is even solid since $\Id(R)$ is compact).
    Now for $k \in \NN$ let $J_k = (x \in R : (1 - a)x \in \bb \text{ for some } a \in I_k) (R / \bb) \in \Id(R / \bb)$.
    Consider $x \in R$ along with $a \in I_k$ such that $(1 - a)x \in \bb$, then we have $(1 - a) + I_k = R$ and $(1 - a)(x)(R / \bb) = (0)(R / \bb)$, so $J_k \preceq_{\F_{\perp I_k}}^1 (0)$.

    Thus, by our assumption $1 \preceq_{\F_{\nmid (0)}}^1 J_k \preceq_{\F_{\perp I_k}}^1 (0)$, so $\={R / \bb} \preceq_{\F_{\perp I_k}}^1 \={(0)}$ in $\Id(R / \bb)_{\nmid (0)}$ for all $k \in \NN$.
    As a result, $\={R / \bb} \preceq_{\br{1}}^1 \={(0)}$ in $\Id(R / \bb)_{\nmid (0)}$ by \Cref{lem: countable filter merge}, which contradicts \Cref{cor: localize at dense is nontrivial}.
\end{proof}

\begin{cor} \label{cor: algebraic point Baire Category}
    Let $R$ be a ring with countably many maximal ideals.
    Pick radical ideal $\bb \varsubsetneq R$.
    Then there exist maximal ideal $\mm \sub R$ and $r \in R - \bb$ satisfying $r x \in \bb$ for all $x \in R$ such that $(1 - a)x \in \bb$ for some $a \in \mm$.
\end{cor}

\begin{proof}
    This is a direct consequence of \Cref{thm: algebraic Baire Category} after letting $\br{I_k}_{k = 1}^\oo$ enumerate all maximal ideals.
\end{proof}

\smallskip

\subsection{Conjectures and Questions}

Note that in the proof of \Cref{thm: Baire Category regular subpsace}, we only used the condition $\br{1} \sub \mF(\O(X))$ is locally solid (instead of $X$ is locally compact).
A locally compact topological space induces quantale that has $\br{1}$ being locally solid, and we wonder if the converse holds.

\begin{conj}
    Let $X$ be a topological space.
    Then $X$ is locally compact if and only if $\br{1} \in \mF(\O(X))$ is locally solid.
\end{conj}

The condition that $Q$ is reduced in \Cref{thm: dense filter is conormal} is kind of annoying, and we wonder if we can remove that.

\begin{conj}
    Let $Q$ be a quantale with bottom $0$.
    Then $\F_{\nmid 0}$ is localizable over $Q$.
\end{conj}

Note that if $\F_{\nmid 0}$ is localizable (rather than conormal), the arguments in \Cref{thm: Baire Category regular subpsace}, \Cref{thm: Baire Category map}, and \Cref{thm: algebraic Baire Category} can run.

Another natural question to ask is

\begin{prob}
    For which $q \in Q$ is $\F_{\nmid q}$ conormal over $Q$?
\end{prob}

\bigskip

\appendix

\section{Filter Merging Theorems}
\label{sec: app}

Here we list the three filter merging theorems, which apply to different scenarios.

\begin{theorem} \label{thm: app, finite merge}
    Let $Q$ be a quantale with $\F_1, \ldots, \F_n \sub \mF(Q)$ and $M$ be a shrinkable $Q$-module.
    Let $\F = \prod_{k = 1}^n \F_k \in \mF(Q)$ and consider map (between $Q$-premodules) $\vp : M_\F \to \prod_{k = 1}^n M_{\F_k}$.
    Then
    \begin{enum}
        \item $\vp$ is injective.
        \item If $\F_i + \F_j$ is 1-step over $M$ for all $1 \le i < j \le n$, then for $x_1, \ldots, x_n \in M$, we have $(\={x_1}, \ldots, \={x_n}) \in \im \vp$ if and only if the image of $\={x_i}$ under map $M_{\F_i} \to M_{\F_i + \F_j}$ and the image of $\={x_j}$ under map $M_{\F_j} \to M_{\F_i + \F_j}$ agree for all $1 \le i < j \le n$.
    \end{enum}
\end{theorem}

\begin{theorem} \label{thm: app, countable merge}
    Let $Q$ be a quantale with m-filters $\br{\F_k}_{k = 1}^\oo$ such that $\F = \bigcap_{k = 1}^\oo \F_k$ is locally solid.
    Let $M$ be a shrinkable $Q$-module such that all $\F_k$'s are 1-step relative to $M$.
    Then we have injective $Q$-linear map
    \[
    M_\F \to \prod_{k = 1}^\oo M_{\F_k}.
    \]
\end{theorem}

\begin{theorem} \label{thm: app, arbitrary merge}
    Let $Q$ be a quantale and $M$ a continuous $Q$-module.
    For $\F_i \in \mF(Q)$ and $\F = \bigcap_{i \in I} \F_i$, if each $\F_i$ is 1-step over $M$, then we have injective $Q$-linear map
    \[
    M_\F \to \prod_{i \in I} M_{\F_i}.
    \]
\end{theorem}

Here is a table comparing the conditions for the injectivity of $M_{\bigcap \F_i} \to \prod M_{\F_i}$ in each theorem.

\begin{center}
    \begin{tabular}{|c|c|c|c|c|c|}
        \hline
        Theorem & Number of m-filters & $\F_i$ & $Q$ & $M$ & Additional Comment \\
        \hline
        \ref{thm: gluing axioms for localization} & Finitely many & None & None & Shrinkable & Map between $Q$-premodules \\
        \hline
        \ref{thm: countable solid inj} & Countably many & 1-step & None & Shrinkable & Assume $\bigcap \F_i$ is locally solid \\
        \hline
        \ref{thm: arbitrary merge} & Arbitrarily many & 1-step & None & Continuous & $\bigcap \F_i$ is automatically 1-step \\
        \hline
    \end{tabular}
    \label{tab: app, compare conditions}
\end{center}

%
%

\bigskip

\end{document}